\documentclass[a4paper,12pt]{amsart}

\usepackage[foot]{amsaddr}

\usepackage{a4wide}
\usepackage{subcaption}

\usepackage{tikz}
\usepackage{tikz-3dplot}
\usetikzlibrary[patterns]

\usepackage[breaklinks,bookmarks=false]{hyperref}
\hypersetup{colorlinks, linkcolor=blue, citecolor=blue,
urlcolor=blue, plainpages=false, pdfwindowui=false,
pdfstartview={FitH}}

\usepackage{todonotes}
\usepackage{bm}
\usepackage[normalem]{ulem} 
\usepackage{cancel} 
\usepackage{enumitem}

\usepackage{amssymb}
\usepackage{mathrsfs}

\newcommand{\eq}{:=}

\newcommand{\grad}{\boldsymbol \nabla}
\renewcommand{\div}{\grad {\cdot}}
\newcommand{\curl}{\grad {\times}}

\newcommand{\ccurl}{{\operatorname{\bf {curl}}}}
\newcommand{\ddiv}{\operatorname{div}}

\newcommand{\BH}{\boldsymbol H}

\newcommand{\BL}{\boldsymbol L}

\newcommand{\BR}{\boldsymbol R}

\newcommand{\BV}{\boldsymbol V}

\newcommand{\ba}{\boldsymbol a}
\newcommand{\bb}{\boldsymbol b}
\newcommand{\bc}{\boldsymbol c}
\newcommand{\bd}{\boldsymbol d}
\newcommand{\be}{\boldsymbol e}
\newcommand{\bg}{\boldsymbol g}
\newcommand{\bh}{\boldsymbol h}

\newcommand{\bj}{\boldsymbol j}

\newcommand{\bn}{\boldsymbol n}

\newcommand{\br}{\boldsymbol r}

\newcommand{\bv}{\boldsymbol v}
\newcommand{\bw}{\boldsymbol w}
\newcommand{\bx}{\boldsymbol x}

\newcommand{\CE}{\mathcal E}
\newcommand{\CF}{\mathcal F}

\newcommand{\CP}{\mathcal P}

\newcommand{\CT}{\mathcal T}

\newcommand{\CV}{\mathcal V}

\newcommand{\LE}{\mathscr E}

\newcommand{\LR}{\mathscr R}

\newcommand{\LT}{\mathscr T}

\newcommand{\LV}{\mathscr V}

\newcommand{\TE}{\textup E}
\newcommand{\TF}{\textup F}

\newcommand{\TM}{\textup M}

\newcommand{\TR}{\textup R}
\newcommand{\TS}{\textup S}

\newcommand{\BCP}{\boldsymbol{\CP}}

\newtheorem{theorem}{Theorem}
\newtheorem{proposition}[theorem]{Proposition}
\newtheorem{corollary}[theorem]{Corollary}
\newtheorem{lemma}[theorem]{Lemma}
\newtheorem{remark}[theorem]{Remark}
\newtheorem{definition}[theorem]{Definition}

\numberwithin{equation}{section}
\numberwithin{theorem}{section}

\newcommand{\sign}{\operatorname{sign}}

\newcommand{\bzero}{\boldsymbol 0}

\newcommand{\bchi}{\boldsymbol \chi}
\newcommand{\btau}{\boldsymbol \tau}

\newcommand{\om}{\omega}
\newcommand{\oma}{\omega_{\ba}}
\newcommand{\omo}{\omega_{\boldsymbol 0}}
\newcommand{\tomo}{\widetilde \omega_{\boldsymbol 0}}

\newcommand{\omb}{\widetilde \omega_{\bb}}
\newcommand{\ome}{\omega_e}

\newcommand{\Ga}{\Gamma_{\ba}}

\newcommand{\tGamma}{\widetilde \Gamma}

\newcommand{\Gaess}{\Gamma_{\ba}^{\rm ess}}
\newcommand{\Ganat}{\Gamma_{\ba}^{\rm nat}}

\newcommand{\wGaess}{\widehat \Gamma_{\ba}^{\rm ess}}
\newcommand{\wGanat}{\widehat \Gamma_{\ba}^{\rm nat}}

\newcommand{\Goess}{\Gamma_{\bzero}^{\rm ess}}
\newcommand{\Gonat}{\Gamma_{\bzero}^{\rm nat}}

\newcommand{\tGoess}{\widetilde \Gamma_{\bzero}^{\rm ess}}
\newcommand{\tGonat}{\widetilde \Gamma_{\bzero}^{\rm nat}}

\newcommand{\CTa}{\CT_{\ba}}
\newcommand{\CTb}{\widetilde \CT_{\bb}}
\newcommand{\CFb}{\widetilde \CF_{\bb}}
\newcommand{\tCTa}{\widetilde \CT_{\ba}}
\newcommand{\tCTo}{\widetilde \CT_{\boldsymbol 0}}

\newcommand{\CTo}{\CT_{\boldsymbol 0}}
\newcommand{\hCTo}{\widehat \CT_{\boldsymbol 0}}

\newcommand{\CVa}{\CV_{\ba}}

\newcommand{\toma}{\widetilde \omega_{\ba}}

\newcommand{\tG}{\widetilde \Gamma}

\newcommand{\tbchi}{\widetilde \bchi}

\newcommand{\CFa}{\CF_{\ba}}

\newcommand{\CTe}{\CT_e}
\newcommand{\CVe}{\CV_e}

\newcommand{\RT}{\bm{\mathcal{R\hspace{-0.1em}T}}\hspace{-0.25em}}
\newcommand{\ND}{\bm{\mathcal{N}}\hspace{-0.2em}}

\newcommand{\tbj}{\widetilde \bj}
\newcommand{\tbr}{\widetilde \br}
\newcommand{\tbv}{\widetilde \bv}

\newcommand{\bxi}{\boldsymbol \xi}
\newcommand{\bphi}{\boldsymbol \phi}

\newcommand{\scurl}{\operatorname{curl}}

\newcommand{\col}{{\tt col}}

\newcommand{\Kin }{K_{\rm in }}
\newcommand{\Kout}{K_{\rm out}}

\newcommand{\JAC}{\mathbb J}

\newcommand{\enorm}[1]{|\!|\!|#1|\!|\!|}

\newcommand\scp{{\cdot}} 
\newcommand\vp{{\times}} 
\newcommand\nv{\bm 0}
\newcommand{\bsig}{\bm{\sigma}}
\newcommand\pt{\partial}

\title[Stable discrete minimization in the de Rham complex]%
{Constrained and unconstrained stable discrete minimizations
for $\MakeLowercase{p}$-robust local reconstructions in~vertex
patches in the de Rham complex$^\star$}

\author{T. Chaumont-Frelet$^\dagger$ and M. Vohral\'ik$^\ddagger$}

\address{\vspace{-.5cm}}
\address{\noindent \tiny \textup{$^\dagger$Inria Univ. Lille and Laboratoire Paul Painlev\'e, 59655 Villeneuve-d'Ascq, France}}
\address{\noindent \tiny \textup{$^\ddagger$Inria, 2 rue Simone Iff, 75589 Paris, France \&
CERMICS, Ecole des Ponts, 77455 Marne-la-Vall\'ee, France}}
\address{\vspace{.5cm}}
\address{\noindent \tiny \textup{$^\star$This project has received funding from the European Research Council (ERC) under
the European Union’s Horizon 2020 research and innovation program (grant agreement No 647134
GATIPOR).}}

\begin{document}

\begin{abstract}
We analyze constrained and unconstrained minimization problems on patches
of tetrahedra sharing a common vertex with discontinuous piecewise polynomial
data of degree $p$. We show that the discrete minimizers in the spaces of piecewise
polynomials of degree $p$ conforming in the $H^1$, $\BH(\ccurl)$, or $\BH(\ddiv)$
spaces are as good as the minimizers in these entire (infinite-dimensional)
Sobolev spaces, up to a constant that is independent of $p$. These results are useful
in the analysis and design of finite element methods, namely for devising stable local
commuting projectors and establishing local-best--global-best equivalences in a priori
analysis and in the context of a posteriori error estimation. Unconstrained minimization
in $H^1$ and constrained minimization in $\BH(\ddiv)$ have been previously treated in the
literature. Along with improvement of the results in the $H^1$ and $\BH(\ddiv)$ cases,
our key contribution is the treatment of the $\BH(\ccurl)$ framework. This enables us
to cover the whole De Rham diagram in three space dimensions in a single setting.

\vspace{.5cm}

\vspace{.5cm}
\end{abstract}


\maketitle

\thispagestyle{empty}


\section{Introduction}

The concept of ``equilibrated flux'', dating back to at least the seminal
paper~\cite{prager_synge_1947a}, is the basis for the design of guaranteed
a posteriori error estimates for finite element discretization of various
PDE problems, see~\cite{Dest_Met_expl_err_CFE_99, ainsworth_oden_2000a, luce_wohlmuth_2004a, Repin_book_08, braess_schoberl_2008a, ern_vohralik_2015a, Lad_Cham_CRE_16, Ern_Sme_Voh_heat_HO_Y_17, Sme_Voh_RD_20, chaumontfrelet_ern_vohralik_2021b} and the references therein.
One key feature of this family of estimators is that they
can be designed so that they are ``polynomial-degree-robust'' (or simply, $p$-robust), meaning
that their overestimation factor does not depend on the polynomial degree $p$ of the discretization
space. This fact has first been established in~\cite{braess_pillwein_schoberl_2009a} when
considering a conforming finite element discretization of the two-dimensional Poisson problem.
The proof hinges on the following result: if $\CTa$ is a vertex patch of triangles sharing a
vertex $\ba$, $\oma$ is the corresponding domain, $p \geq 0$ is a polynomial degree, and
$r_p \in \CP_p(\CTa)$ as well as $\btau_p \in \RT_p(\CTa)$
are given (discontinuous) piecewise polynomial data (these notations are rigorously
introduced below), there holds
\begin{equation}
\label{eq_constrained_minimization}
\min_{\substack{
\bw_p \in \RT_p(\CTa) \cap \BH_0(\ddiv,\oma)
\\
\div \bw_p = r_p
}}
\|\bw_p - \btau_p\|_{\oma}
\leq
C_{\rm st}
\min_{\substack{
\bw \in \BH_0(\ddiv,\oma)
\\
\div \bw = r_p
}}
\|\bw - \btau_p\|_{\oma},
\end{equation}
where the constant $C_{\rm st}$ only depends on the shape-regularity parameter of
the patch; crucially, $C_{\rm st}$ is independent of $p$.
The proof of~\eqref{eq_constrained_minimization} hinges on the volume and normal trace
$p$-robust polynomial extensions on a single tetrahedron
of~\cite[Proposition~4.2]{Cost_McInt_Bog_Poinc_10} and~\cite[Theorem~7.1]{Demk_Gop_Sch_ext_III_12},
and the result also holds in three space dimensions, see~\cite[Corollary~3.3]{ern_vohralik_2020a}.
The (fully computable) minimizer on the left-hand side of~\eqref{eq_constrained_minimization}
is directly involved in the construction of the a posteriori error estimator, while the minimum
of the right-hand side is not computable but can be straightforwardly related to the
discretization error in the patch domain $\oma$. The constant $C_{\rm st}$ thus naturally enters
the efficiency estimate of the estimator, and the $p$-robustness is a consequence of the fact that
it is independent of $p$.

Constrained minimization problems of the form~\eqref{eq_constrained_minimization}
are sufficient for the a posteriori analysis of conforming finite element discretizations.
However, when considering nonconforming discretizations~\cite{ern_vohralik_2015a},
another family of minimization problems comes into play. Specifically, the following
stability result is of paramount importance: given a (discontinuous) piecewise
polynomial $\chi_p \in \CP_{p+1}(\CTa)$ vanishing on $\partial \oma$, there holds
\begin{equation}
\label{eq_unconstrained_minimization}
\min_{v_p \in \CP_{p+1}(\CTa) \cap H^1_0(\oma)} \|\grad (v_p - \chi_p)\|_{\oma}
\leq
C_{\rm st}
\min_{v \in H^1_0(\oma)} \|\grad (v - \chi_p)\|_{\oma},
\end{equation}
where $\grad$ denotes the broken (elementwise) gradient and $C_{\rm st}$ again does not depend
on $p$, see~\cite[Corollary~3.1]{ern_vohralik_2020a} which builds
on~\cite[Theorem~6.1]{Demk_Gop_Sch_ext_I_09}. Similarly to~\eqref{eq_constrained_minimization},
the minimizer of the left-hand side of~\eqref{eq_unconstrained_minimization} is computed as a part
of the estimator construction, while the right-hand side can be linked to the discretization error
in the patch domain $\oma$.

The $H^1$ and $\BH(\ddiv)$ spaces in~\eqref{eq_constrained_minimization}
and~\eqref{eq_unconstrained_minimization} are naturally involved in the context
of the Poisson problem, since the Laplace differential operator is a composition
of gradient and divergence. When considering Maxwell's equations and their discretization
by N\'ed\'elec's elements, minimization problems similar
to~\eqref{eq_constrained_minimization} and~\eqref{eq_unconstrained_minimization} but
involving the $\BH(\ccurl)$ Sobolev space and the curl operator naturally
emerge~\cite{chaumontfrelet_ern_vohralik_2020a,chaumontfrelet_ern_vohralik_2021a,Chaum_Voh_Maxwell_equil_21}.
In particular, an equivalent to~\eqref{eq_constrained_minimization}
on a smaller edge patch of tetrahedra has been recently established
in~\cite[Theorem~3.1]{chaumontfrelet_ern_vohralik_2021a}, building
on~\cite[Proposition~4.2]{Cost_McInt_Bog_Poinc_10} and~\cite[Theorem~7.2]{Demk_Gop_Sch_ext_II_09}.

In addition to the analysis of a posteriori error estimators,
constrained and unconstrained minimization problems of the form~\eqref{eq_constrained_minimization}
and~\eqref{eq_unconstrained_minimization} are also instrumental in the design of stable local
commuting interpolation operators having the projection property under minimal regularity
and in the equivalence of ``global-best'' and ``local-best'' approximations,
see~\cite{Veeser_approx_grads_16, Tant_Vees_Verf_loc_RD_15, chaumontfrelet_vohralik_2021a, ern_gudi_smears_vohralik_2022}.

The goal of the present work is threefold: (i) to establish a $\BH(\ccurl)$-variant
of~\eqref{eq_constrained_minimization} and~\eqref{eq_unconstrained_minimization} on a
vertex patch of tetrahedra; (ii) present a complete theory of constrained
and unconstrained local minimization problems in the De Rham complex in three space dimensions,
realizing that the $\BH(\ccurl)$-minimization was the last piece missing; (iii) complement on
and improve the results presented in~\cite{ern_vohralik_2020a} for the treatment of boundary
patches.

The remainder of this work is organized as follows. In Section~\ref{section_settings},
we specify the setting as well as the notation. Section~\ref{section_results}
presents our main results, and we show in Section~\ref{sec_inhom_BC} that these also
cover the case of inhomogeneous boundary conditions. Section~\ref{section_preliminary}
then collects some technical results and detailed notations used in the bulk of the proofs
for interior patches in Section~\ref{section_interior_patches}. We treat the case
of boundary patches in Section~\ref{section_boundary_patches}. We label as ``Proposition'' known results, whereas the main new results are named
``Theorem'' or ``Corollary''.

\section{Setting}
\label{section_settings}

\subsection{Vertex patch}
\label{sec_VP}

Throughout this work, $\CTa$ denotes a patch of tetrahedra, a finite collection
of closed nontrivial tetrahedra $K \subset \mathbb R^3$ that all have $\ba$ as vertex,
and which is such that for two elements $K_\pm \in \CTa$, the intersection $K_- \cap K_+$
is either the vertex $\ba$, a full edge of both $K_-$ and $K_+$, or a full face of both
$K_-$ and $K_+$. We also assume that the patch is face connected, meaning
that a path between two points in two different tetrahedra in $\CTa$ can always pass through
interiors of tetrahedra faces. We denote by $\oma$ the interior of $\cup_{K \in \CTa} K$.
We suppose that it has a Lipschitz boundary $\partial \oma$ and that $\overline{\oma}$ is 
homotopic to a ball. For a tetrahedron $K$, $\bn_K$ is its unit normal vector, outward to $K$.
For the applications we have in mind, this situation appears when $\ba$ is a vertex of a
simplicial mesh $\CT_h$ of some computational domain $\Omega$ in the context of finite
element methods, see, e.g., \cite{ciarlet_2002a,Bof_Brez_For_MFEs_13,ern_guermond_2021a}.

\begin{figure}
\centerline{\includegraphics[width=0.4\linewidth]{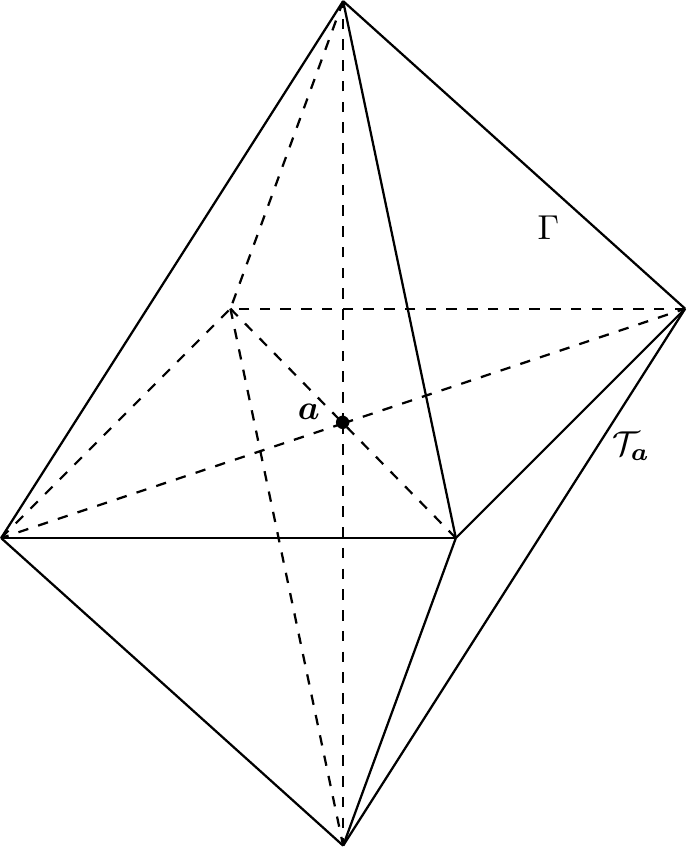} \qquad \includegraphics[width=0.4\linewidth]{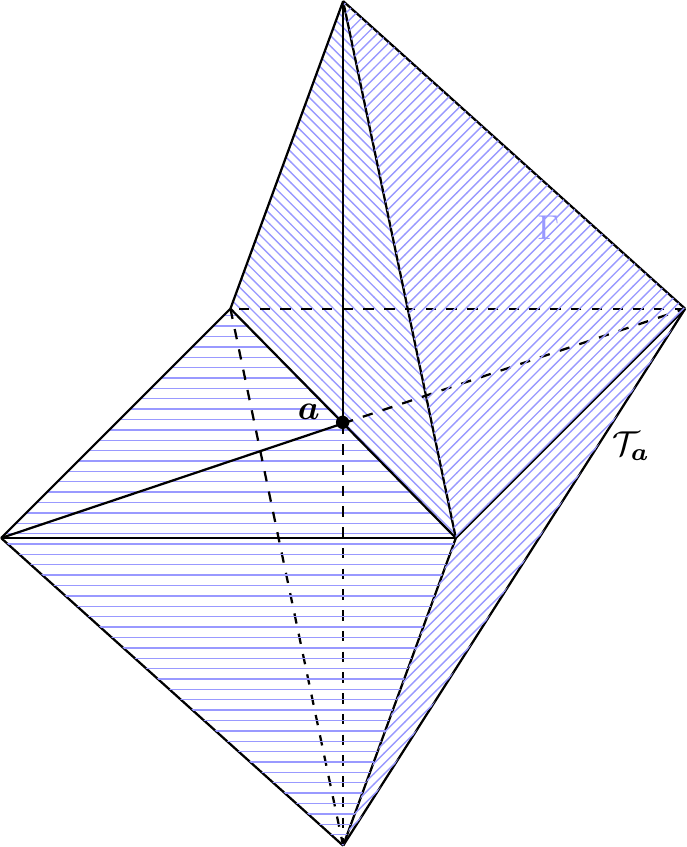}}
\caption{Interior patch (left) and boundary patch (right, with $\Ga = \emptyset$ and consequently $\Gamma = \partial \oma$)}
\label{fig_patches_1}
\end{figure}

Let $\CFa$ denote the set of the (closed) faces of the tetrahedra of the patch;
with each face $F \in \CFa$,
we associate a unit normal vector $\bn_F$ of an arbitrary but fixed orientation.
We will distinguish two situations. When $\oma$ contains an open ball around $\ba$,
we call $\CTa$ an ``interior patch'' and we set $\Gamma \eq \partial \oma$ and
$\Ga \eq \emptyset$, see Figure~\ref{fig_patches_1}, left, for an illustration.
When this is not the case, we speak of a ``boundary patch''. Then, we suppose that
there is a cone $\mathcal{C}$ with the vertex $\ba$ and a strictly positive solid angle such
that $\mathcal{C} \cap \oma = \emptyset$, forming an ``opening''.
Since $\oma$ has a Lipschitz boundary, there is exactly one such an opening.
In this case, $\Ga$ corresponds to some or all faces $F \in \CFa$ lying on the
boundary of $\oma$ and sharing the vertex $\ba$ and
$\Gamma \eq \partial \oma \setminus \overline \Ga$. In all cases,
we suppose that both $\Gamma$ and $\Ga$ are connected
and have Lipschitz boundaries, which in particular excludes
``checkerboard'' boundary patterns. 
Figure~\ref{fig_patches_1}, right, and Figure~\ref{fig_patches_2} provide illustrations.

\begin{figure}
\centerline{\includegraphics[width=0.4\linewidth]{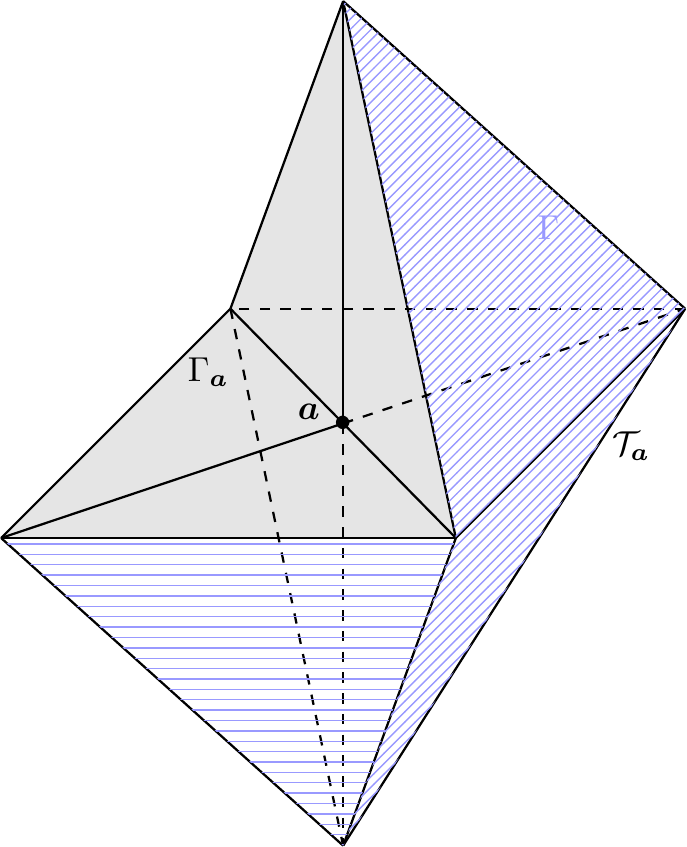} \qquad \includegraphics[width=0.4\linewidth]{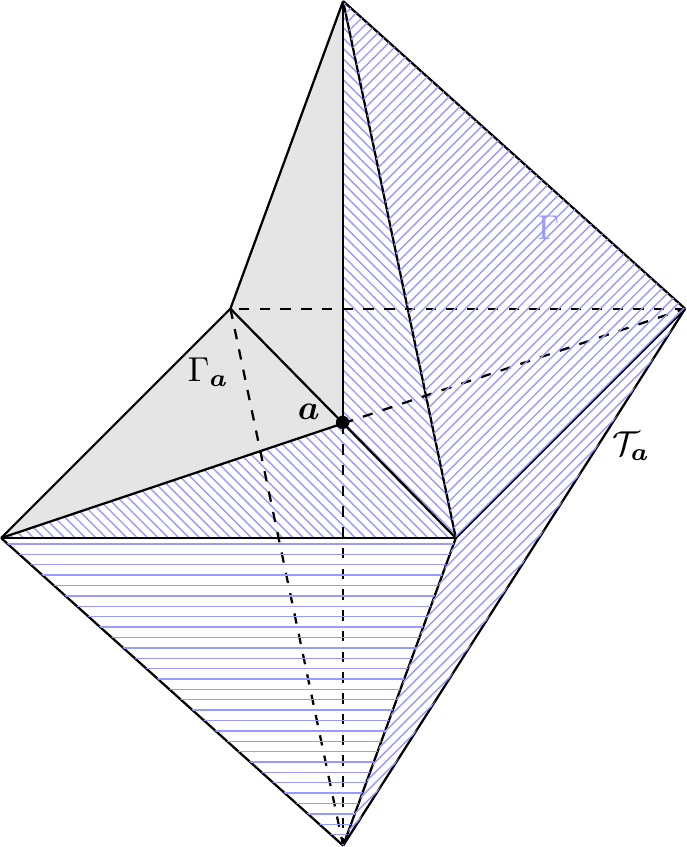}}
\caption{Boundary patches with $\Ga \neq \emptyset$; $\Ga$ corresponding to all faces $F \in \CFa$ lying on the boundary of $\oma$ and sharing the vertex $\ba$ (left), $\Ga$ corresponding to some faces $F \in \CFa$ lying on the boundary of $\oma$ and sharing the vertex $\ba$ (right)}
\label{fig_patches_2}
\end{figure}

We will say that two elements $K_\pm \in \CTa$ of the patch are neighbors if and only
if they share a face, i.e., if there exists $F \in \CFa$ such that $F = K_- \cap K_+$.

\subsection{Shape regularity}

For a tetrahedron $K$, let $h_K$ and $\rho_K$ respectively
denote the diameter of $K$ and the diameter of the largest ball
contained in $K$. The shape-regularity parameter $\kappa_K \eq h_K/\rho_K$ is then
a measure of the ``flatness'' of $K$, see, e.g., \cite{ciarlet_2002a,ern_guermond_2021a}.
If $\CT$ is a collection of tetrahedra, we denote by $\kappa_\CT \eq \max_{K \in \CT} \kappa_K$
the shape-regularity parameter of $\CT$.

\subsection{Functional spaces} \label{sec_fct_sp}

If $\omega \subset \mathbb R^3$ is a domain (open, bounded, and connected set)
with Lipschitz boundary, $H^1(\omega)$,
$\BH(\ccurl,\omega)$, and $\BH(\ddiv,\omega)$ are the usual Sobolev
spaces~\cite{adams_fournier_2003a,Bof_Brez_For_MFEs_13,ern_guermond_2021a,girault_raviart_1986a},
$\BH^1(\omega) \eq [ H^1(\omega) ]^3$, and $\BL^2(\omega) \eq [L^2(\omega)]^3$.
If $\gamma \subset \partial \omega$ is a relatively
open subset of the boundary of $\omega$, $H^1_{0,\gamma}(\omega)$ is the
subset of functions of $H^1(\omega)$ with vanishing trace on $\gamma$, and
$\BH^1_{0,\gamma}(\omega) \eq [H^1_{0,\gamma}(\omega) ]^3$. In the $\BH(\ccurl)$ setting,
we then denote
\begin{equation*}
\BH_{0,\gamma}(\ccurl,\omega)
\eq
\left \{
\bv \in \BH(\ccurl,\omega) \; | \; \bv \vp \bn_{\om} = \nv \text{ on } \gamma
\right \},
\end{equation*}
where the notion of trace is understood by duality, i.e.,
$\bv \vp \bn_{\om} = \nv$ on $\gamma$ means that
\[
(\curl \bv,\bphi)_{\omega} - (\bv,\curl \bphi)_{\omega} = 0 \qquad
\forall \bphi \in \BH^1_{0,\partial \omega \setminus \overline{\gamma}}(\omega).
\]
In the $\BH(\ddiv)$ setting, similarly,
\begin{equation*}
\BH_{0,\gamma}(\ddiv,\omega)
\eq
\left \{
\bv \in \BH(\ddiv,\omega) \; | \; \bv \scp \bn_{\om} = 0 \text{ on } \gamma
\right \},
\end{equation*}
where $\bv \scp \bn_{\om} = 0$ on $\gamma$ means that
\[
    (\div \bv,\phi)_{\omega} + (\bv,\grad \phi)_{\omega} = 0
\qquad
\forall \phi \in H^1_{0,\partial \omega \setminus \overline{\gamma}}(\omega).
\]

We refer the reader to~\cite{fernandes_gilardi_1997a} for a detailed treatment of
boundary conditions in $\BH(\ccurl,\omega)$ and $\BH(\ddiv,\omega)$. For the sake
of simplicity, we also define $L^2_{0,\gamma}(\omega)$ as $L^2(\omega)$ if $\gamma$ is
non-empty, and as the subset of $L^2(\omega)$ with functions of zero mean value on $\omega$
if $\gamma$ is empty. We will also employ the above notations if $\omega$ is a (closed)
tetrahedron and $\gamma$ the union of some of its (closed) faces.

\subsection{Piecewise polynomial spaces}

Consider a tetrahedron $K$. For $q \geq 0$, $\CP_q(K)$ is the set of polynomials
of degree less than or equal to $q$, and $\BCP_q(K) \eq [\CP_q(K)]^3$. The spaces
of Raviart--Thomas and N\'ed\'elec polynomials are then defined by
\begin{equation*}
\RT_q(K) \eq \BCP_q(K) + \bx \CP_q(K),
\qquad
\ND_q(K) \eq \BCP_q(K) + \bx \times \BCP_q(K),
\end{equation*}
see~\cite{raviart_thomas_1977a, nedelec_1980a, Bof_Brez_For_MFEs_13, ern_guermond_2021a}.
If $\CT$ is a collection of tetrahedra, we employ the notations $\CP_q(\CT)$, $\RT_q(\CT)$,
and $\ND_q(\CT)$ for functions whose restrictions to each $K \in \CT$ belong respectively to
$\CP_q(K)$, $\RT_q(K)$, and $\ND_q(K)$. Notice that these spaces have no ``built-in''
continuity conditions (they form the so-called broken spaces); we impose the continuity
conditions by an intersection with the Sobolev spaces from Section~\ref{sec_fct_sp}.

\section{Main results}
\label{section_results}

We start by stating the following result
from~\cite[Corollaries~3.3 and~3.8]{ern_vohralik_2020a}.%
\footnote{Actually, some geometrical situations for boundary patches were excluded
in~\cite{ern_vohralik_2020a} (at most two simplices in the patch $\CTa$ or the existence
of an interior vertex in $\Gamma$; these are now covered by the proof detailed in
Section~\ref{section_boundary_patches} below).} Here and below, $C(x)$ means a generic
constant only depending on the quantity $x$. Thus, our results only depend on the
shape-regularity $\kappa_{\CTa}$ of the patch $\CTa$ and not on the underlying polynomial
degree $p$ (or mesh size $h$ or any other parameter).

\begin{proposition}[Constrained minimization in $\BH_{0,\Gamma}(\ddiv,\oma)$]
\label{prop_Hd_cons}
For all $p \geq 0$, $\btau_p \in \RT_p(\CTa)$, and
$r_p \in \CP_p(\CTa) \cap L^2_{0,\Ga}(\oma)$, we have
\begin{equation*}
\min_{\substack{
\bw_p \in \RT_p(\CTa) \cap \BH_{0,\Gamma}(\ddiv,\oma)
\\
\div \bw_p = r_p
}}
\|\bw_p - \btau_p\|_{\oma}
\leq
C(\kappa_{\CTa})
\min_{\substack{
\bw \in \BH_{0,\Gamma}(\ddiv,\oma)
\\
\div \bw = r_p
}}
\|\bw - \btau_p\|_{\oma}.
\end{equation*}
\end{proposition}

Our first new result is an easy consequence of the Proposition~\ref{prop_Hd_cons}
and treats unconstrained minimization in $\BH(\ccurl,\oma)$.

\begin{corollary}[Unconstrained minimization in $\BH_{0,\Gamma}(\ccurl,\oma)$]
\label{cor_Hc}
For all $p \geq 0$ and all $\btau_p \in \RT_p(\CTa)$, we have
\begin{equation*}
\min_{\bv_p \in \ND_p(\CTa) \cap \BH_{0,\Gamma}(\ccurl,\oma)}
\|\curl \bv_p - \btau_p\|_{\oma}
\leq
C(\kappa_{\CTa})
\min_{\bv \in \BH_{0,\Gamma}(\ccurl,\oma)}
\|\curl \bv - \btau_p\|_{\oma}.
\end{equation*}
\end{corollary}

\begin{proof}
We proceed as in~\cite[proof of Theorem~1]{chaumontfrelet_vohralik_2021a}.
From our assumptions in Section~\ref{sec_VP}, $\oma$ is such that $\overline{\oma}$
is homotopic to a ball, the boundary $\partial \oma$ is Lipschitz, and $\Gamma$ is connected
and has a Lipschitz boundary. Thus, it follows that the range of the curl
operator $\curl$ acting on $\BH_{0,\Gamma}(\ccurl,\oma)$ is exactly the kernel of the divergence operator
on $\BH_{0,\Gamma}(\ddiv,\oma)$, and a similar property holds for the discrete spaces
$\ND_p(\CTa) \cap \BH_{0,\Gamma}(\ccurl,\oma)$ and
$\RT_p(\CTa) \cap \BH_{0,\Gamma}(\ddiv,\oma)$, see, e.g.,
\cite{Arn_Falk_Winth_FEC_06,Bof_Brez_For_MFEs_13,fernandes_gilardi_1997a}.
Then, the result follows from Proposition~\ref{prop_Hd_cons}, since
\begin{align*}
\min_{\bv_p \in \ND_p(\CTa) \cap \BH_{0,\Gamma}(\ccurl,\oma)}
\|\curl \bv_p - \btau_p\|_{\oma}
&=
\min_{\substack{\bw_p \in \RT_p(\CTa) \cap \BH_{0,\Gamma}(\ddiv,\oma) \\ \div \bw_p = 0}}
\|\bw_p - \btau_p\|_{\oma}
\\
&\leq
C(\kappa_{\CTa})
\min_{\substack{\bw \in \BH_{0,\Gamma}(\ddiv,\oma) \\ \div \bw = 0}}
\|\bw - \btau_p\|_{\oma}
\\
&=
C(\kappa_{\CTa})
\min_{\bv \in \BH_{0,\Gamma}(\ccurl,\oma)}
\|\curl \bv - \btau_p\|_{\oma}.
\end{align*}
\end{proof}

The central result of this work is the following theorem which addresses
constrained minimization in $\BH_{0,\Gamma}(\ccurl,\oma)$. Its proof is lengthy, and
postponed to Sections~\ref{section_preliminary}--\ref{section_boundary_patches}.

\begin{theorem}[Constrained minimization in $\BH_{0,\Gamma}(\ccurl,\oma)$]
\label{thm_Hc_cons}
For all $p \geq 0$, $\bchi_p \in \ND_p(\CTa)$, and
$\bj_p \in \RT_p(\CTa) \cap \BH_{0,\Gamma}(\ddiv,\oma)$ with $\div \bj_p = 0$,
we have
\begin{equation*}
\min_{\substack{
\bv_p \in \ND_p(\CTa) \cap \BH_{0,\Gamma}(\ccurl,\oma)
\\
\curl \bv_p = \bj_p
}} \|\bv_p - \bchi_p\|_{\oma}
\leq
C(\kappa_{\CTa})
\min_{\substack{
\bv \in \BH_{0,\Gamma}(\ccurl,\oma)
\\
\curl \bv = \bj_p
}} \|\bv - \bchi_p\|_{\oma}.
\end{equation*}
\end{theorem}

Our last result concerns unconstrained minimization in $H^1(\oma)$. It generalizes
the result previously obtained in~\cite[Corollaries~3.1 and~3.7]{ern_vohralik_2020a},
which was limited to the case where $\bchi_p = \grad \chi_p$ for $\chi_p \in \CP_{p+1}(\CTa)$
with $\chi_p = 0$ on $\Gamma$, and where the geometrical setting of boundary patches had some
restrictions.

\begin{corollary}[Unconstrained minimization in $H^1_{0,\Gamma}(\oma)$]
\label{cor_H1}
For all $p \geq 0$ and all $\bchi_p \in \ND_p(\CTa)$, we have
\begin{equation*}
\min_{v_p \in \CP_{p+1}(\CTa) \cap H^1_{0,\Gamma}(\oma)} \|\grad v_p - \bchi_p\|_{\oma}
\leq
C(\kappa_{\CTa})
\min_{v \in H^1_{0,\Gamma}(\oma)} \|\grad v - \bchi_p\|_{\oma}.
\end{equation*}
\end{corollary}

\begin{proof}
We proceed as in~\cite[proof of Theorem~2]{chaumontfrelet_vohralik_2021a}, similarly as
above in Corollary~\ref{cor_Hc}. Because the patch subdomain $\oma$ is
such that $\overline{\oma}$ is homotopic to a ball, the boundary $\partial \oma$
is Lipschitz, and $\Gamma$ is connected and has a Lipschitz boundary,
the kernel of the curl operator in $\BH_{0,\Gamma}(\ccurl,\oma)$ is exactly
$\grad (H^1_{0,\Gamma}(\oma))$, so that
\begin{equation*}
\min_{v \in H^1_{0,\Gamma}(\oma)} \|\grad v - \bchi_p\|_{\oma}
=
\min_{\substack{\bv \in \BH_{0,\Gamma}(\ccurl,\oma) \\ \curl \bv = \bzero}}
\|\bv - \bchi_p\|_{\oma}.
\end{equation*}
Similarly, at the discrete level, the equality
\begin{equation*}
\min_{v_p \in \CP_{p+1}(\CTa) \cap H^1_{0,\Gamma}(\oma)} \|\grad v_p - \bchi_p\|_{\oma}
=
\min_{\substack{\bv_p \in \ND_p(\CTa) \cap \BH_{0,\Gamma}(\ccurl,\oma) \\ \curl \bv_p = \bzero}}
\|\bv_p - \bchi_p\|_{\oma}
\end{equation*}
holds true, see, e.g., \cite{Arn_Falk_Winth_FEC_06,Bof_Brez_For_MFEs_13,fernandes_gilardi_1997a}.
Then the result follows from Theorem~\ref{thm_Hc_cons}.
\end{proof}

\begin{remark}[Converse inequalities]
The converse inequalities to all the statements above trivially hold with constant one.
\end{remark}

\begin{remark}[Unconstrained $L^2(\oma)$ and constrained $H^1(\oma)$ minimizations]
In principle, we could consider two additional minimization problems with the considered spaces,
namely (i) the unconstrained minimization in $L^2(\oma)$; and (ii) the constrained minimization
in $H^1(\oma)$. However, these problems are trivial, since in both cases, the continuous and
discrete minimizers are the same. Specifically, we have
\begin{equation*}
\min_{q_p \in \CP_{p}(\CTa)}
\|q_p-r_p\|_{\oma}
=
\min_{q \in L^2(\oma)}
\|q-r_p\|_{\oma}
=
0
\end{equation*}
for all $r_p \in \CP_p(\CTa)$, as well as
\begin{equation*}
\min_{\substack{
v_p \in \CP_{p+1}(\CTa) \cap H^1_{0,\Gamma}(\oma)
\\
\grad v_p = \bg_p
}}
\|v_p-\chi_p\|_{\oma}
=
\min_{\substack{
v \in H^1_{0,\Gamma}(\oma)
\\
\grad v = \bg_p
}}
\|v-\chi_p\|_{\oma}
\end{equation*}
for all $\chi_p \in \CP_{p+1}(\CTa)$ and $\bg_p \in \ND_p(\CTa) \cap \BH_{0,\Gamma}(\ccurl,\oma)$
such that $\curl \bg_p = \bzero$. We refer to~\cite[Section~3.3]{chaumontfrelet_vohralik_2021a}
for some more considerations in this direction.
\end{remark}

\begin{remark}[Stable broken polynomial extensions]
The minimization problems considered above can be equivalently formulated as broken
polynomial extensions as initially stated in~\cite{braess_pillwein_schoberl_2009a},
where a discontinuous minimizer with prescribed jumps is sought for instead of a
conforming one as above. The two formulations are actually equivalent up to a shift,
as shown in~\cite[Section~3.1]{ern_vohralik_2020a}
or~\cite[Lemma 6.8]{chaumontfrelet_ern_vohralik_2021a}.
\end{remark}

\section{Extension to inhomogeneous boundary conditions}
\label{sec_inhom_BC}

Proposition~\ref{prop_Hd_cons}, Corollary~\ref{cor_Hc}, Theorem~\ref{thm_Hc_cons}, and
Corollary~\ref{cor_H1} are only stated for homogeneous boundary conditions on the $\Gamma$
part of the boundary of $\oma$. Supposing inhomogeneous boundary conditions that are suitable
piecewise polynomials, these can be lifted to see that the above theory also covers this case.
We now present the equivalent reformulations together with their proofs. In place of the boundary
data, we rather start directly from the liftings, denoted as $\bsig_p$ and $\sigma_p$ below.
These results are in practice particularly useful in the case of boundary patches, where the
inhomogeneous boundary conditions of the patch problems stem from inhomogeneous Dirichlet, Neumann or (homogeneous) Robin boundary conditions of the original partial differential equation
(cf. respectively the discussion in~\cite[Section~4]{ern_vohralik_2020a} and in~\cite{chaumontfrelet_ern_vohralik_2021b}). More precisely, in the applications,
inhomogeneous boundary conditions only appear on the part of $\Gamma$ corresponding to the faces sharing the vertex $\ba$,
which is of course covered by the presentation here.

We start with the $\BH(\ddiv,\oma)$-case of Proposition~\ref{prop_Hd_cons}.
For the datum $\bsig_p$ given below, we say that $\bw \scp \bn_{\oma} = \bsig_p \scp \bn_{\oma}$ on $\Gamma$ if $\bw-\bsig_p \in \BH_{0,\Gamma}(\ddiv,\oma)$

\begin{corollary}[Constrained minimization in $\BH(\ddiv,\oma)$ with inhomogeneous boundary conditions]
\label{cor_Hd_cons_B}
For all $p \geq 0$, $\btau_p \in \RT_p(\CTa)$, $\bsig_p \in \RT_p(\CTa) \cap \BH(\ddiv,\oma)$,
and $r_p \in \CP_p(\CTa)$ with the additional condition
$(\bsig_p \scp \bn_{\oma},1)_{\pt \oma} = (r_p,1)_{\oma}$ if $\Ga = \emptyset$,
we have
\begin{equation*}
\min_{\substack{
\bw_p \in \RT_p(\CTa) \cap \BH(\ddiv,\oma)
\\
\div \bw_p = r_p
\\
\bw_p \scp \bn_{\oma} = \bsig_p \scp \bn_{\oma} \text{ on } \Gamma
}}
\|\bw_p - \btau_p\|_{\oma}
\leq
C(\kappa_{\CTa})
\min_{\substack{
\bw \in \BH(\ddiv,\oma)
\\
\div \bw = r_p
\\
\bw \scp \bn_{\oma} = \bsig_p \scp \bn_{\oma} \text{ on } \Gamma
}}
\|\bw - \btau_p\|_{\oma}.
\end{equation*}
\end{corollary}

\begin{proof}
We show the equivalence with Proposition~\ref{prop_Hd_cons}, by a shift by the piecewise
polynomial datum $\bsig_p$. Indeed, suppose the setting of Corollary~\ref{cor_Hd_cons_B};
the converse direction is similar. Let $\bw = \bw^0 + \bsig_p$ with
$\bw^0 \in \BH_{0,\Gamma}(\ddiv,\oma)$ and $\bw_p = \bw^0_p + \bsig_p$ with
$\bw^0_p \in \RT_p(\CTa) \cap \BH_{0,\Gamma}(\ddiv,\oma)$. Note that
$\div \bsig_p \in \CP_p(\CTa)$ satisfies $(r_p - \div \bsig_p,1)_{\oma} = 0$ if $\Ga = \emptyset$.
Thus, setting $\tilde r_p \eq r_p - \div \bsig_p$ and $\tilde \btau_p \eq \btau_p - \bsig_p$,
we have $\tilde r_p \in \CP_p(\CTa) \cap L^2_{0,\Ga}(\oma)$ and $\tilde \btau_p \in \RT_p(\CTa)$.
This means that $\tilde r_p$ and $\tilde \btau_p$ are eligible data for Theorem~\ref{thm_Hc_cons},
which crucially lead to the same minimization values.
\end{proof}

The unconstrained $\BH(\ccurl,\oma)$-case of Corollary~\ref{cor_Hc} is actually easier,
since there is no differential operator constraint. Similarly to the $\BH(\ddiv,\oma)$
case, the notation $\bw \times \bn_{\oma} = \bsig_p \times \bn_{\oma}$ on $\Gamma$ means that
$\bw-\bsig_p \in \BH_{0,\Gamma}(\ccurl,\oma)$.

\begin{corollary}[Unconstrained minimization in $\BH(\ccurl,\oma)$ with inhomogeneous boundary conditions]
\label{cor_Hc_B}
For all $p \geq 0$, all $\bsig_p \in \ND_p(\CTa) \cap \BH(\ccurl,\oma)$,
and all $\btau_p \in \RT_p(\CTa)$, we have
\begin{equation*}
\min_{\substack{
\bv_p \in \ND_p(\CTa) \cap \BH(\ccurl,\oma)
\\
\bv_p \times \bn_{\oma} = \bsig_p \times \bn_{\oma} \text{ on } \Gamma
}}
\|\curl \bv_p - \btau_p\|_{\oma}
\leq
C(\kappa_{\CTa})
\min_{\substack{\bv \in \BH(\ccurl,\oma)\\
\bv \times \bn_{\oma} = \bsig_p \times \bn_{\oma} \text{ on } \Gamma}}
\|\curl \bv - \btau_p\|_{\oma}.
\end{equation*}
\end{corollary}

\begin{proof}
We show the equivalence with Corollary~\ref{cor_Hc}, by a shift by the piecewise polynomial datum $\bsig_p$. Suppose the setting of Corollary~\ref{cor_Hc_B}; the converse direction is similar. Let $\bv = \bv^0 + \bsig_p$ with $\bv^0 \in \BH_{0,\Gamma}(\ccurl,\oma)$ and $\bv_p = \bv^0_p + \bsig_p$ with $\bv^0_p \in \ND_p(\CTa) \cap \BH_{0,\Gamma}(\ccurl,\oma)$. Note that $\curl \bsig_p \in \RT_p(\CTa)$ (actually also in $\BH(\ddiv,\oma)$ with $\div(\curl \bsig_p) = 0$).
Thus, setting $\tilde \btau_p \eq \btau_p - \curl \bsig_p$, we have $\tilde \btau_p \in \RT_p(\CTa)$, which is an eligible datum for Corollary~\ref{cor_Hc}, crucially leading to the same minimization values.
\end{proof}

The constrained $\BH(\ccurl,\oma)$ case of Theorem~\ref{thm_Hc_cons} is similar to the situation
of Corollary~\ref{cor_Hd_cons_B}:

\begin{corollary}[Constrained minimization in $\BH(\ccurl,\oma)$ with inhomogeneous boundary conditions]
\label{cor_Hc_cons_B}
For all $p \geq 0$, $\bchi_p \in \ND_p(\CTa)$, $\bsig_p \in \ND_p(\CTa) \cap \BH(\ccurl,\oma)$,
and $\bj_p \in \RT_p(\CTa) \cap \BH(\ddiv,\oma)$ with
$\bj_p {\cdot} \bn_{\oma} = (\curl \bsig_p) {\cdot} \bn_{\oma}$
on $\Gamma$ and $\div \bj_p = 0$, we have
\begin{equation*}
\min_{\substack{
\bv_p \in \ND_p(\CTa) \cap \BH(\ccurl,\oma)
\\
\curl \bv_p = \bj_p
\\
\bv_p \times \bn_{\oma} = \bsig_p \times \bn_{\oma} \text{ on } \Gamma
}} \|\bv_p - \bchi_p\|_{\oma}
\leq
C(\kappa_{\CTa})
\min_{\substack{
\bv \in \BH(\ccurl,\oma)
\\
\curl \bv = \bj_p
\\
\bv \times \bn_{\oma} = \bsig_p \times \bn_{\oma} \text{ on } \Gamma
}} \|\bv - \bchi_p\|_{\oma}.
\end{equation*}
\end{corollary}

\begin{proof} We show the equivalence with Theorem~\ref{thm_Hc_cons}, again by a shift by the piecewise polynomial datum $\bsig_p$. Suppose the setting of Corollary~\ref{cor_Hc_cons_B}; the converse direction is similar. Let $\bv = \bv^0 + \bsig_p$ with $\bv^0 \in \BH_{0,\Gamma}(\ccurl,\oma)$ and $\bv_p = \bv^0_p + \bsig_p$ with $\bv^0_p \in \ND_p(\CTa) \cap \BH_{0,\Gamma}(\ccurl,\oma)$. Note that $\curl \bsig_p \in \RT_p(\CTa) \cap \BH(\ddiv,\oma)$ with $\div(\curl \bsig_p) = 0$ and $(\curl \bsig_p) {\cdot} \bn_{\oma} = \bj_p {\cdot} \bn_{\oma}$ on $\Gamma$.
Thus, setting $\tilde \bj_p \eq \bj_p - \curl \bsig_p$ and $\tilde \bchi_p \eq \bchi_p - \bsig_p$, we have $\tilde \bj_p \in \RT_p(\CTa) \cap \BH_{0,\Gamma}(\ddiv,\oma)$ with $\div \tilde \bj_p = 0$ and $\tilde \bchi_p \in \ND_p(\CTa)$. This means that $\tilde \bj_p$ and $\tilde \bchi_p$ are eligible data for Theorem~\ref{thm_Hc_cons}, which crucially lead to the same minimization values.
\end{proof}

We now finally present how Corollary~\ref{cor_H1} covers inhomogeneous boundary conditions:

\begin{corollary}[Unconstrained minimization in $H^1(\oma)$ with inhomogeneous boundary conditions]
\label{cor_H1_B}
For all $p \geq 0$, $\sigma_p \in \CP_{p+1}(\CTa) \cap H^1(\oma)$, and all $\bchi_p \in \ND_p(\CTa)$, we have
\begin{equation*}
\min_{\substack{
v_p \in \CP_{p+1}(\CTa) \cap H^1(\oma)
\\
v_p = \sigma_p \text{ on } \Gamma
}}
\|\grad v_p - \bchi_p\|_{\oma}
\leq
C(\kappa_{\CTa})
\min_{\substack{
v \in H^1(\oma)
\\
v = \sigma_p \text{ on } \Gamma
}}
\|\grad v - \bchi_p\|_{\oma}.
\end{equation*}
\end{corollary}

\begin{proof}
The proof passes through equivalence with Corollary~\ref{cor_H1}. It is similar as above but again easier, since there is no differential operator constraint. Every $v_p \in \CP_{p+1}(\CTa) \cap H^1(\oma)$ with $v_p = \sigma_p$ on $\Gamma$ can be written as $v_p = v^0_p + \sigma_p$, where $v^0_p \in \CP_{p+1}(\CTa) \cap H^1_{0,\Gamma}(\oma)$, and similarly for $v = v^0 + \sigma_p$ with $v^0 \in H^1_{0,\Gamma}(\oma)$. Then, $\tilde \bchi_p \eq \bchi_p - \grad \sigma_p$ lies in $\ND_p(\CTa)$ and forms an eligible datum for Corollary~\ref{cor_H1}, which leads to the same minimization values.
\end{proof}

The remainder of this manuscript is dedicated to establishing Theorem~\ref{thm_Hc_cons}.

\section{Detailed notation and preliminary results for the proofs}
\label{section_preliminary}

\subsection{Tangential traces} \label{sec_tang_traces}

Consider a tetrahedron $K$ and $\CF \subset \CF_K$, a (sub)set of its faces.
The definition of tangential traces of $\BH(\ccurl,K)$ functions on the faces
$F \in \CF$ is a subtle matter. As we only work with piecewise polynomial traces,
one way is to proceed with the liftings as in Corollaries~\ref{cor_Hc_B} and~\ref{cor_Hc_cons_B}.
We rather proceed here following~\cite{chaumontfrelet_ern_vohralik_2020a, chaumontfrelet_ern_vohralik_2021a},
introducing the more general concept of a weak tangential trace, solely working with the boundary
data (called $\br_p$ here) and not directly their N\'ed\'elec liftings ($\bsig_p$ in
Corollaries~\ref{cor_Hc_B} and~\ref{cor_Hc_cons_B}).

If $\bv \in \BH^1(K)$ and $F \in \CF_K$, we denote by
\begin{equation} \label{eq_tang_trace}
\pi_F^{\btau}(\bv)
\eq
\left (\bv - (\bv {\cdot} \bn_F) \bn_F\right )|_F \in \BL^2(F)
\end{equation}
its ``usual'' tangential trace on the face $F$
(the orientation of $\bn_F$ is not important here).
We then define surface N\'ed\'elec spaces on faces as traces of volume N\'ed\'elec polynomials,
setting
\begin{equation*}
\ND_p(F) \eq \left \{
\pi^{\btau}_F(\bv)
\; | \;
\bv \in \ND_p(K)
\right \},
\end{equation*}
and, if $\bw \in \ND_p(F)$,
\begin{equation*}
\scurl_F \bw \eq (\curl \bv)|_F {\cdot} \bn_F,
\end{equation*}
where $\bv \in \ND_p(K)$ is such that $\bw = \pi_F^{\btau} (\bv)$
(the orientation of $\bn_F$ counts in the definition of $\scurl_F$).
One easily checks that for any face $F$ (which is geometrically a triangle),
these definitions are independent of the choice of the tetrahedron $K$ such
that $F \in \CF_K$.

For a collection of faces $\CF \subset \CF_K$, we introduce
\begin{equation*}
\ND_p(\CF) \eq \prod_{F \in \CF} \ND_p(F)
\end{equation*}
and
\begin{equation}
\label{eq_ttrace_Gamma}
\ND_p(\Gamma_{\CF}) \eq \left \{
\bw \in \ND_p(\CF) \; | \;
\exists \bv \in \ND_p(K); \;
\bw|_F = \pi^{\btau}_F (\bv)
\;
\forall F \in \CF
\right \}.
\end{equation}
Notice that there is an induced tangential trace compatibility condition on each edge
shared by faces of $\CF$ in the definition of $\ND_p(\Gamma_{\CF})$.

We then define a weak notion of tangential trace using integration by parts.
Specifically, if $\bv \in \BH(\ccurl,K)$ and $\br_p \in \ND_p(\Gamma_{\CF})$
for some $p \geq 0$, the statement ``$\bv|_{\CF}^{\btau} = \br_p$'' means that
\begin{equation}
\label{eq_weak_trace}
(\curl \bv,\bphi)_K - (\bv,\curl \bphi)_K
=
\sum_{F \in \CF} (\br_p,\bphi \times \bn_K)_F
\quad
\forall \bphi \in \BH^1_{\btau,\CF^{\rm c}}(K),
\end{equation}
where
\begin{equation*}
\BH^1_{\btau,\CF^{\rm c}}(K)
\eq
\left \{
\bw \in \BH^1(K)
\; | \;
\pi_F^{\btau}(\bw) = \bzero
\quad
\forall F \in \CF^{\rm c} \eq \CF_K \setminus \CF
\right \}.
\end{equation*}
Notice that when $\bv \in \BH^1(K)$, $\bv|_{\CF}^{\btau} = \br_p$ if and only if
$\pi_F^{\btau}(\bv) = \br_p|_F$ for all $F \in \CF$.

\begin{remark}[Compatibility of the weak definitions of tangential traces]
Let $\Gamma_{\CF}$ be the portion of the boundary of $K$ corresponding
to the faces in $\CF$ and $\Gamma_{\CF}^{\rm c}$ the corresponding complement
(both open). We note that when $\br_p = \bzero$, the subspace of $\bv \in \BH(\ccurl,K)$
such that $\bv|_{\CF}^{\btau} = \bzero$ is identical with $\BH_{0,\Gamma_{\CF}}(\ccurl,K)$
from Section~\ref{sec_fct_sp}, where test functions $\bw \in \BH^1(K)$ such that
$\bw = \bzero$ on $\Gamma_{\CF}^{\rm c}$ are used. Using test functions $\bw \in \BH^1(K)$
such that only $\pi_F^{\btau}(\bw) = \bzero$ on $\Gamma_{\CF}^{\rm c}$ will be exploited
below for the Piola transforms.
\end{remark}

\subsection{Piola mappings}

Consider two tetrahedra $\Kin,\Kout \in \CTa$ and an invertible affine transformation
$\psi: \Kin \to \Kout$. Such a transformation can be uniquely identified by specifying
which vertex of $\Kin$ is mapped to which vertex of $\Kout$. We denote by $\JAC$ the
(constant) Jacobian matrix of $\psi$ and we let $\varepsilon \eq \sign (\det \JAC)$.

We associate with $\psi$ two ``Piola'' mappings for vector-valued functions
$\bv: \Kin \to \mathbb R^3$ defined by
\begin{equation} \label{eq_Piola}
\psi^{\rm c}(\bv) \eq \JAC^{-T} \left (\bv \circ \psi^{-1}\right ),
\qquad
\psi^{\rm d}(\bv) \eq (\det \JAC)^{-1} \JAC \left (\bv \circ \psi^{-1}\right ).
\end{equation}
These mappings commute with the curl operator in the sense that
\begin{equation}
\label{eq_piola_commute}
\curl \left (\psi^{\rm c}(\bv)\right ) = \psi^{\rm d}\left (\curl \bv\right )
\end{equation}
whenever $\bv \in \BH(\ccurl,\Kin)$. In addition,
if $\bv_{\rm in} \in \BH(\ccurl,\Kin)$ and $\bw_{\rm out} \in \BH(\ccurl,\Kout)$
we have
\begin{equation}
\label{eq_piola_v_curlw}
(\psi^{\rm c}(\bv_{\rm in}),\curl \bw_{\rm out})_{\Kout}
=
\varepsilon (\bv_{\rm in},\curl \left ((\psi^{\rm c})^{-1} (\bw_{\rm out})\right ))_{\Kin}.
\end{equation}
Finally, we use the fact that the Piola mappings are stable in the sense that
\begin{equation}
\label{eq_piola_stable}
\|\psi^{\rm c}(\bv)\|_{\Kout}
\leq
C(\kappa_{\CTa})\|\bv\|_{\Kin}
\quad
\forall \bv \in \BL^2(\Kin).
\end{equation}
We refer the reader to~\cite[Section~9]{ern_guermond_2021a} for an in-depth presentation
of Piola mappings and proofs of the properties stated above.

\subsection{Stability in one tetrahedron}

We close this section with a simple extension of a result
from~\cite[Theorem~2]{chaumontfrelet_ern_vohralik_2020a},
corresponding to Theorem~\ref{thm_Hc_cons} (or more precisely Corollary~\ref{cor_Hc_cons_B})
where the vertex patch $\CTa$ is replaced by a single tetrahedron $K$.

\begin{definition}[Compatible data in a tetrahedron]
\label{def_compatible_data}
Let $K$ be a tetrahedron. Consider a (sub)set $\CF \subset \CF_K$ of the faces of $K$.
We say that $\bj_p \in \RT_p(K)$ and $\br_p \in \ND_p(\CF)$ are compatible data if
\begin{subequations}
\begin{align}
\label{eq_compatible_data_div}
\div \bj_p &= 0,
\\
\label{eq_compatible_data_trace}
\br_p &\in \ND_p(\Gamma_{\CF}),
\\
\label{eq_compatible_data_curl}
\bj_p {\cdot} \bn_F &= \scurl_F (\br_p|_F) \quad \forall F \in \CF.
\end{align}
\end{subequations}
\end{definition}

\begin{lemma}[Stable minimization in a tetrahedron]
\label{lemma_tetrahedron}
Consider a tetrahedron $K$ and a (sub)set $\CF$ of its faces.
For all $p \geq 0$, for all compatible data $\bj_p \in \RT_p(K)$
and $\br_p \in \ND_p(\CF)$ as per Definition~\ref{def_compatible_data},
and for all $\bchi_p \in \ND_p(K)$, we have
\begin{equation*}
\min_{\substack{
\bv_p \in \ND_p(K)
\\
\curl \bv_p = \bj_p
\\
\bv_p|^{\btau}_{\CF} = \br_p
}}
\|\bv_p - \bchi_p\|_K
\leq
C(\kappa_K)
\min_{\substack{
\bv \in \BH(\ccurl,K)
\\
\curl \bv = \bj_p
\\
\bv|^{\btau}_{\CF} = \br_p
}}
\|\bv - \bchi_p\|_K.
\end{equation*}
\end{lemma}

\begin{proof}
The proof proceeds by a shift by $\bchi_p$, similarly to that of Corollary~\ref{cor_Hc_B}.
Let us introduce $\tbj_p \eq \bj_p-\curl \bchi_p$ and
$\tbr_p|_F \eq \br_p|_F - \pi_F^{\btau}(\bchi_p)$ for all $F \in \CF$.
The new data $\tbj_p$ and $\tbr_p$ are compatible as per
Definition~\ref{def_compatible_data}, since $\bchi_p \in \ND_p(K)$. We now have
from~\cite[Theorem 2]{chaumontfrelet_ern_vohralik_2020a} (which corresponds to
Lemma~\ref{lemma_tetrahedron} when $\bchi_p = \bzero$) that
\begin{equation*}
\min_{\substack{
\tbv_p \in \ND_p(K)
\\
\curl \tbv_p = \tbj_p
\\
\tbv_p|^{\btau}_{\CF} = \tbr_p
}}
\|\tbv_p\|_K
\leq
C(\kappa_K)
\min_{\substack{
\tbv \in \BH(\ccurl,K)
\\
\curl \tbv = \tbj_p
\\
\tbv|^{\btau}_{\CF} = \tbr_p
}}
\|\tbv\|_K.
\end{equation*}
Denote respectively by $\bv_p^\star,\tbv_p^\star \in \ND_p(K)$ and
$\bv^\star,\tbv^\star \in \BH(\ccurl,K)$ the (unique) minimizers of the above left-
and right-hand sides. Then the respective inequalities write as
$\|\bv^\star_p - \bchi_p\|_K \leq C(\kappa_K) \|\bv^\star-\bchi_p\|_K$
and $\|\tbv^\star_p\|_K \leq C(\kappa_K)\|\tbv^\star\|_K$ and a shift by $\bchi_p$
shows that actually $\tbv_p^\star = \bv_p^\star - \bchi_p$ and $\tbv^\star = \bv^\star-\bchi_p$.
\end{proof}

\section{Proof of Theorem~\ref{thm_Hc_cons} for interior patches}
\label{section_interior_patches}

We first consider interior patches, i.e.,
the case where $\oma$ contains an open ball around $\ba$ (so that
$\ba \notin \partial \oma$), where $\Gamma = \partial \oma$ and $\Ga = \emptyset$.

We follow the approach introduced in~\cite{braess_pillwein_schoberl_2009a} and
extended in~\cite{ern_vohralik_2020a} and~\cite{chaumontfrelet_ern_vohralik_2021a},
so that our proof relies on an explicit construction of a discrete element
$\bxi_p \in \ND_p(\CTa) \cap \BH_{0,\Gamma}(\ccurl,\oma)$ satisfying $\curl \bxi_p = \bj_p$ and
\begin{equation*}
\|\bxi_p-\bchi_p\|_{\oma} \leq C(\kappa_{\CTa})
\min_{\substack{
\bv \in \BH_{0,\Gamma}(\ccurl,\oma)
\\
\curl \bv = \bj_p
}}
\|\bv-\bchi_p\|_{\oma}.
\end{equation*}
To construct this element, we pass through the patch one tetrahedron at a time,
following a suitable enumeration $K_1,K_2,\dots,K_{|\CTa|}$. At each step
$1 \leq j \leq |\CTa|$, $\bxi_p|_{K_j}$ is defined as the minimizer of an
element-wise constrained minimization problem like in Lemma~\ref{lemma_tetrahedron},
with carefully chosen boundary data.

For this argument to function, we need a suitable enumeration of the tetrahedra of the
patch, to pass in the right order. This is elaborated in Section~\ref{section_shelling}.
Then, we need to ensure that data we prescribe for the minimization problem in each
element $K_j$ are compatible as per Definition~\ref{def_compatible_data}. It turns
out that two arduous cases appear. First, the argument becomes subtle when
$K_j$ is the last element closing a loop around an edge $e$ of the patch. Similarly,
the last element $K_{|\CTa|}$ of the patch must be carefully addressed.
Section~\ref{section_two_color} and~\ref{section_three_color} provide intermediate
results to treat these two cases.

\subsection{Enumeration of the elements in the patch}
\label{section_shelling}

For $K \in \CTa$, we denote by $\CF_K^{\rm int}$ the set of faces of $K$ sharing the vertex $\ba$.
If $e$ is an edge of the patch having $\ba$ as a vertex, we denote by $\CTe \subset \CTa$ the
edge patch of elements sharing the edge $e$ and by $\ome$ the associated open subdomain.

We call an enumeration of the patch $\CTa$ an ordering of its elements $K_1,\dots,K_{|\CTa|}$.
For such an enumeration, for $1 \leq j \leq |\CTa|$, we denote by
$\CF_j^{\sharp} \subset \CF_{K_j}^{\rm int}$ the set of faces of $K_j$ shared
with an already enumerated element $K_i$ with $i < j$, and we set
$\CF_j^\flat \eq \CF_{K_j}^{\rm int} \setminus \CF_j^\sharp$.
The result in~\cite[Lemma~B.1]{ern_vohralik_2020a} provides us with a suitable
enumeration featuring the key properties listed below and illustrated in Figure~\ref{fig_ex_numbering}.

\begin{figure}
\centerline{\includegraphics[width=0.4\linewidth]{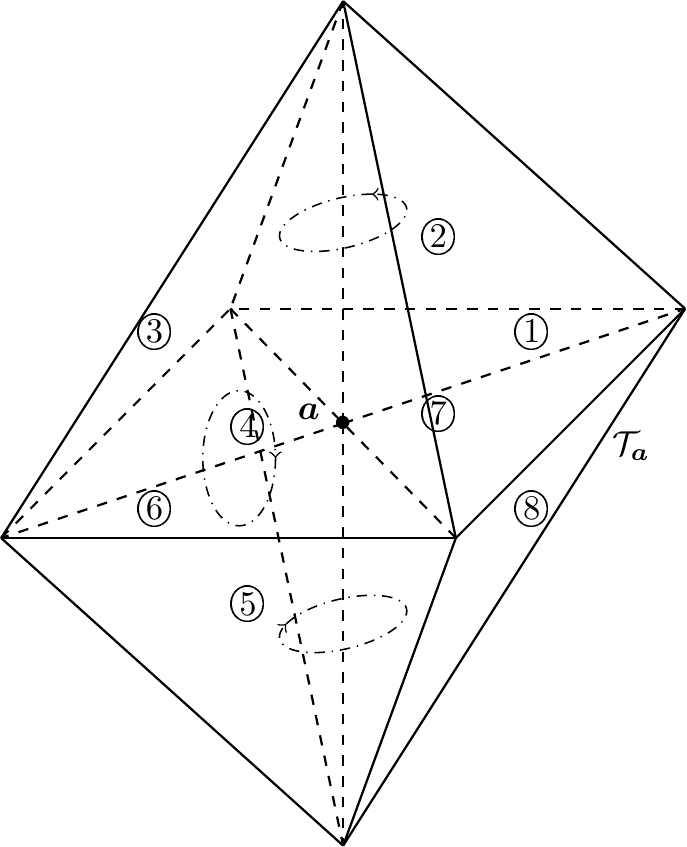} \qquad \includegraphics[width=0.4\linewidth]{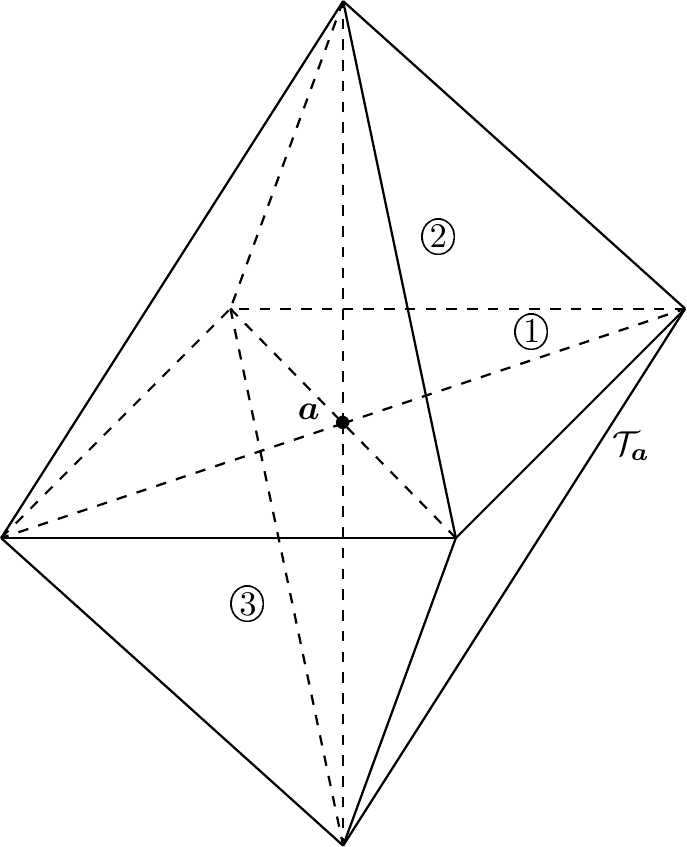}}
\caption{Patch enumeration of Proposition~\ref{prop_enumeration}, where ``loops'' around edges give property~\ref{it_enum_i} and where all elements except $K_1$ and $K_8$ have both already enumerated and not yet enumerated neighbors, i.e., property~\ref{it_enum_ii} (left). Invalid enumeration (right).}
\label{fig_ex_numbering}
\end{figure}

\begin{proposition}[Patch enumeration]
\label{prop_enumeration}
There exists an enumeration $\{K_1,\dots,K_{|\CTa|}\}$ of the vertex patch $\CTa$ such that:
\begin{enumerate}[label=(\roman*)]
\item For $1 < j \leq |\CTa|$, if there are at least two faces in $\CF_j^\sharp$ intersecting
in an edge, then all the elements sharing this edge come sooner in the enumeration, i.e.,
if $|\CF_j^\sharp| \geq 2$ with $F^1,F^2 \in \CF_j^\sharp$, then letting $e \eq F^1 \cap F^2$,
$K_i \in \CTe \setminus \{K_j\}$ implies that $i < j$. \label{it_enum_i}
\item For all $1 < j < |\CTa|$, there are one or two neighbors of $K_j$ which have been
already enumerated and correspondingly two or one neighbors of $K_j$ which have not been
enumerated yet, i.e., $|\CF_j^\sharp| \in \{1,2\}$ (so that
$|\CF_j^\flat| = 3 - |\CF_j^\sharp| \in \{1,2\}$ as well) for all but the first and the
last element. In particular, $\CF_j^\sharp$ is empty if and only if $j = 1$ and
$\CF_j^\sharp$ contains all the interior faces of $K_j$
(so that $\CF_j^\flat$ is empty) if and only if $j = |\CTa|$. \label{it_enum_ii}
\end{enumerate}
\end{proposition}

\subsection{Two-color refinement of edge patches}
\label{section_two_color}

This section recalls the following useful result to deal with the last
element of an edge patch of~\cite[Lemma~B.2]{ern_vohralik_2020a},
illustrated in Figure~\ref{fig_two_three_color}, left.

\begin{figure}
\centerline{\includegraphics[width=0.265\linewidth]{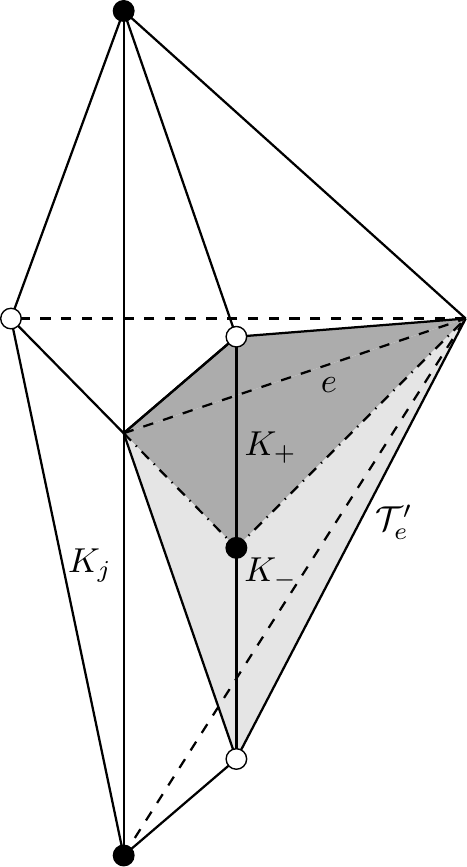} \qquad \includegraphics[width=0.4\linewidth]{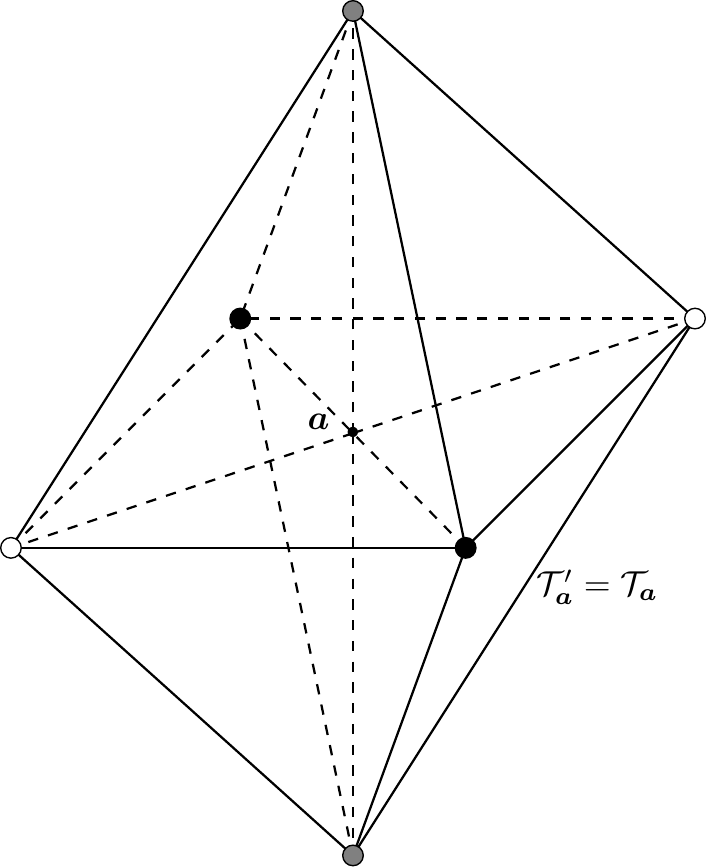}}
\caption{Two-color refinement (black and white) $\CTe'$ around an edge $e$ of Proposition~\ref{prop_two_color} (one of the tetrahedra in $\CTe$, different from $K_j$, is cut into $K_{+}$ and $K_{-}$) (left). Three-color refinement (black, grey, and white) around a vertex $\ba$ of Proposition~\ref{prop_three_color} (trivial situation where $\CTa'$ can be taken as $\CTa$) (right).}
\label{fig_two_three_color}
\end{figure}

\begin{proposition}[Two-color refinement around edges]
\label{prop_two_color}
Fix a tetrahedron $K_j \in \CTa$ and an edge $e$ of $K_j$ having $\ba$ as
one endpoint.
Then there exists a conforming refinement $\CTe'$ of $\CTe$ composed
of tetrahedra such that
\begin{itemize}
\item[(i)]
$\CTe'$ contains $K_j$.
\item[(ii)]
All the tetrahedra in $\CTe'$ have $e$ as an edge, and their two other vertices lie on $\partial \oma$.
\item[(iii)] There holds $\kappa_{\CTe'} \leq 2 \kappa_{\CTe}$.
\item[(iv)]
Collecting all the vertices of $\CTe'$ that are not endpoints of $e$ in the set $\CVe'$,
there is a two-color map $\col: \CVe' \to \{1,2\}$ so that for all $\kappa \in \CTe'$,
the two vertices of $\kappa$ that are not endpoints of $e$, say
$\{\ba_\kappa^n\}_{1 \leq n \leq 2}$, satisfy $\col(\ba_\kappa^n) = n$.
\end{itemize}
\end{proposition}

Above, $\CTe'$ can be taken as $\CTe$ when the number of tetrahedra in $\CTe$ is
even. When the number of tetrahedra in $\CTe$ is odd, it is enough to cut one of
the tetrahedra in $\CTe$, different from $K_j$, into two tetrahedra still sharing
the edge $e$. This is illustrated in Figure~\ref{fig_two_three_color}, left, with
the two tetrahedra $K_{+}$ and $K_{-}$ in dark and light grey, respectively.

\subsection{Three-color refinement of vertex patches}
\label{section_three_color}

Here, we present the following technical result to address the last element of the
vertex patch from~\cite[Lemma~B.3]{ern_vohralik_2020a}, illustrated in
Figure~\ref{fig_two_three_color}, right.

\begin{proposition}[Three-color patch refinement]
\label{prop_three_color}
Fix a tetrahedron $K_j \in \CTa$. There exists a conforming refinement
$\CTa'$ of $\CTa$ composed of tetrahedra such that
\begin{itemize}
\item[(i)]
$\CTa'$ contains $K_j$.
\item[(ii)]
All the tetrahedra in $\CTa'$ have $\ba$ as a vertex, and their three other vertices
lie on $\partial \oma$.
\item[(iii)] There holds $\kappa_{\CTa'} \leq C(\kappa_{\CTa})$.
\item[(iv)]
Collecting all the vertices of $\CTa'$ distinct from $\ba$ in the set $\CVa'$, there is
a three-color map $\col: \CVa' \to \{1,2,3\}$ so that for all $\kappa \in \CTa'$,
the three vertices of $\kappa$ distinct from $\ba$, say $\{\ba_\kappa^n\}_{1 \leq n \leq 3}$,
satisfy $\col(\ba_\kappa^n) = n$.
\end{itemize}
\end{proposition}

\subsection{Proof of Theorem~\ref{thm_Hc_cons} for interior patches} \label{sec_main_proof}

We are now ready to prove Theorem~\ref{thm_Hc_cons} for interior patches.

\begin{proof}[Proof of Theorem~\ref{thm_Hc_cons} for interior patches]
Denote by
\begin{equation*}
\bv^\star \eq
\arg \min_{\substack{
\bv \in \BH_{0,\Gamma}(\ccurl,\oma) \\ \curl \bv = \bj_p
}}
\|\bv - \bchi_p\|_{\oma}
\end{equation*}
the continuous minimizer.

We rely on the enumeration $K_j$, $1 \leq j \leq |\CTa|$, from Proposition~\ref{prop_enumeration},
see Figure~\ref{fig_ex_numbering} for illustration.
Following~\cite{braess_pillwein_schoberl_2009a,chaumontfrelet_ern_vohralik_2021a,ern_vohralik_2020a},
we construct an admissible $\bxi_p$ from the discrete minimization set
$\ND_p(\CTa) \cap \BH_{0,\Gamma}(\ccurl,\oma)$
by sequential element-wise minimizations following this enumeration.
Specifically, for each element $K_j$, $1 \leq j \leq |\CTa|$, we define
$F_j^{\rm ext} \eq \partial K_j \cap \partial \oma$ and the set of faces
$\CF_j \eq \{F_j^{\rm ext}\} \cup \CF_j^\sharp$ consisting of the face
$F_j^{\rm ext}$ on the patch boundary and of the faces of $K_j$ with neighbors that come sooner
in the enumeration, with a smaller index. We also denote the local volume data by
$\bj_p^j \eq \bj_p|_{K_j} \in \RT_p(K_j)$ and $\bchi_p^j \eq \bchi_p|_{K_j} \in \ND_p(K_j)$.
We will then iteratively define a boundary datum $\br_p^j$, see Figure~\ref{fig_ex_minimizations}
for illustration and Step~1 below for details, leading to the definition
\begin{equation}
\label{tmp_definition_bxipj}
\bxi_p|_{K_j}
\eq
\bxi_p^j
\eq
\arg \min_{\substack{
\bv_p \in \ND_p(K_j)
\\
\curl \bv_p = \bj_p^j
\\
\bv_p|^{\btau}_{\CF_j} = \br_p^j
}}
\|\bv_p - \bchi_p^j\|_{K_j}.
\end{equation}

\begin{figure}
\centerline{\includegraphics[width=0.5\linewidth]{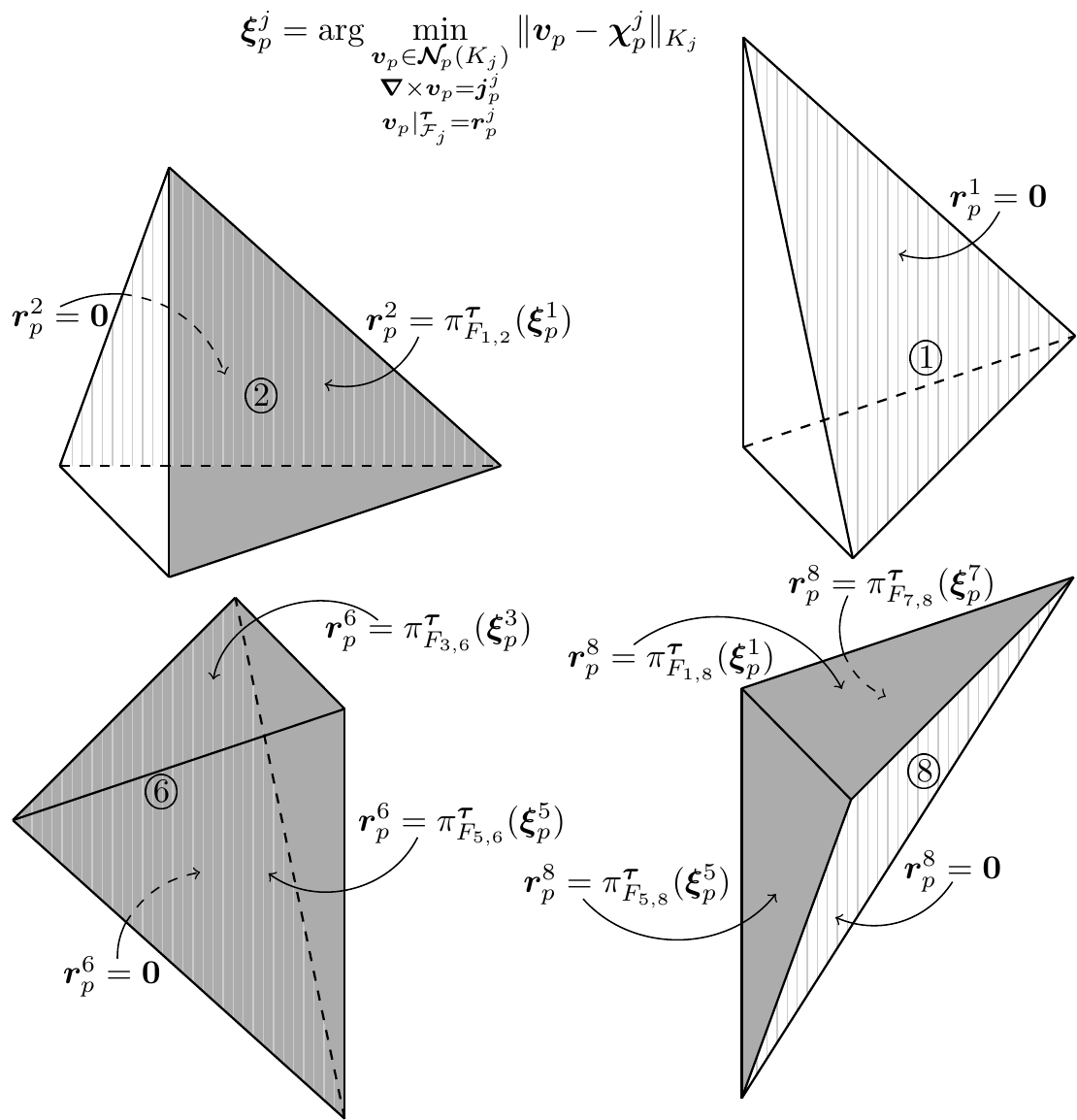}}
\caption{Minimizations~\eqref{tmp_definition_bxipj} on the elements $K_j$, $j=1,2,6,8$, of the patch $\CTa$ as shown and enumerated in Figure~\ref{fig_ex_numbering}, left. The boundary datum $\br_p^j$ is zero on the external faces $F_j^{\rm ext} \subset \pt \oma$ (this is always imposed) (faces $F_j^{\rm ext}$ hatched) and induced by the tangential trace of $\bxi_p^i$  on the previously enumerated elements $K_i$ as defined in Step~1 (there is no such datum on $K_1$ and respectively 1,2, and 3 on $K_2$, $K_6$, and $K_8$) (interfaces in grey).}
\label{fig_ex_minimizations}
\end{figure}

We prove, by induction, in Step~2 below, that at each step $j$, the data
$\bj_p^j$ and $\br_p^j$ are admissible in the sense of Definition~\ref{def_compatible_data},
with in particular $\br_p^j\in \ND_p(\Gamma_{\CF_j})$. Thus, the problems
in~\eqref{tmp_definition_bxipj} are well-posed. Then, in Step~3, we prove that
\begin{equation}
\label{tmp_stability_bxipj}
\|\bxi_p^j-\bchi_p^j\|_{K_j}
\leq
C(\kappa_{\CTa})
\|\bv^\star-\bchi_p\|_{\oma}.
\end{equation}
Finally, in Step~4, relying on the boundary data $\br_p^j$, we will establish that
$\bxi_p \in \ND_p(\CTa) \cap \BH_{0,\Gamma}(\ccurl,\oma)$,
showing that $\bxi_p$ belongs to the discrete minimization set in
Theorem~\ref{thm_Hc_cons}. This will conclude the proof since then
\begin{multline*}
\min_{\substack{
\bv_p \in \ND_p(\CTa) \cap \BH_{0,\Gamma}(\ccurl,\oma)
\\
\curl \bv_p = \bj_p
}}
\|\bv_p - \bchi_p\|_{\oma}
\leq
\|\bxi_p - \bchi_p\|_{\oma}
=
\Bigg\{\sum_{j=1}^{|\CTa|}
\|\bxi_p^j - \bchi_p^j\|_{K_j}^2\Bigg \}^{\frac 1 2}
\\
\leq
C(\kappa_{\CTa}) |\CTa|^{\frac 1 2}
\|\bv^\star - \bchi_p\|_{\oma}
=
C(\kappa_{\CTa}) |\CTa|^{\frac 1 2}
\min_{\substack{
\bv \in \BH_{0,\Gamma}(\ccurl,\oma)
\\
\curl \bv = \bj_p
}}
\|\bv - \bchi_p\|_{\oma},
\end{multline*}
and the number $|\CTa|$ of tetrahedra in the patch $\CTa$ is bounded by constant
only depending on the shape-regularity parameter $\kappa_{\CTa}$.

{\bf Step~1.}
We start by defining the boundary data $\br_p^j$ used in~\eqref{tmp_definition_bxipj}.
We let (i) $\br_p^j|_{F_j^{\rm ext}} \eq \bzero$ on the external face $F_j^{\rm ext}$;
and (ii) on each face $F_{i,j} \in \CF_j^\sharp$ shared by $K_j$ and $K_i$, $i<j$,
we set $\br_p^j|_{F_{i,j}} \eq \pi^{\btau}_{F_{i,j}}(\bxi_p^i)$, which we
can do since $\bxi_p^i$ is already defined on the simplices $K_i$ with a smaller index $i$.
This is illustrated in Figure~\ref{fig_ex_minimizations}.

{\bf Step~2.} We now verify that the data constructed above are admissible as per
Definition~\ref{def_compatible_data}, so that problem~\eqref{tmp_definition_bxipj} is well-posed.
Notice that since $\div \bj_p^j = \div (\bj|_{K_j}) = 0$, \eqref{eq_compatible_data_div} is
satisfied by construction. Considering~\eqref{eq_compatible_data_curl}, we have
$\CF_j = \{F_j^{\rm ext}\} \cup \CF_j^\sharp$.
For the exterior face $F_j^{\rm ext}$, the associated data $\br_p^j|_{F_j^{\rm ext}} = \bzero$
always vanishes, and
$\bj_p {\cdot} \bn_{F_j^{\rm ext}} = \scurl_{F_j^{\rm ext}} \br_p^j|_{F_j^{\rm ext}} =0 $
holds true since $\bj_p \in \BH_0(\ddiv,\oma) \cap \RT_p(\CTa)$ by assumption.
According to the enumeration from Proposition~\ref{prop_enumeration}, $\CF_j^\sharp$
is empty on the first element $K_1$, so there is nothing more to verify
for~\eqref{eq_compatible_data_curl} when $j=1$. On the other hand, when $j>1$, the
remaining faces in $\CF_j^\sharp $ are of the form $F_{i,j} = \partial K_i \cap \partial K_j$,
where $K_i$ has been previously visited, $i<j$. We then have
\begin{equation*}
\scurl_{F_{i,j}}(\br_p^j|_{F_{i,j}})
=
\scurl_{F_{i,j}} (\pi_{F_{i,j}}^{\btau}(\bxi_p^i))
=
(\curl \bxi_p^i)|_{F_{i,j}} {\cdot} \bn_{F_{i,j}}
=
\bj_p^i {\cdot} \bn_{F_{i,j}}
=
\bj_p^j {\cdot} \bn_{F_{i,j}}
\end{equation*}
since $\bj_p \in \BH(\ddiv,\oma) \cap \RT_p(\CTa)$ and since $\curl \bxi_p^i = \bj_p^i$
on $K_i$ by induction. As a result, we are left to check~\eqref{eq_compatible_data_trace}.
To do so, following~\eqref{eq_ttrace_Gamma}, we need to find $\BR_p^j \in \ND_p(K_j)$
such that $\br_p^j|_F = \pi^{\btau}_F (\BR_p^j)$ for all faces $F \in \CF_j$.
We distinguish 4 subcases for this purpose.

{\bf Step~2a.} In the first element $K_1$, we have $\CF_1 = \{ F_1^{\rm ext} \}$
and $\br_p^1|_{F_1^{\rm ext}} = \bzero$. It is clear that
$\br_p^1|_{F_1^{\rm ext}} = \pi_{F_1^{\rm ext}}^{\btau}(\bzero)$
which shows~\eqref{eq_compatible_data_trace} for $\BR_p^1 \eq \bzero$.

{\bf Step~2b.} We then consider the case where the element $K_j$, $1 < j < |\CTa|$, is
such that $|\CF_j^\sharp| = 1$, i.e., there is a single element $K_i$ with $i < j$
such that $\CF_j^\sharp = \{F_{i,j}\}$, $F_{i,j} = \partial K_i \cap \partial K_j$.
There exists a unique affine mapping $\psi_{i,j}: K_i \to K_j$
that leaves the face $F_{i,j}$ invariant, and we set
$\BR_p^j \eq \psi_{i,j}^{\rm c}(\bxi_p^i) \in \ND_p(K_j)$.
Since the Piola mapping preserves tangential traces, maps $F_i^{\rm ext}$
onto $F_j^{\rm ext}$, and leaves $F_{i,j}$ invariant, we clearly have
$\pi^{\btau}_{F_{i,j}} (\BR_p^j) = \pi^{\btau}_{F_{i,j}} (\bxi_p^i) = \br_p^j|_{F_{i,j}}$
and $\pi^{\btau}_{F_j^{\rm ext}} (\BR_p^j) = \bzero$ since
$\pi^{\btau}_{F_i^{\rm ext}}(\bxi_p^i) = \bzero$.
This shows that $\br_p^j|_F = \pi^{\btau}_F (\BR_p^j)$ for all $F \in \CF_j$,
so that~\eqref{eq_compatible_data_trace} is satisfied in view of
definition~\eqref{eq_ttrace_Gamma}.

{\bf Step~2c.}
The next case is an element $K_j$ with $1 < j < |\CTa|$ such that $|\CF_j^\sharp| = 2$.
We will use an argument similar to the one above in Step~2b, relying this time on Piola
mappings from all tetrahedra sharing the edge $e$ common to the two faces in $\CF_j^\sharp$.
First, since $|\CF_j^\sharp| = 2$, Proposition~\ref{prop_enumeration} implies that $K_j$ is
the lastly enumerated element of the edge patch $\CTe$. Invoking Proposition~\ref{prop_two_color},
there is a refined edge patch $\CTe'$ such that $\CTe' = \CTe$ if the number of tetrahedra in
$\CTe$ is even, whereas one of the tetrahedra in $\CTe$, different from $K_j$,
has been cut into two if $|\CTe|$ is odd. In any case, $\CTe' = \{\kappa_1,\dots,\kappa_n\}$,
$\kappa_n = K_j$, and the vertices of $\CTe'$ that are not endpoints of the edge $e$ are
colored by two colors (alternating along the numbering $1,\dots,n$).

Let $\psi_{\ell,n}: \kappa_\ell \to \kappa_n$ be the unique invertible affine mapping of the
tetrahedron $\kappa_\ell$ to the tetrahedron $\kappa_n$ preserving the two endpoints of the
edge $e$ and the colors of the two other vertices; this in particular means that the faces
$F_\ell^{\rm ext}$ are mapped to $F_j^{\rm ext}$, the two faces in $\CF_j^\sharp$ are left
invariant, and the other faces sharing the edge $e$ have their remaining vertex mapped while
preserving its color.
Denote by $\varepsilon_{\ell,n}$ the sign of determinant of the Jacobian of $\psi_{\ell,n}$.
Let finally, for $1 \leq \ell \leq n-1$, $\iota(\ell)$ be the index of the element
$K_{\iota(\ell)} \in \CTe$ such that $\kappa_\ell \subset K_{\iota(\ell)}$ (if the number
of tetrahedra in $\CTe$ is even, we can actually always write $\kappa_\ell = K_{\iota(\ell)}$;
if not, a strict inclusion only holds on the two subsimplices of the simplex that has been cut).
This allows for the following ``folding'' Piola mappings definition:
\begin{equation}
\label{eq_R_c}
\BR_p^j \eq - \sum_{\ell=1}^{n-1}
\varepsilon_{\ell,n} \psi_{\ell,n}^{\rm c}(\bxi_p^{\iota(\ell)}|_{\kappa_\ell})
\in \ND_p(K_j).
\end{equation}
As $K_j$ is the last element of the edge patch $\CTe$, for all $1 \leq \ell \leq n-1$,
$\bxi_p^{\iota(\ell)}$ have been previously defined, and this in such a way
that (i) their tangential traces vanish on $\partial \oma$; and (ii) their tangential
traces match on faces shared by two previously enumerated elements.
Now, since the faces in $F_\ell^{\rm ext} \subset \partial \oma$ are mapped to $F_j^{\rm ext}$,
$\pi^{\btau}_{F_j^{\rm ext}} (\BR_p^j) = \bzero$ follows from
$\pi^{\btau}_{F_\ell^{\rm ext}} (\bxi_p^{\iota(\ell)}|_{\kappa_\ell}) = \bzero$.
Similarly, all the faces sharing the edge $e$ other than the two faces in $\CF_j^\sharp$
are mapped twice, with two opposite signs in view of $\varepsilon_{\ell,n}$ (indeed,
$\varepsilon_{\ell_{-},n} + \varepsilon_{\ell_+,n} =0$ if the two elements
$\kappa_{\ell_{-}}$ and $\kappa_{\ell_+}$ from $\CTe'$ share a common face),
leaving only the contributions from the neighbors from the two faces in $\CF_j^\sharp$.
Thus $\pi^{\btau}_{F_{i,j}} (\BR_p^j)= \pi^{\btau}_{F_{i,j}}(\bxi_p^i) = \br_p^j|_{F_{i,j}}$
for the (two) faces $F_{i,j} \in \CF_j^\sharp$, and~\eqref{eq_compatible_data_trace} is satisfied.

{\bf Step~2d.} We finish with the last element $K_j$, $j = |\CTa|$.
In this case we have $|\CF_j^\sharp| = 3$. In extension of Step~2c,
we rely on Piola mappings from all the tetrahedra of the patch $\CTa$
other than $K_j$. Following Proposition~\ref{prop_three_color}, we invoke
for this purpose a three-color patch refinement $\CTa'$ such that
$\CTa' = \{\kappa_1,\dots,\kappa_n\}$, $\kappa_n = K_j$.

Let $\psi_{\ell,n}: \kappa_\ell \to \kappa_n$ be the unique invertible affine mapping of
the tetrahedron $\kappa_\ell$ to the tetrahedron $\kappa_n$ preserving the vertex $\ba$
and the colors of the three other vertices; this in particular means that the faces
$F_\ell^{\rm ext}$ are mapped to $F_j^{\rm ext}$ and the other faces have their vertices
mapped while preserving their color.
Denote by $\varepsilon_{\ell,n}$ the sign of determinant of the Jacobian of $\psi_{\ell,n}$.
Let finally, for $1 \leq \ell \leq n-1$, $\iota(\ell)$ be the index of the element
$K_{\iota(\ell)} \in \CTa$ such that $\kappa_\ell \subset K_{\iota(\ell)}$.
This allows for the following ``folding'' Piola mappings definition:
\begin{equation}
\label{eq_R_d}
\BR_p^j \eq -\sum_{\ell=1}^{n-1}
\varepsilon_{\ell,n} \psi_{\ell,n}^{\rm c}(\bxi_p^{\iota(\ell)}|_{\kappa_\ell})
\in \ND_p(K_j).
\end{equation}
As above in Step~2c, we observe that (i) all $\bxi_p^{\iota(\ell)}$ have been
previously defined and have a zero/matching tangential trace; (ii) each boundary
face of $\CTa'$ (except of $F_j^{\rm ext}$) is mapped to $F_j^{\rm ext}$; (iii) each
interior face of $\CTa'$ other than the three faces from $\CF_j^\sharp$ is mapped twice,
each time with an opposite sign; and (iv) the three faces from $\CF_j^\sharp$ are only
mapped once. This yields $\pi^{\btau}_{F_j^{\rm ext}} (\BR_p^j) = \bzero$
together with
$\pi^{\btau}_{F_{i,j}} (\BR_p^j)= \pi^{\btau}_{F_{i,j}}(\bxi_p^i) = \br_p^j|_{F_{i,j}}$
for the (three) faces $F_{i,j} \in \CF_j^\sharp$, so that~\eqref{eq_compatible_data_trace} follows.

{\bf Step~3.}
We now show~\eqref{tmp_stability_bxipj}, that is, at each step $1 \leq j \leq |\CTa|$,
the element $\bxi^j_p$ given by~\eqref{tmp_definition_bxipj} is stable as compared to the
continuous minimizer $\bv^\star$. Let
\begin{equation}
\label{eq_VKj}
\BV(K_j)
\eq
\left \{
\bv \in \BH(\ccurl,K_j) \; | \;
\curl \bv = \bj_p^j, \;
\bv|_{\CF_j}^{\btau} = \br_p^j
\right \}.
\end{equation}
From Step~2, we know that this set is nonempty. To show~\eqref{tmp_stability_bxipj},
we will construct for every $1 \leq j \leq |\CTa|$ an element $\bw^\star_j \in \BV(K_j)$ such that
\begin{equation}
\label{tmp_stability_wstar}
\|\bw^\star_j-\bchi_p^j\|_{K_j}
\leq
C(\kappa_{\CTa}) \|\bv^\star-\bchi_p\|_{\oma}.
\end{equation}
Estimate~\eqref{tmp_stability_bxipj} then follows from Lemma~\ref{lemma_tetrahedron} since
\begin{align*}
\|\bxi_p^j-\bchi_p^j\|_{K_j}
&=
\min_{\bw_p \in \BV(K_j) \cap \ND_p(K_j)}
\|\bw_p-\bchi_p^j\|_{K_j}
\\
&\leq
C(\kappa_{K_j})
\min_{\bw \in \BV(K_j)}
\|\bw-\bchi_p^j\|_{K_j}
\leq
C(\kappa_{K_j})
\|\bw^\star_j-\bchi_p^j\|_{K_j}.
\end{align*}

{\bf Step~3a.} In the first element $K_1$, we actually readily observe that
$\bw^\star_j \eq \bv^\star|_{K_1}$ belongs to the minimization set $\BV(K_1)$,
so that~\eqref{tmp_stability_wstar} is immediately satisfied with the constant
$C(\kappa_{\CTa})=1$.

{\bf Step~3b.} We next consider those elements $K_j$, $1 < j < |\CTa|$, for which
$|\CF_j^\sharp| = 1$, and we denote by $1 \leq i < j$ the index such that
$\CF_j^\sharp = \{F_{i,j}\}$ with $F_{i,j} = \partial K_i \cap \partial K_j$.
As in Step~2b, we consider the affine map $\psi_{i,j}: K_i \to K_j$
that leaves the face $F_{i,j}$ invariant, and we set
\begin{equation}
\label{eq_bw_star}
\bw^\star_j \eq \bv^\star|_{K_j} - \psi_{i,j}^{\rm c}(\bv^\star|_{K_i}-\bxi_p^i).
\end{equation}

We now show that $\bw^\star_j$ belongs to $\BV(K_j)$ given by~\eqref{eq_VKj}.
First, by the Piola mapping, $\bw^\star_j \in \BH(\ccurl,K_j)$. Moreover,
recalling~\eqref{eq_piola_commute}, it is clear that
\begin{equation*}
\curl \bw^\star_j
=
\curl (\bv^\star|_{K_j}) - \psi_{ij}^{\rm d}(\curl (\bv^\star|_{K_i}-\bxi_p^i))
=
\curl (\bv^\star|_{K_j})
=
\bj_p^j.
\end{equation*}
Finally, roughly speaking, the fact that $\bw^\star_j|_{\CF_j}^{\btau} = \br_p^j$
follows from~\eqref{eq_bw_star} since all $\bv^\star|_{K_j}$, $\bv^\star|_{K_i}$, and $\bxi_p^i$
have a zero tangential trace on $\partial \oma$ and the tangential trace of $\bv^\star$
is continuous across $F_{i,j}$, so that its contribution vanishes in $\bw^\star_j$ and only
the desired contribution from $\bxi_p^i$ is left. However, in contrast to Step~2b carried
out for piecewise polynomials, we cannot rigorously prove this in this strong sense
because we cannot easily localize the notion of tangential trace to one face for
$\BH(\ccurl)$ functions. As a result, we have to resort to the weak
notion of tangential trace introduced in~\eqref{eq_weak_trace}.
For this purpose, we first note that, following Step~2b,
\begin{equation*}
\bw^\star_{j} = \bv^\star|_{K_j} - \psi_{i,j}^{\rm c}(\bv^\star|_{K_i}) + \BR_p^{j},
\end{equation*}
with $\BR_p^{j}|_{\CF_j}^{\btau} = \br_p^j$. Thus, we need to show that
$(\bv^\star|_{K_j} - \psi_{i,j}^{\rm c}(\bv^\star|_{K_i}))|_{\CF_j}^{\btau} = \bzero$.
Recall that $\CF_j = \{F_j^{\rm ext}\} \cup \CF_j^\sharp = \{F_j^{\rm ext},F_{i,j}\}$.
Following~\eqref{eq_weak_trace}, let $\bphi \in \BH^1_{\btau,\CF_j^{\rm c}}(K_j)$.
Letting $\psi_{j,i} = (\psi_{i,j})^{-1}$, the function
\begin{equation}
\label{eq_tphi}
\widetilde \bphi|_{K_j} \eq \bphi,
\qquad
\widetilde \bphi|_{K_i} \eq \psi_{j,i}^{\rm c}(\bphi)
\end{equation}
belongs to $\BH_{0,\partial(K_i \cup K_j) \setminus \partial \oma}(\ccurl,K_i \cup K_j)$.
Then, noticing that the sign of the determinant of the Jacobian of $\psi_{i,j}$ is negative,
\eqref{eq_piola_v_curlw} allows us to write
\begin{align*}
& {} (\curl (\bv^\star|_{K_j} - \psi_{i,j}^{\rm c}(\bv^\star|_{K_i})),\bphi)_{K_j}
-
(\bv^\star|_{K_j} - \psi_{i,j}^{\rm c}(\bv^\star|_{K_i}),\curl \bphi)_{K_j} \\
= {} &
(\curl \bv^\star,\bphi)_{K_j}
-
(\bv^\star,\curl \bphi)_{K_j}
+
(\curl \bv^\star,\psi_{j,i}^{\rm c}(\bphi))_{K_i}
-
(\bv^\star,\curl \psi_{j,i}^{\rm c}(\bphi))_{K_i}
\\
= {} &
(\curl \bv^\star,\widetilde \bphi)_{K_i \cup K_j}
-
(\bv^\star,\curl \widetilde \bphi)_{K_i \cup K_j} =0,
\end{align*}
since $\bv^\star|_{K_i \cup K_j} \in \BH_{0,\partial(K_i \cup K_j) \cap \partial \oma}(\ccurl,K_i \cup K_j)$.

Finally, we have
\begin{equation*}
\bw^\star_j-\bchi_p^j = (\bv^\star|_{K_j}-\bchi_p^j)
-\psi_{i,j}^{\rm c}(\bv^\star|_{K_i}-\bchi_p^i)
+\psi_{i,j}^{\rm c}(\bxi_p^i-\bchi_p^i),
\end{equation*}
and recalling~\eqref{eq_piola_stable}, it follows that
\begin{equation*}
\|\bw^\star_j-\bchi_p^j\|_{K_j}
\leq
\|\bv^\star-\bchi_p^j\|_{K_j}
+ C(\kappa_{\CTa})\big(
\|\bv^\star-\bchi_p^i\|_{K_i}
+
\|\bxi_p^i-\bchi_p^i\|_{K_i}\big)
\leq
C(\kappa_{\CTa})
\|\bv^\star-\bchi_p\|_{\oma},
\end{equation*}
since~\eqref{tmp_stability_bxipj} holds in $K_i$ by induction.
Thus~\eqref{tmp_stability_wstar} holds.

{\bf Step~3c.}
The next situation is the case of an element $K_j$, $1 < j < |\CTa|$, with $|\CF_j^\sharp| = 2$.
We keep the notation of Step~2c for the two-color refinement
$\CTe'$ of the patch around the edge $e$ and the associated affine mappings $\psi_{\ell,n}$.
We set, in extension of~\eqref{eq_bw_star} from the previous Step~3b and following the
recipe~\eqref{eq_R_c},
\begin{equation}
\label{eq_bw_star_c}
\bw^\star_j
\eq
\bv^\star|_{K_j}
+
\sum_{\ell=1}^{n-1}
\varepsilon_{\ell,n}\psi^{\rm c}_{\ell,n}
(
\bv^\star|_{\kappa_\ell}-\bxi_p^{\iota(\ell)}|_{\kappa_\ell}
).
\end{equation}

We now again show that $\bw^\star_j \in \BV(K_j)$ given by~\eqref{eq_VKj}. First,
by the Piola mappings, $\bw^\star_{j} \in \BH(\ccurl,K_j)$. Moreover, since
$\curl(\bv^\star|_{K_j}) = \bj_p^j$ and
$\curl(\bv^\star|_{\kappa_\ell})
=
\curl(\bxi_p^{\iota(\ell)}|_{\kappa_\ell})
=
\bj_p^{\iota(\ell)}|_{\kappa_\ell}$,
it is clear that $\bw^\star_j$ satisfies the curl constraint of $\BV(K_j)$,
$\curl \bw^\star_j = \bj_p^j$. We then turn to the trace constraint
$\bw^\star_j|_{\CF_j}^{\btau} = \br_p^j$. We rewrite~\eqref{eq_bw_star_c}, using~\eqref{eq_R_c}, as
\begin{equation*}
\bw^\star_j
=
\sum_{\ell=1}^n \varepsilon_{\ell,n} \psi_{\ell,n}^{\rm c}(\bv^\star|_{\kappa_\ell})
-
\sum_{\ell=1}^{n-1} \varepsilon_{\ell,n} \psi_{\ell,n}^{\rm c}(\bxi_p^{\iota(\ell)}|_{\kappa_\ell})
=
\sum_{\ell=1}^n \varepsilon_{\ell,n} \psi_{\ell,n}^{\rm c}(\bv^\star|_{\kappa_\ell})
+
\BR_p^j,
\end{equation*}
with $\psi_{n,n}^{\rm c}$ identity and $\varepsilon_{n,n}=1$.
Since $\BR_p^{j}|_{\CF_j}^{\btau} = \br_p^j$ from Step~2c, we merely need to show that
$\big(\sum_{\ell=1}^n \varepsilon_{\ell,n} \psi_{\ell,n}^{\rm c}(\bv^\star|_{\kappa_\ell})\big)|_{\CF_j}^{\btau} = \bzero$.
Intuitively, this is rather clear; following the reasoning of Step 2c,
(i) all the faces $F_\ell^{\rm ext}$ are mapped to $F_j^{\rm ext}$, yielding a zero
tangential trace; (ii) all the faces sharing the edge $e$, including the two faces
in $\CF_j^\sharp$, are mapped twice with two opposite signs, yielding a zero tangential
trace. To show this rigorously, we again rely on the  characterization~\eqref{eq_weak_trace}.
Recalling that $\CF_j = \{F_j^{\rm ext}\} \cup \CF_j^\sharp$, consider thus
$\bphi \in \BH^1_{\btau,\CF_j^{\rm c}}(K_j)$. In extension of~\eqref{eq_tphi}, let us define
\begin{equation} \label{eq_tphi_c}
\widetilde \bphi|_{\kappa_\ell} \eq \psi_{n,\ell}^{\rm c}(\bphi) \qquad 1 \leq \ell \leq n.
\end{equation}
By the two-coloring of Proposition~\ref{prop_two_color}, as in Step~2c, this ``unfolding''
of $\bphi$ gives $\widetilde \bphi \in \BH_{0,\partial\ome \setminus \partial \oma}(\ccurl,\ome)$.
Then, using~\eqref{eq_piola_v_curlw}, we have
\begin{align*}
{} & \Bigg(
\curl \Bigg(
\sum_{\ell=1}^n \varepsilon_{\ell,n} \psi_{\ell,n}^{\rm c}(\bv^\star|_{\kappa_\ell})
\Bigg),
\bphi
\Bigg)_{K_j}
-
\Bigg(
\sum_{\ell=1}^n \varepsilon_{\ell,n} \psi_{\ell,n}^{\rm c}(\bv^\star|_{\kappa_\ell}),
\curl \bphi\Bigg)_{K_j} \\
= {} &
\sum_{\ell=1}^n
\left \{
(\curl \bv^\star,\psi_{n,\ell}^{\rm c}(\bphi))_{\kappa_\ell}
-
(\bv^\star,\curl (\psi_{n,\ell}^{\rm c}(\bphi)))_{\kappa_\ell}
\right \} \\
= {} & (\curl \bv^\star,\widetilde \bphi)_{\ome}
-
(\bv^\star,\curl \widetilde \bphi)_{\ome} = 0,
\end{align*}
since $\bv^\star|_{\ome} \in \BH_{0,\partial \ome \cap \partial \oma}(\ccurl,\ome)$,
so that indeed $\bw^\star_j \in \BV(K_j)$.

We conclude this step by showing that~\eqref{tmp_stability_wstar} holds true.
First, we write that
\begin{align*}
\bw^\star_j - \bchi_p^j
= {} &
\bv^\star|_{K_j}-\bchi_p^j
+
\sum_{\ell-1}^{n-1}
\varepsilon_{\ell,n} \psi_{\ell,n}^{\rm c}
(
\bv^\star|_{\kappa_\ell} - \bchi_p^{\iota(\ell)}|_{\kappa_\ell}
) \\
{} & +
\sum_{\ell-1}^{n-1}
\varepsilon_{\ell,n} \psi_{\ell,n}^{\rm c}
(
\bchi_p^{\iota(\ell)}|_{\kappa_\ell} - \bxi_p^{\iota(\ell)}|_{\kappa_\ell}
),
\end{align*}
and it follows that, recalling~\eqref{eq_piola_stable},
\begin{align*}
\|\bw^\star_j - \bchi_p^j\|_{K_j}
&\leq
\|\bv^\star-\bchi_p^j\|_{K_j}
+
\sum_{\ell-1}^{n-1}
\|\psi_{\ell,n}^{\rm c}(\bv^\star|_{\kappa_\ell} - \bchi_p^{\iota(\ell)}|_{\kappa_\ell})\|_{K_j}
+
\sum_{\ell-1}^{n-1}
\|\psi_{\ell,n}^{\rm c}(\bchi_p^{\iota(\ell)}|_{\kappa_\ell} - \bxi_p^{\iota(\ell)}|_{\kappa_\ell})\|_{K_j}
\\
&\leq
\|\bv^\star-\bchi_p^j\|_{K_j}
+
C(\kappa_{\CTe'})
\left (
\sum_{\ell-1}^{n-1}
\|\bv^\star - \bchi_p^{\iota(\ell)}\|_{\kappa_\ell}
+
\sum_{\ell-1}^{n-1}
\|\bchi_p^{\iota(\ell)} - \bxi_p^{\iota(\ell)}\|_{\kappa_\ell}
\right )
\\
&\leq
C(\kappa_{\CTe'})
\left (
\|\bv^\star-\bchi_p\|_{\oma}
+
\sum_{\ell-1}^{n-1}
\|\bchi_p^{\iota(\ell)} - \bxi_p^{\iota(\ell)}\|_{K_{\iota(\ell)}}
\right ).
\end{align*}
Then~\eqref{tmp_stability_wstar} follows since
$\kappa_{\CTe'} \leq 2 \kappa_{\CTe} \leq C(\kappa_{\CTa})$
and, by induction, \eqref{tmp_stability_bxipj} holds for for $i=\iota(\ell) < j$,
$1 \leq \ell \leq n-1$.

{\bf Step~3d.}
The proof for the last element is analogous to that of Step~3c.
We in particular still rely on~\eqref{eq_bw_star_c} and~\eqref{eq_tphi_c} where,
this time, the three-color patch refinement $\CTa' = \{\kappa_1,\dots,\kappa_n\}$,
$\kappa_n = K_j$, of Proposition~\ref{prop_three_color} is employed. Here,
$\widetilde \bphi \in \BH(\ccurl,\oma)$, whereas, since $\Gamma = \partial \oma$
for the considered interior patch case, $\bv^\star \in \BH_{0,\partial \oma}(\ccurl,\oma)$
by definition.
-
{\bf Step~4.}
We finally define $\bxi_p \in \ND_p(\CTa)$ by setting $\bxi_p|_{K_j} \eq \bxi_p^j$
for $1 \leq j \leq |\CTa|$ and verify that $\bxi_p \in \BH_{0,\Gamma}(\ccurl,\oma)$.
By construction, the tangential trace of each $\bxi_p^j$
vanishes on $F_j^{\rm ext}$, and if $F_{i,j}$ is the face shared by two tetrahedra
$K_i$ and $K_j$, the tangential traces of $\bxi_p^i$ and $\bxi_p^j$ match on $F_{i,j}$.
It follows that $\bxi_p \in \BH_{0,\Gamma}(\ccurl,\oma)$. Since, by construction,
$\curl (\bxi_p|_{K_j}) = \curl \bxi_p^j = \bj_p^j = \bj_p|_{K_j}$ for all
$K_j \in \CTa$, this means that that $\curl \bxi_p = \bj_p$ globally in $\oma$,
which concludes the proof.
\end{proof}

\section{Proof of Theorem~\ref{thm_Hc_cons} for boundary patches}
\label{section_boundary_patches}

In this section, we study the boundary patches.
We will work with different patches obtained by geometrical mappings,
some of those will be boundary and some interior in the terminology of Section~\ref{sec_VP}.
In addition to the notation $\Ga$ and $\Gamma$ therefrom
(recall Figures~\ref{fig_patches_1} and~\ref{fig_patches_2}),
we denote by $\Gaess$ the faces from the boundary of $\oma$ sharing the vertex $\ba$ and
belonging to $\Gamma$. Here, the tangential trace (essential boundary condition) is
imposed in Theorem~\ref{thm_Hc_cons}. $\Ga=\Ganat$ then collects the remaining faces
from the boundary of $\oma$ sharing $\ba$. Here, no the tangential trace (natural boundary
condition) is imposed in Theorem~\ref{thm_Hc_cons}. By the assumptions, $\Gaess$ and $\Ganat$
are both connected and have Lipschitz boundaries.

Let $\CTa$ be a vertex patch in the sense of Section~\ref{sec_VP},
interior or boundary. For a given $p \geq 0$ and $\bchi_p \in \ND_p(\CTa)$
and $\bj_p \in \RT_p(\CTa) \cap \BH_{0,\Gamma}(\ddiv,\oma)$ with $\div \bj_p  = 0$,
let
\begin{equation}
\label{eq_minimizers}
\bv_p^\star \eq \arg \min_{\substack{
\bv_p \in \ND_p(\CTa) \cap \BH_{0,\Gamma}(\ccurl,\oma)
\\
\curl \bv_p = \bj_p
}}
\|\bv_p-\bchi_p\|_{\oma}, \quad
\bv^\star \eq \arg \min_{\substack{
\bv \in \BH_{0,\Gamma}(\ccurl,\oma)
\\
\curl \bv = \bj_p
}}
\|\bv-\bchi_p\|_{\oma}
\end{equation}
be respectively the discrete and continuous minimizers from Theorem~\ref{thm_Hc_cons}.
We will consider here the best uniform constant $C_{{\rm st},p,\CTa,\Gamma}$
in the inequality
\begin{equation*}
\|\bv_p^\star-\bchi_p\|_{\oma}
\leq
C_{{\rm st},p,\CTa,\Gamma}
\|\bv^\star-\bchi_p\|_{\oma},
\end{equation*}
i.e.
\begin{equation}
\label{eq_stab}
C_{{\rm st},p,\CTa,\Gamma}
\eq
\sup_{\substack{
\bj_p \in \RT_p(\CTa) \cap \BH_{0,\Gamma}(\ddiv,\oma); \div \bj_p  = 0
\\
\bchi_p \in \ND_p(\CTa)
}}
\frac{\|\bv_p^\star-\bchi_p\|_{\oma}}{\|\bv^\star-\bchi_p\|_{\oma}}.
\end{equation}
For interior patches, we have shown in Section~\ref{section_interior_patches} that
$C_{{\rm st},p,\CTa,\Gamma}$ is uniformly bounded by a constant only dependent on
the patch shape-regularity parameter $\kappa_{\CTa}$. For boundary patches, it is
clear that $C_{{\rm st},p,\CTa,\Gamma}$ is bounded for each $p$, and our goal here
is to show that it is actually uniformly bounded in $p$, again by $C(\kappa_{\CTa})$
only.

\subsection{Plan of the proof}

Let us start by structurally describing how the proof is performed.
The central idea is to transform an arbitrary boundary patch $\CTa$, covering
a polyhedron $\oma$ in the full $\mathbb R^3$ space, into a reference tetrahedron
patch $\hCTo$ that covers the domain given by the right-angled reference tetrahedron
$\widehat \omega_{\bzero}$; this is part of the $\bx_1,\bx_2,\bx_3 \geq 0$ eight-space
such that $\bx_1 + \bx_2 + \bx_3 \leq 1$. The reference tetrahedron patch $\hCTo$ can
possibly have the same mesh topology/connectivity as $\CTa$. In this case, one can imagine 
that the transformation is doable by moving the vertices of the patch $\CTa$, which leads
to the notion of ``equivalent patches''. Establishing this rigorously is, however, an involved
task. To achieve it, we will rely on a graph-drawing result known as Tutte's embedding theorem
in graph theory~\cite{tutte_1963a}. In general, a more involved ``extension'' concept will be
necessary, which serves to prepare conditions in which the Tutte embedding theorem applies.
Once the patch $\CTa$ is transformed into a right-angled reference tetrahedron pach $\hCTo$,
we can use arguments based on mirror symmetries around the faces of $\hCTo$ sharing the vertex
$\bzero$ to further transform the boundary patch $\hCTo$ into an interior patch involving eight
copies of $\hCTo$ itself. Then, we can apply the stability result already established for the
interior patch in Section~\ref{section_interior_patches} and use it to establish the result
for $\hCTo$ and finally $\CTa$. The overall procedure considerably extends and generalizes
the approach in~\cite[Section~7]{ern_vohralik_2020a}, where only one mapping by a mirror
symmetry over a plane was employed to deduce the desired stability result for a boundary
patch form that of an interior patch.

\subsection{Equivalent patches}

As discussed, we will first need the concept of equivalent patches, which,
roughly speaking, corresponds to patches having the same mesh topology/connectivity.
Let $\CTa$ and $\CTb$ be two vertex patches around two possibly different vertices
$\ba$ and $\bb$ in the sense of Section~\ref{sec_VP}. $\CTa$ and $\CTb$ can be interior
or boundary.

\begin{definition}[Equivalent patches]
\label{definition_equiv}
Two vertex patches $\CTa$ and $\CTb$ around the vertices
$\ba$ and $\bb$ and covering the domains $\oma$ and $\omb$
are said to be equivalent if there exists a bilipschitz mapping
$\psi: \oma \to \omb$ such that $\psi|_K$ is affine and $\psi(K) \in \CTb$
for all $K \in \CTa$. Note that $\psi$ necessarily preserves the
topology/connectivity, i.e., $\bb = \psi(\ba)$, if a (boundary) face $F \in \CFa$
shares $\ba$, then $\psi(F) \in \CFb$ is a (boundary) face that shares $\bb$,
and if $K,L \in \CTa$ are neighbors over a face $F$, then $\psi(K), \psi(L) \in \CTb$
are neighbors over the face $\psi(F)$.
\end{definition}

The stability constants of equivalent patches are tightly linked together.
Actually, they simply differ up to a factor depending only on the shape
regularity parameter of the two patches.

\begin{lemma}[Equivalent patches]
\label{lemma_equiv}
If $\CTa$ and $\CTb$ are equivalent patches in the sense of Definition~\ref{definition_equiv},
then, for all $p \geq 0$,
\begin{equation}
\label{eq_stab_equiv}
C_{{\rm st},p,\CTa,\Gamma}
\leq
C(\kappa_{\CTa},\kappa_{\CTb})
C_{{\rm st},p,\CTb,\widetilde \Gamma},
\end{equation}
where $\widetilde \Gamma \eq \psi(\Gamma)$ and $\psi$ is the bilipschitz mapping of
Definition~\ref{definition_equiv}.
\end{lemma}

\begin{proof}
Fix a polynomial degree $p \geq 0$. Consider data
$\bj_p \in \RT_p(\CTa) \cap \BH_{0,\Gamma}(\ddiv,\oma)$
with $\div \bj_p = 0$ and $\bchi_p \in \ND_p(\CTa)$. We define
$\tbj_p \eq \psi^{\rm d}(\bj_p)$ and $\tbchi_p \eq \psi^{\rm c}(\bchi_p)$,
where $\psi^{\rm d}$ and $\psi^{\rm c}$ are the Piola mappings from~\eqref{eq_Piola}.
Because the mapping $\psi$ is Lipschitz and piecewise affine, we have
$\tbj_p \in \RT_p(\CTb) \cap \BH_{0,\widetilde \Gamma}(\ddiv,\omb)$
and $\tbchi_p \in \ND_p(\CTb)$. As a result, if we denote
by $\tbv^\star$ and $\tbv_p^\star$ the continuous and discrete
minimizers on $\widetilde \CT_{\bb}$ with data $\tbj_p$ and $\tbchi_p$,
we have
\begin{equation*}
\|\tbv_p^\star-\tbchi_p\|_{\omb}
\leq
C_{{\rm st},p,\CTb,\widetilde \Gamma}
\|\tbv^\star-\tbchi_p\|_{\omb}.
\end{equation*}
Then, letting $\enorm{{\cdot}}$ be the usual operator norm from $\BL^2(\oma)$
to $\BL^2(\omb)$ (or vice-versa), we have, on the one hand, since $\tbv^\star$
is the minimizer and $\psi^{\rm c}(\bv^\star) \in \BH_{0,\widetilde \Gamma}(\ccurl,\omb)$
with $\curl (\psi^{\rm c}(\bv^\star)) =\tbj_p$ that
\begin{equation*}
\|\tbv^\star-\tbchi_p\|_{\omb}
\leq
\|\psi^{\rm c}(\bv^\star)-\tbchi_p\|_{\omb}
=
\|\psi^{\rm c}(\bv^\star-\bchi_p)\|_{\omb}
\leq
\enorm{\psi^{\rm c}}\|\bv^\star-\bchi_p\|_{\oma},
\end{equation*}
and, on the other hand, that
\begin{equation*}
\|\bv_p^\star-\bchi_p\|_{\oma}
\leq
\|(\psi^{\rm c})^{-1}(\tbv_p^\star-\tbchi_p)\|_{\oma}
\leq
\enorm{(\psi^{\rm c})^{-1}}\|\tbv_p^\star-\tbchi_p\|_{\omb}.
\end{equation*}
It follows that
\begin{equation*}
\|\bv_p^\star-\bchi_p\|_{\oma}
\leq
C_{{\rm st},p,\CTb,\widetilde \Gamma}
\enorm{\psi^{\rm c}}
\enorm{(\psi^{\rm c})^{-1}}
\|\bv^\star-\bchi_p\|_{\oma}.
\end{equation*}
Since the data was arbitrary, \eqref{eq_stab_equiv} follows from the estimate
\begin{equation*}
\enorm{\psi^{\rm c}} \enorm{(\psi^{\rm c})^{-1}}
\leq
\bar \kappa_{\CTa}^4 \bar \kappa_{\CTb}^4
\end{equation*}
with
\begin{equation*}
\bar\kappa_{\CTa} \eq \frac{\max_{K \in \CTa} h_K}{\min_{K \in \CTa}\rho_K}, \quad
\bar\kappa_{\CTb} \eq \frac{\max_{K \in \CTb} h_K}{\min_{K \in \CTb}\rho_K}
\end{equation*}
that may be easily obtained from standard estimates on the Jacobian
matrices $\JAC$ defining the affine mappings (see, e.g., \cite[Theorem 3.1.2]{ciarlet_2002a}).
The conclusion then follows since $\bar\kappa_{\CTa} \leq C(\kappa_{\CTa})$
and $\bar\kappa_{\widetilde \CT_{\bb}} \leq C(\kappa_{\widetilde \CT_{\bb}})$.
\end{proof}

\subsection{Extensions of patches}

We will next need the concept of ``patch extension''. Specifically, if a patch
$\CTa$ can be extended into another patch $\tCTa \supset \CTa$ in a suitable
way, then the stability constant~$C_{{\rm st},p,\CTa,\Gamma}$ of $\CTa$
given by~\eqref{eq_stab} will be controlled by that of $\tCTa$. Here, $\CTa$
will typically be a boundary patch and $\tCTa$ either boundary or interior.
The precise definition is as follows.

\begin{definition}[Patch extension]
\label{definition_extension}
Consider two patches $\CTa$ and $\tCTa$ around the same vertex $\ba$,
with associated domains $\oma$ and $\toma$. We say that $\tCTa$ is an extension of $\CTa$ if
the following holds:
\begin{enumerate}
\item $\CTa \subset \tCTa$.
\item
There exist extension operators
$\LE^{\rm c},\LE^{\rm d}: \BL^2(\oma) \to \BL^2(\toma)$
such that
\begin{enumerate}
\item \label{it_Ec_a}
$\LE^{\rm c}(\bv)|_{\oma} = \LE^{\rm d}(\bv)|_{\oma} = \bv$ for all $\bv \in \BL^2(\oma)$;
\item \label{it_Ec_b}
$\LE^{\rm c}: \BH_{0,\Gamma}(\ccurl,\oma) \to \BH_{0,\tG}(\ccurl,\toma)$
and
$\LE^{\rm d}: \BH_{0,\Gamma}(\ddiv,\oma) \to \BH_{0,\tG}(\ddiv,\toma)$;
\item \label{it_Ec_c}
$\LE^{\rm c}: \ND_q(\CTa) \to \ND_q(\tCTa)$
and
$\LE^{\rm d}: \RT_q(\CTa) \to \RT_q(\tCTa)$;
\item \label{it_Ec_d}
$\curl (\LE^{\rm c}(\bv)) = \LE^{\rm d}(\curl \bv)$ for all $\bv \in \BH_{0,\Gamma}(\ccurl,\oma)$.
\end{enumerate}
\item
There exist restriction
operators $\LR^{\rm c},\LR^{\rm d}: \BL^2(\toma) \to \BL^2(\oma)$ such that
\begin{enumerate}
\item \label{it_Rd_a}
$(\LR^{\rm c} \circ \LE^{\rm c})(\bv) = (\LR^{\rm d} \circ \LE^{\rm d})(\bv) = \bv$
for all $\bv \in \BL^2(\oma)$;
\item \label{it_Rd_b}
$\LR^{\rm c}: \BH_{0,\tG}(\ccurl,\toma) \to \BH_{0,\Gamma}(\ccurl,\oma)$
and
$\LR^{\rm d}: \BH_{0,\tG}(\ddiv,\toma) \to \BH_{0,\Gamma}(\ddiv,\oma)$;
\item \label{it_Rd_c}
$\LR^{\rm c}: \ND_q(\tCTa) \to \ND_q(\CTa)$
and
$\LR^{\rm d}: \RT_q(\tCTa) \to \RT_q(\CTa)$;
\item \label{it_Rd_RC_com}
$\curl (\LR^{\rm c}(\tbv)) = \LR^{\rm d}(\curl \tbv)$ for all $\tbv \in \BH_{0,\tG}(\ccurl,\toma)$.
\end{enumerate}
\end{enumerate}
\end{definition}

As we state below, extensions can be composed, so that it is possible to extend an initial
patch several times.

\begin{lemma}[Composition of extensions] \label{lem_comp}
If, in the sense of Definition~\ref{definition_extension},
$\tCTa^1$ is an extension of $\CTa$
with operators $\LE_1^{\rm c}$ and $\LR_1^{\rm c}$
and $\tCTa^2$ is an extension of $\tCTa^1$
with operators $\LE_{1,2}^{\rm c}$ and $\LR_{1,2}^{\rm c}$,
then $\tCTa^2$ is an extension of $\CTa$
with operators $\LE_2^{\rm c} \eq \LE_{1,2}^{\rm c} \circ \LE_1^{\rm c}$
and $\LR_2^{\rm c} \eq \LR_1^{\rm c} \circ \LR_{1,2}^{\rm c}$, and corresponding definitions
for the operators $\LE^{\rm d}$ and $\LR^{\rm d}$.
\end{lemma}

Crucially, if we can prove the stability of discrete minimization in
an extension of a given patch, then it also holds on the original patch.
Indeed, we have the following inequality with the constant that only
depends on the norms of the extension and restriction operators of
Definition~\ref{definition_extension}.

\begin{lemma}[Patch extensions]
\label{lemma_extension}
Consider a vertex patch $\CTa$ and an extension $\tCTa \supset \CTa$
in the sense of Definition~\ref{definition_extension}. Then, for all $p \geq 0$, we have
\begin{equation*}
C_{{\rm st},p,\CTa,\Gamma}
\leq
\enorm{\LE^{\rm c}}\enorm{\LR^{\rm c}}
C_{{\rm st},p,\tCTa,\tG}.
\end{equation*}
\end{lemma}

\begin{proof}
Let $\bj_p \in \RT_p(\CTa) \cap \BH_{0,\Gamma}(\ddiv,\oma)$ with $\div \bj_p  = 0$
and $\bchi_p \in \ND_p(\CTa)$. Recall the discrete and continuous minimizers
$\bv_p^\star$ and $\bv^\star$ from~\eqref{eq_minimizers}.

We start by introducing $\tbj_p \eq \LE^{\rm d}(\bj_p) \in \RT_p(\tCTa) \cap \BH_{0,\tG}(\ddiv,\toma)$
and $\tbchi_p \eq \LE^{\rm c}(\bchi_p) \in \ND_p(\tCTa)$.
Due to the commuting properties in Definition~\ref{definition_extension}, we have
$\div \tbj_p  = 0$, so that we can consider the curl-constrained minimization in
the extended patch $\tCTa$ with data $\tbj_p$ and $\tbchi_p$, with essential boundary
conditions on $\tG$. Henceforth, we denote by $\tbv^\star$ and $\tbv_p^\star$ the
associated continuous and discrete minimizers associated with these data in the extended
patch $\tCTa$.

The proof then follows from the following considerations. First,
\begin{align*}
\|\bv_p^\star-\bchi_p\|_{\oma}
\leq
\|\LR^{\rm c}(\tbv_p^\star)-\bchi_p\|_{\oma}
=
\|\LR^{\rm c}(\tbv_p^\star-\tbchi_p)\|_{\oma}
\leq
\enorm{\LR^{\rm c}} \|\tbv_p^\star-\tbchi_p\|_{\toma},
\end{align*}
where we used that $\LR^{\rm c}(\tbv_p^\star)$ is in the discrete minimization set
of the original patch due to our assumptions on $\LR^{\rm c}$ and $\LR^{\rm d}$
in the first inequality and the fact that $\LR^{\rm c}(\tbchi_p) = (\LR^{\rm c} \circ \LE^{\rm c})(\bchi_p) = \bchi_p$
in the equality. Second, we use the stable minimization property in the extended patch, giving
\begin{equation*}
\|\tbv_p^\star-\tbchi_p\|_{\toma}
\leq
C_{{\rm st},p,\tCTa,\widetilde \Gamma}
\|\tbv^\star-\tbchi_p\|_{\toma}.
\end{equation*}
Finally, since $\LE^{\rm c}(\bv^\star)$ is in the continuous minimization of the
extended patch, we conclude the proof with
\begin{equation*}
\|\tbv^\star-\tbchi_p\|_{\toma}
\leq
\|\LE^{\rm c}(\bv^\star)-\tbchi_p\|_{\toma}
=
\|\LE^{\rm c}(\bv^\star-\bchi_p)\|_{\toma}
\leq
\enorm{\LE^{\rm c}} \|\bv^\star-\bchi_p\|_{\oma}.
\end{equation*}
\end{proof}

\subsection{Parachute patches} \label{sec_parachute}

The next central concept is the one of a ``parachute patch'' where
all the vertices except the central vertex lie in the same plane.
As a result, the faces not sharing the central vertex are easily identified
with a two-dimensional planar triangular mesh, which makes the reasoning
easier. An illustration is given in Figure~\ref{fig_patches_par_tetr}, left.

\begin{definition}[Parachute patch]
\label{definition_parachute}
A parachute patch $\CTo$ is a boundary patch around the vertex $\bzero \in \mathbb R^3$
with associated domain $\omo$ such that all the non-central vertices of the patch $\bb \neq \bzero$ lie in the plane $H \eq \{\bx \in \mathbb R^3 \; | \; \bx_3 = 1\}$. In this case, we denote
by $\lceil \CTo \rceil$ the planar triangular mesh induced by $\CTo$ on $H$ and by
$\lceil \omo \rceil \subset \mathbb R^2$ the corresponding two-dimensional planar domain.
\end{definition}

\begin{figure}
\centerline{\raisebox{0.5cm}{\includegraphics[width=0.55\linewidth]{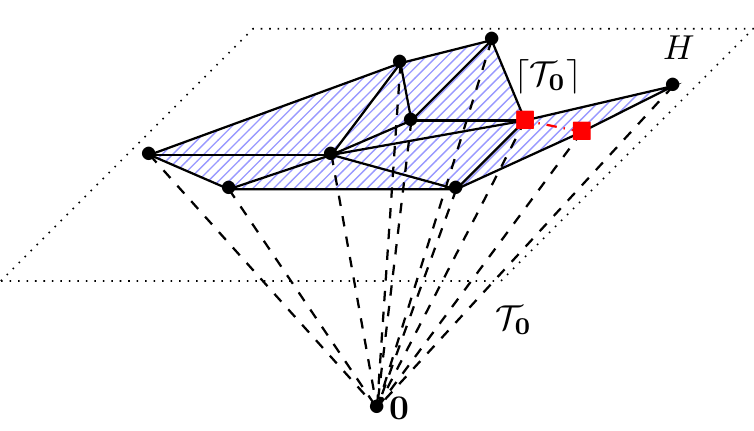}} \includegraphics[width=0.44\linewidth]{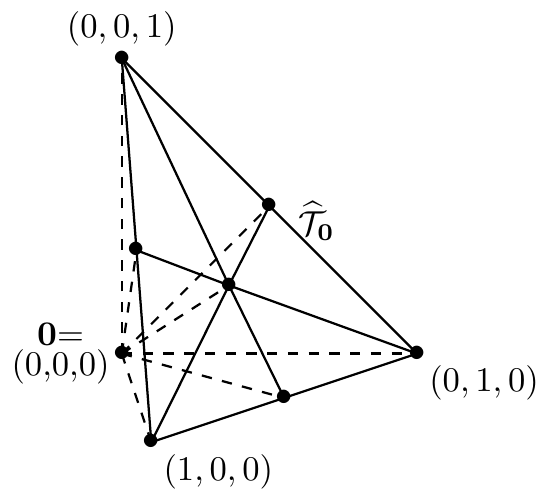}}
\caption{Parachute patch (left) and a reference tetrahedron patch (right)}
\label{fig_patches_par_tetr}
\end{figure}

Crucially, every boundary patch is equivalent to a ``reference'' parachute patch,
as we next demonstrate.

\begin{lemma}[Reference parachute patches]
\label{lemma_reference_parachutes}
Consider a shape-regularity parameter $\kappa > 0$. There exists a finite set
of reference parachute patches $\widehat \LT_\kappa = \{\hCTo\}$
such that if $\CTa$ is a boundary patch with shape-regularity parameter
$\kappa_{\CTa} \leq \kappa$, then there exists exactly one
$\hCTo \in \widehat \LT_\kappa$ that is equivalent to $\CTa$ in the sense of
Definition~\ref{definition_equiv}.
\end{lemma}

\begin{proof}
For each $\kappa > 0$, we denote by $N(\kappa)$ the maximum number of tetrahedra
in a boundary patch $\CTa$ with shape-regularity parameter
$\kappa_{\CTa} \leq \kappa$. For each $N \in \mathbb N$, let $\lceil\widehat \LT_N\rceil$
denote the set of possible reference configurations (in terms of mesh topology/mesh connectivity)
of conforming planar triangular meshes with $N$ elements. Then $\widehat \LT_\kappa$ is defined by
distording each element of $\lceil \widehat \LT_{N(\kappa)} \rceil$ into a parachute patch.

Fix now $\kappa > 0$. If $\CTa$ is a boundary patch with
$\kappa_{\CTa} \leq \kappa$, then $\CTa$ has at most $N(\kappa)$
tetrahedra $K$. Because $\CTa$ is a boundary patch, there is a cone $\mathcal{C}$
with the vertex $\ba$ and a strictly positive solid angle such that
$\mathcal{C} \cap \oma = \emptyset$ (forming an ``opening''). From
the assumptions in Section~\ref{sec_VP}, namely that $\oma$ has a Lipschitz
boundary $\partial \oma$, there is exactly one such an opening. Thus, $\CTa$ can
be transformed into a first parachute patch $\CTo$ with the corresponding planar mesh
$\lceil \CTo \rceil$. This can be done by transforming $\ba$ into $\bzero$ by translation,
dilating the solid angle of the opening until all the vertices lie on the same side of a plane,
rotating the coordinate system, and then shortening the sizes of edges connecting each non-central vertex to the central
one to align them onto the plane $H$. Then, there exists a parachute patch
$\hCTo \in \widehat \LT_\kappa$ such that $\lceil \hCTo \rceil$ has the same
topology/connectivity as $\lceil \CTo \rceil$. Finally, $\CTo$ is equivalent to $\hCTo$,
and therefore $\CTa$ is equivalent to $\hCTo \in \widehat \LT_\kappa$.
\end{proof}

\subsection{Reference tetrahedron patches}

Next, the concept of a ``reference tetrahedron patch'' is key. It is visualized in Figure~\ref{fig_patches_par_tetr}, right, and defined as follows:

\begin{definition}[Reference tetrahedron patch]
\label{definition_tetrahedron}
A reference tetrahedron patch is a patch $\hCTo$ such that
the corresponding domain $\widehat \omega_{\bzero} = \widehat K$, where
\begin{equation*}
\widehat K
\eq
\left \{
\bx \in \mathbb R^3 \; | \;
\bx_1,\bx_2,\bx_3 \geq 0,
\quad
\bx_1 + \bx_2 + \bx_3 \leq 1
\right \}
\end{equation*}
is the reference tetrahedron. We further say that the patch has
mixed boundary conditions if either $\wGaess = \widehat F$ or $\wGanat = \widehat F$
where
\begin{equation*}
\widehat F
\eq
\left \{
\bx \in \widehat K \; | \; \bx_1 = 0
\right \}
\end{equation*}
is one face of the reference tetrahedron $\widehat K$ sharing the vertex $\bzero$.
We say that is has unique boundary conditions if $\wGanat = \widehat C$ or
$\wGaess = \widehat C$ with
\begin{equation*}
\widehat C
\eq
\left \{
\bx \in \widehat K \; | \; \bx_1\bx_2\bx_3 = 0
\right \},
\end{equation*}
i.e., $\widehat C$ corresponds to the three faces of the reference tetrahedron
$\widehat K$ that share the vertex $\bzero$.
\end{definition}

\subsection{Stability for reference tetrahedron patches}

From Section~\ref{sec_parachute}, we know that every boundary
patch is equivalent to a parachute patch. Moreover, it will follow from
Appendix~\ref{section_triangular_representation} that ``most'' parachute patches
are equivalent to a reference tetrahedron patch. Let us thus now establish
the stability of the discrete minimization on a reference tetrahedron patch.
This is done by using the following symmetrization operators.

\begin{definition}[Symmetrization operators]\label{def_sym}
For $1 \leq d \leq 3$, we let
\begin{equation*}
\TS_d(\bx) \eq \bx-2\bx_d\be^d
\end{equation*}
denote the change of coordinates flipping the direction of the $d^{\rm th}$ space
dimension. We call ``mirroring operator'' around $\{\bx_d = 0\}$ the map
$\TM_d^{\bullet}: \BL^2(\{\bx_d > 0\}) \to \BL^2(\mathbb R^3)$ defined by
\begin{equation*}
\left (\TM_d^{\bullet}\bv\right )|_{\{\bx_d > 0\}} \eq \bv \qquad
\left (\TM_d^{\bullet}\bv\right )|_{\{\bx_d < 0\}} \eq \TS_d^{\bullet}(\bv),
\end{equation*}
where $\TS_d^{\bullet}$, $\bullet \in \{{\rm c},{\rm d}\}$, is the Piola
mapping from~\eqref{eq_Piola} associated to $\TS_d$. We also introduce the trivial extension
operator $\TE_d^{\bullet}: \BL^2(\{\bx_d > 0\}) \to \BL^2(\mathbb R^3)$
\begin{equation*}
\left (\TE_d^{\bullet}\bv\right )|_{\{\bx_d > 0\}} \eq \bv \qquad
\left (\TE_d^{\bullet}\bv\right )|_{\{\bx_d < 0\}} \eq \bzero
\end{equation*}
for $\bullet \in \{{\rm c},{\rm d}\}$.
We call ``folding operator'' around $\{\bx_d = 0\}$ the map
$\TF_d^{\bullet}: \BL^2(\mathbb R^3) \to \BL^2(\{\bx_d > 0 \})$
defined by
\begin{equation*}
\TF^\bullet_d(\bv) \eq \bv|_{\{\bx_d > 0\}} - \TS_d^\bullet(\bv|_{\{\bx_d < 0\}})
\end{equation*}
for $\bullet \in \{{\rm c},{\rm d}\}$. Finally, for $\bullet \in \{{\rm c},{\rm d}\}$,
we define the trivial restriction operator
$\TR_d^{\bullet}: \BL^2(\mathbb R^3) \to \BL^2(\{\bx_d > 0 \})$ by
\begin{equation*}
\TR^\bullet_d(\bv) \eq \bv|_{\{\bx_d > 0\}}.
\end{equation*}
\end{definition}

We will apply the following result to reference tetrahedron patches $\hCTo$:

\begin{lemma}[Plane symmetrization]
\label{lemma_symmetrization}
Consider a vertex patch $\CTa$ around $\ba$, fix $d  \in \{1,2,3\}$, and
let $H_d$ denote the plane $\{\bx_d = 0 \}$. We assume either $\Gaess \cap H_d = \emptyset$
or $\Ganat \cap H_d = \emptyset$.  Then, if the symmetrized patch $\tCTa = \CTa \cup \TS_d(\CTa)$
still corresponds to a connected and Lipschitz domain $\toma \eq \oma \cup \TS_d(\oma)$
with admissible boundary $\tGamma \eq \Gamma \cup \TS_d(\Gamma)$, then
$\tCTa$ is an extension of $\CTa$ as per Definition~\ref{definition_extension},
with the operator norms bounded by $2$,
\begin{equation} \label{eq_norms_2}
    \enorm{\LE^{\rm c}}\enorm{\LR^{\rm c}} \leq 2.
\end{equation}
\end{lemma}

\begin{proof}
Following Definition~\ref{definition_extension}, we need to construct extension
and restriction operators between $\BH_{0,\Gamma}(\ccurl,\oma)$ and $\BH_{0,\tGamma}(\ccurl,\toma)$
that properly commute, and we distinguish two cases.

{\bf Case 1.} We first focus on the case where $\Ganat \cap H_d = \emptyset$, so
that functions in $\BH_{0,\Gamma}(\ccurl,\oma)$ satisfy a homogeneous essential
condition on $\partial \oma \cap H_d$. Due to this observation, we can simply extend a
function of $\BH_{0,\Gamma}(\ccurl,\oma)$ by zero onto $\toma$ and preserve its
curl-conformity. A similar observation holds true for $\BH_{0,\Gamma}(\ddiv,\oma)$.
As a result, letting $\LE^{\bullet} \eq \TE_d^{\bullet}$, $\bullet \in \{{\rm c},{\rm d}\}$
is satisfactory.
We then need to construct suitable restriction operators, which we do by folding
around $\{\bx_d = 0\}$. Indeed, here, we cannot simply
take the trivial restriction of a function defined on $\toma$,
since this would violate the (homogeneous) essential boundary conditions on
$\partial \oma \cap H_d$ in general. As a result, we introduce $\LR^{\bullet} \eq \TF^\bullet_d$
for $\bullet \in \{{\rm c},{\rm d}\}$. Because the Piola mappings preserve relevant traces
on $H_d$, they cancel out when folding around $\{ \bx_d = 0 \}$. Hence, the homogeneous essential
boundary conditions are always satisfied, so that we do have
$\LR^{\rm c}: \BH_{0,\tGamma}(\ccurl,\toma) \to \BH_{0,\Gamma}(\ccurl,\oma)$
as well as the counterpart for $\LR^{\rm d}$. The commuting property
between $\LR^{\rm c}$ and $\LR^{\rm d}$ also holds due to standard properties
of Piola mappings, whereas~\eqref{eq_norms_2} is obvious. We finally need to verify that
$\LR^\bullet \circ \LE^\bullet = \operatorname{Id}$.
Let $\bv \in \BL^2(\oma)$. Since $\LE^\bullet(\bv) = \bzero$ on $\toma \setminus \oma$,
we readily see that
$(\LR^\bullet \circ \LE^\bullet)(\bv) = \bv|_{\oma} - \TS_d^{\bullet}(\bzero) = \bv$,
which concludes the proof.

{\bf Case 2.} When $\Gaess \cap H_d = \emptyset$, we essentially proceed the other way
around. It is here harder to extend the functions (their traces do not have to vanish on $H_d$),
but it easier to restrict them (there is no essential condition to satisfy on $H_d$). In fact,
the mirror operators $\LE^\bullet = \TM_d^\bullet$ and the trivial restriction operators
$\LR^\bullet = \TR^\bullet$ have all the required properties. The arguments are similar
to Case 1, and we do not reproduce them for the sake of shortness.
\end{proof}

\begin{figure}
\centerline{\raisebox{1.8cm}{\includegraphics[width=0.35\linewidth]{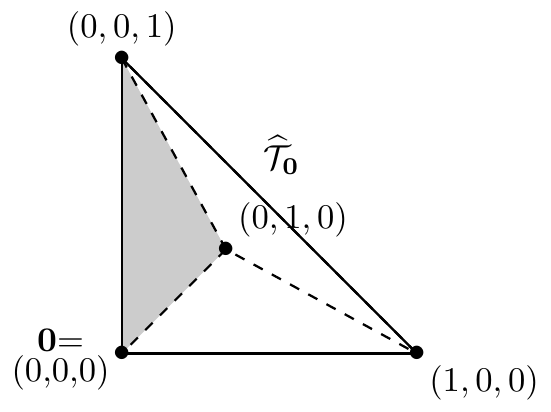}} \qquad \raisebox{0.5cm}{\includegraphics[width=0.43\linewidth]{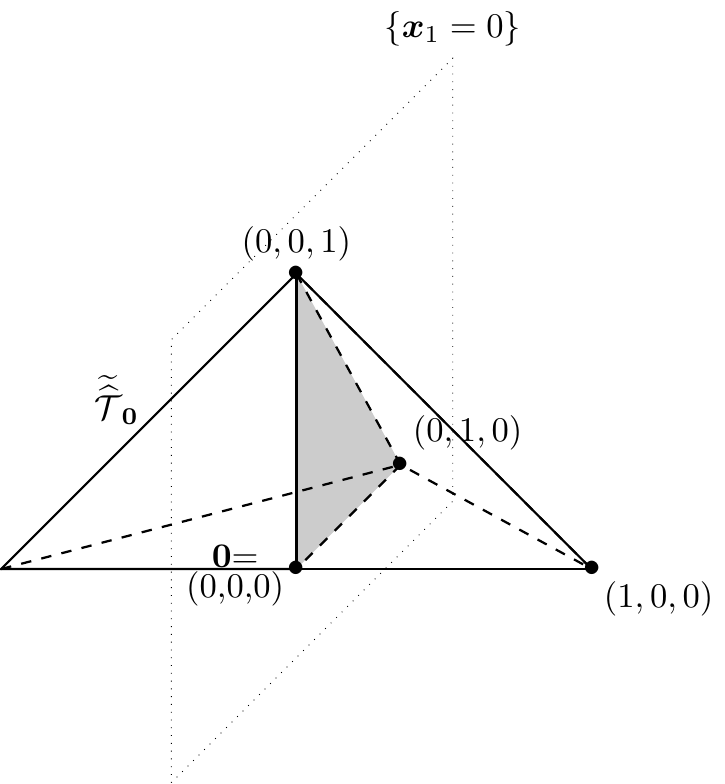}}}
\centerline{\raisebox{0.25cm}{\includegraphics[width=0.43\linewidth]{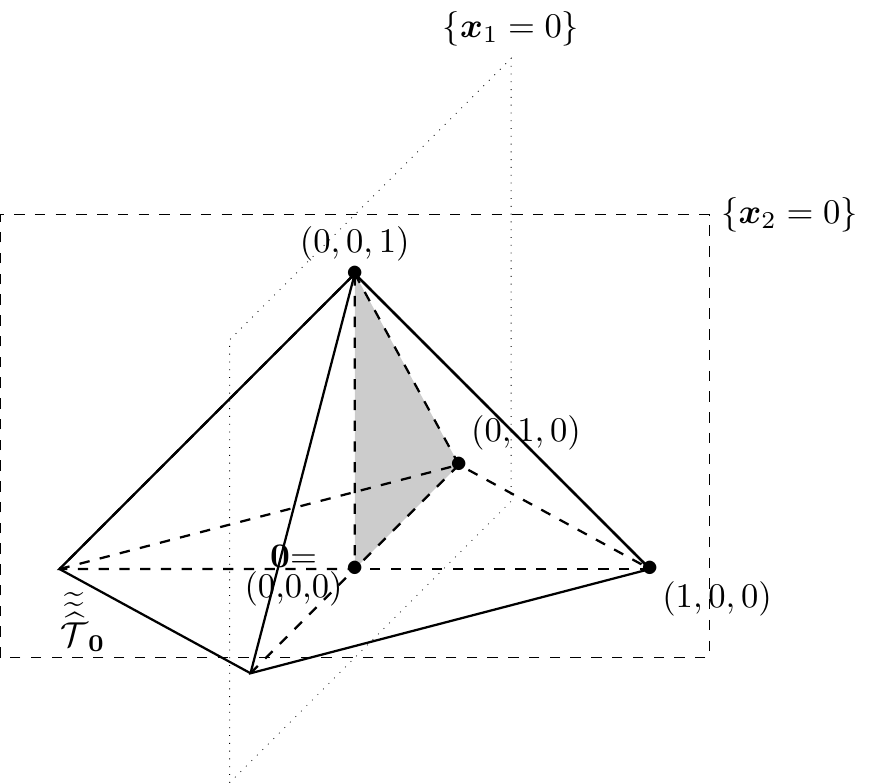}} \includegraphics[width=0.57\linewidth]{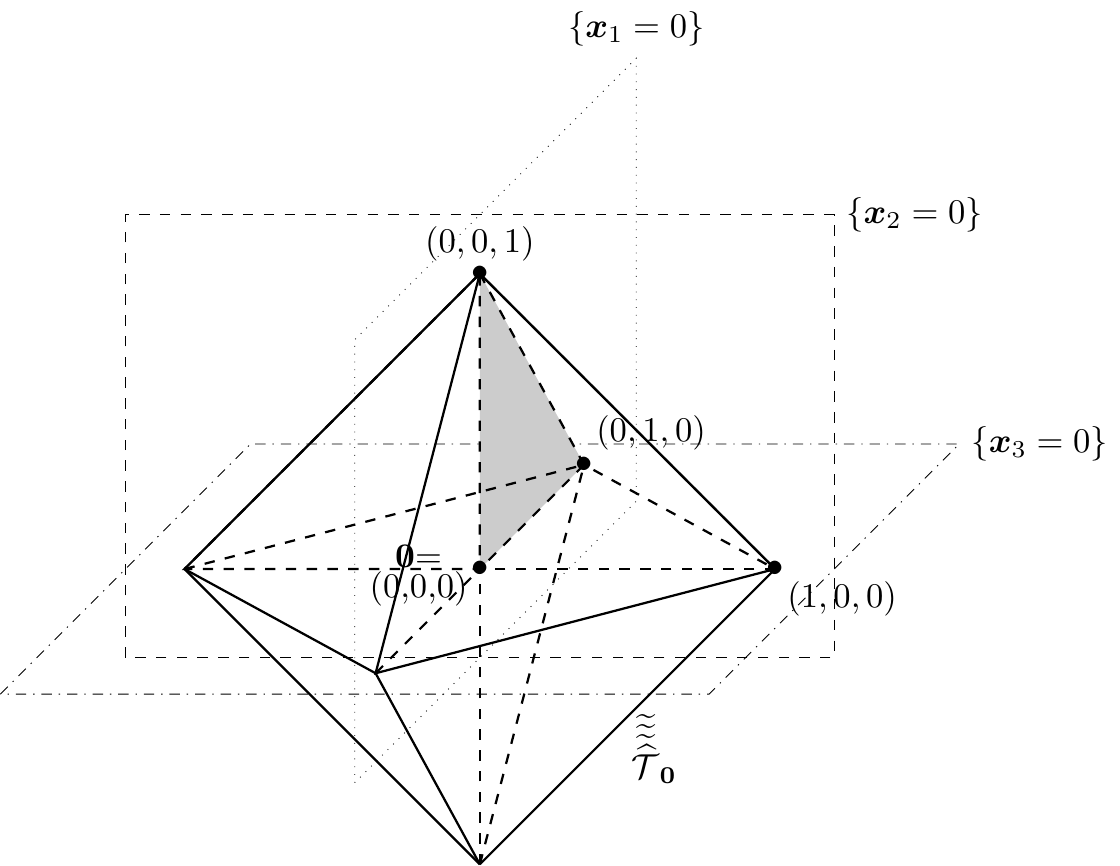}}
\caption{Reference tetrahedron patch (top left), and its consecutive symmetrizations around $\{ \bx_1 = 0 \}$ (top right), $\{ \bx_2 = 0 \}$ (bottom left), and $\{ \bx_3 = 0 \}$ (bottom right). Mixed boundary conditions on $\hCTo$; unique boundary conditions on $\widetilde{\widehat {\mathcal{T}}}_{\boldsymbol 0}$ through $\widetilde{\widetilde{\widehat {\mathcal{T}}}}_{\boldsymbol 0}$}
\label{fig_symm}
\end{figure}

Now, stable discrete minimization for a reference tetrahedron patch easily follows, since it can be extended into an interior patch:

\begin{corollary}[Stable discrete minimization for reference tetrahedron patches]
\label{corollary_tetrahedron}
Let $\hCTo$ be a reference tetrahedron patch in the sense of
Definition~\ref{definition_tetrahedron}. Then, for all $p \geq 0$, we have
\begin{equation*}
C_{{\rm st},p,\hCTo,\widehat\Gamma} \leq C(\kappa_{\CTo}),
\end{equation*}
where the constant on the right-hand side only depends on the shape-regularity parameter
$\kappa_{\hCTo}$.
\end{corollary}

\begin{proof}
The result is a simple consequence of Lemma~\ref{lemma_symmetrization}
(which uses the concept of extension as per Definition~\ref{definition_extension}
and Lemma~\ref{lemma_extension}) and Theorem~\ref{thm_Hc_cons} for interior patches.
Indeed, successively extending the patch by symmetry around the planes
$\{ \bx_d = 0 \}$, $d=1$, $2$, and $3$ results in an interior patch with
8 copies of the original patch $\hCTo$ that is an extension of $\hCTo$,
see Figure~\ref{fig_symm}. For mixed boundary conditions in the sense of
Definition~\ref{definition_tetrahedron}, it is important that $\widehat F$
lies in the plane $\{ \bx_1 = 0 \}$. Then, symmetrizing around $\{ \bx_1 = 0 \}$
following Lemma~\ref{lemma_symmetrization}, the new patch
$\widetilde{\widehat {\mathcal{T}}}_{\boldsymbol 0}$ has unique boundary conditions,
and so is the case therefrom, see Figure~\ref{fig_symm}. Since the last equivalent patch
is interior in the sense of Section~\ref{sec_VP}, we can apply the stable discrete
minimization of Theorem~\ref{thm_Hc_cons} proved in Section~\ref{section_interior_patches}
for any interior patch. Note that because the restriction and extension operators of
Lemma~\ref{lemma_symmetrization} all have operator norms bounded by $2$ as per~\eqref{eq_norms_2},
the resulting constant only depends on the shape-regularity parameter of the tetrahedron patch
$\hCTo$ (the elements are not distorted by the symmetrizations).
\end{proof}

\subsection{Transformation of a general parachute patch into a reference tetrahedron patch}

Unfortunately, not all parachute patches are equivalent to a reference tetrahedron patch.
There in particular exist ``problematic'' cases which are not covered by Appendix
\ref{section_triangular_representation}. Specifically, this happens when the surface mesh
$\lceil \CTo \rceil$, cf. Figure~\ref{fig_patches_par_tetr}, left, has ``problematic'' vertices.
These are the vertices corresponding to ``interior'' edges of $\lceil \CTo \rceil$ with two
vertices on the boundary of $\lceil \omo \rceil$. To give an example, two such vertices and the
corresponding edge are highlighted in Figure~\ref{fig_patches_par_tetr}, left, by the two red
squares connected by the red dash-dotted line.

In this section, we propose a strategy to transform any such a problematic patch into a patch
equivalent to a reference tetrahedron patch. This is done through the concept
of patch extension around the problematic vertices.

\begin{figure}
\begin{minipage}{.40\linewidth}
\begin{tikzpicture}[scale=2]

\coordinate (b)  at (0,0);

\coordinate (b1) at ($(b) + (-  10:1.0)$);
\coordinate (b2) at ($(b) + (-  80:1.0)$);
\coordinate (b3) at ($(b) + (- 120:1.0)$);
\coordinate (b4) at ($(b) + (- 150:1.0)$);

\draw[ultra thick] (b) -- (b1);
\draw[ultra thick] (b) -- (b2);
\draw[ultra thick] (b) -- (b3);
\draw[ultra thick] (b) -- (b4);

\draw[ultra thick] (b1) -- (b2) -- (b3) -- (b4);

\coordinate (c) at ($(b) + ( 60:0.5)$);
\coordinate (d) at ($(b) + (120:0.5)$);

\draw (b) -- (c);
\draw (b) -- (d);

\draw (b1) -- (c) -- (d) -- (b4);

\coordinate (c1) at ($0.333*(c)+0.666*(b1)$);
\coordinate (c2) at ($0.666*(c)+0.333*(b1)$);

\draw[dashed] (b) -- (c1);
\draw[dashed] (b) -- (c2);

\coordinate (d1) at ($0.333*(d)+0.666*(c)$);
\coordinate (d2) at ($0.666*(d)+0.333*(c)$);

\draw[dashed] (b) -- (d1);
\draw[dashed] (b) -- (d2);

\coordinate (e1) at ($0.333*(b4)+0.666*(d)$);
\coordinate (e2) at ($0.666*(b4)+0.333*(d)$);

\draw[dashed] (b) -- (e1);
\draw[dashed] (b) -- (e2);

\draw[ultra thick,dashed] (b4) -- ($(b4) + (160:0.3)$);
\draw[ultra thick,dashed] (b4) -- ($(b4) + (240:0.3)$);

\draw[ultra thick,dashed] (b3) -- ($(b3) + (180:0.3)$);
\draw[ultra thick,dashed] (b3) -- ($(b3) + (280:0.3)$);

\draw[ultra thick,dashed] (b2) -- ($(b2) + (-160:0.3)$);
\draw[ultra thick,dashed] (b2) -- ($(b2) + (- 80:0.3)$);
\draw[ultra thick,dashed] (b2) -- ($(b2) + (-  0:0.3)$);

\draw[ultra thick,dashed] (b1) -- ($(b1) + (- 90:0.3)$);
\draw[ultra thick,dashed] (b1) -- ($(b1) + (- 20:0.3)$);

\draw (c2) node[anchor=south west] {$V'$};
\draw ($0.5*(d1)+0.5*(d2)$) node[anchor=south] {$\widetilde V'$};
\draw (e1) node[anchor=east] {$\widetilde V$};
\draw (b) node {$\bullet$};
\draw (b) node[anchor=north west] {$\bb$};

\end{tikzpicture}
\end{minipage}
\begin{minipage}{.40\linewidth}
\begin{tikzpicture}[scale=2,rotate=90]

\coordinate (b)  at (0,0);

\coordinate (b1) at ($(b) + (   30:1.1)$);
\coordinate (b2) at ($(b) + (- 100:1.3)$);
\coordinate (b3) at ($(b) + (- 210:1.2)$);

\draw[ultra thick] (b) -- (b1);
\draw[ultra thick] (b) -- (b2);
\draw[ultra thick] (b) -- (b3);

\draw[ultra thick] (b1) -- (b2) -- (b3);

\coordinate (c) at ($(b) + ( 60:0.8)$);
\coordinate (d) at ($(b) + (120:0.8)$);

\draw (b) -- (c);
\draw (b) -- (d);

\draw (b1) -- (c) -- (d) -- (b3);

\coordinate (c1) at ($0.5*(c)+0.5*(b1)$);

\draw[dashed] (b) -- (c1);

\coordinate (d1) at ($0.5*(d)+0.5*(c)$);

\draw[dashed] (b) -- (d1);

\coordinate (e1) at ($0.5*(b3)+0.5*(d)$);

\draw[dashed] (b) -- (e1);

\draw[ultra thick,dashed] (b1) -- ($(b1) + (  60:0.3)$);
\draw[ultra thick,dashed] (b1) -- ($(b1) + (- 20:0.3)$);
\draw[ultra thick,dashed] (b1) -- ($(b1) + (- 90:0.3)$);

\draw[ultra thick,dashed] (b2) -- ($(b2) + (  20:0.3)$);
\draw[ultra thick,dashed] (b2) -- ($(b2) + (- 60:0.3)$);
\draw[ultra thick,dashed] (b2) -- ($(b2) + (-140:0.3)$);
\draw[ultra thick,dashed] (b2) -- ($(b2) + (-220:0.3)$);

\draw[ultra thick,dashed] (b3) -- ($(b3) + (-140:0.3)$);
\draw[ultra thick,dashed] (b3) -- ($(b3) + (-240:0.3)$);

\draw (c1) node[anchor=east] {$V'$};
\draw (d1) node[anchor=east] {$\widetilde V'$};
\draw (e1) node[anchor=east] {$\widetilde V$};
\draw (b) node {$\bullet$};
\draw (b) node[anchor=north west] {$\bb$};

\end{tikzpicture}
\end{minipage}

\caption{Extension around a problematic vertex $\bb$ (top view): three virtual tetrahedra added and then each is divided into the number of tetrahedra originally sharing the vertex $\bm{b}$}
\label{figure_vertex_extension}
\end{figure}
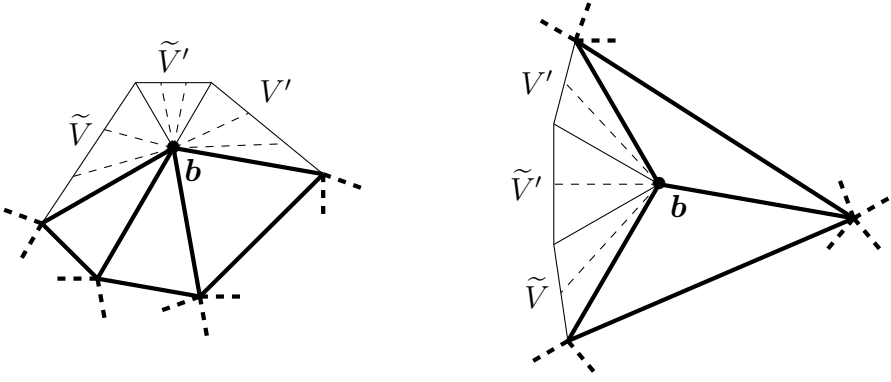

\begin{figure}
\begin{minipage}{.40\linewidth}
\begin{tikzpicture}[scale=2]

\draw (-1, 0) -- ( 1, 0);
\draw ( 0,-1) -- ( 0, 1);

\draw (-1, 0) -- ( 0,-1) -- ( 1, 0) -- ( 0, 1) -- cycle;

\draw[ultra thick] ( 1, 0) -- ( 0,  1) -- (0,0) -- cycle;

\draw[dashed] (0,0) -- (-0.333,-0.666);
\draw[dashed] (0,0) -- (-0.666,-0.333);

\draw[dashed] (0,0) -- ( 0.333,-0.666);
\draw[dashed] (0,0) -- ( 0.666,-0.333);

\draw[ultra thick] (0,0) -- ( 0.333, 0.666);
\draw[ultra thick] (0,0) -- ( 0.666, 0.333);

\draw[dashed] (0,0) -- (-0.333, 0.666);
\draw[dashed] (0,0) -- (-0.666, 0.333);

\draw (-0.5,-0.5) node[anchor=north east] {$\widetilde U'$};
\draw (-0.5, 0.5) node[anchor=south east] {$U'$};
\draw ( 0.5, 0.5) node[anchor=south west] {$U$};
\draw ( 0.5,-0.5) node[anchor=north west] {$\widetilde U$};
\draw (0,0) node {$\bullet$};
\draw (0,0) node[anchor=north west] {$\bb$};

\end{tikzpicture}
\end{minipage}
\begin{minipage}{.40\linewidth}
\begin{tikzpicture}[scale=2]

\draw (-1, 0) -- ( 1, 0);
\draw ( 0,-1) -- ( 0, 1);

\draw (-1, 0) -- ( 0,-1) -- ( 1, 0) -- ( 0, 1) -- cycle;

\draw[ultra thick] ( 1, 0) -- ( 0,  1) -- (0,0) -- cycle;

\draw[dashed     ] (0,0) -- (-0.5,-0.5);
\draw[dashed     ] (0,0) -- ( 0.5,-0.5);
\draw[ultra thick] (0,0) -- ( 0.5, 0.5);
\draw[dashed     ] (0,0) -- (-0.5, 0.5);

\draw (-0.5,-0.5) node[anchor=north east] {$\widetilde U'$};
\draw (-0.5, 0.5) node[anchor=south east] {$U'$};
\draw ( 0.5, 0.5) node[anchor=south west] {$U$};
\draw ( 0.5,-0.5) node[anchor=north west] {$\widetilde U$};
\draw (0,0) node {$\bullet$};
\draw (0,0) node[anchor=north west] {$\bb$};

\end{tikzpicture}
\end{minipage}

\caption{Extension around a problematic vertex $\bb$ after a mapping (top view)}
\label{figure_vertex_extension_2}
\end{figure}
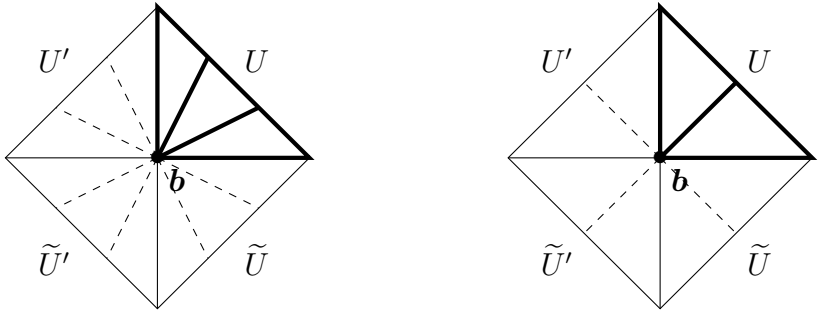

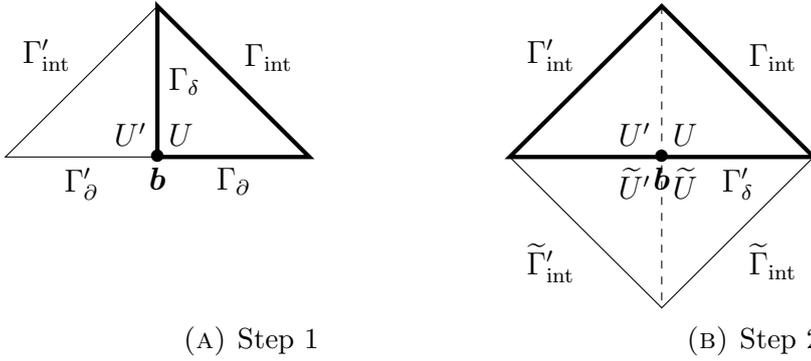
\begin{figure}

\begin{minipage}{.40\linewidth}
\begin{tikzpicture}[scale=2]

\draw[white] (0,1) -- (0,-1);

\draw[ultra thick] (0,0) -- ( 1,0) -- (0, 1) -- cycle;

\draw ( 0.5, 0.5) node[anchor=south west] {$\Gamma_{\rm int}$};
\draw ( 0.5, 0.0) node[anchor=north     ] {$\Gamma_{\partial}$};
\draw ( 0.0, 0.5) node[anchor=west      ] {$\Gamma_{\delta}$};
\draw (-0.5, 0.5) node[anchor=south east] {$\Gamma_{\rm int}'$};
\draw (-0.5, 0.0) node[anchor=north     ] {$\Gamma_{\partial}'$};

\draw (0,0) -- (0,1) -- (-1,0) -- cycle;

\draw ( 0., 0.) node[anchor=south west] {$U$};
\draw (-0., 0.) node[anchor=south east] {$U'$};

\draw (0,0) node {$\bullet$};
\draw (0,0) node[anchor=north] {$\bb$};

\end{tikzpicture}
\subcaption{Step 1}
\end{minipage}
\begin{minipage}{.40\linewidth}
\begin{tikzpicture}[scale=2]

\draw[ultra thick] (-1,0) -- (0, 1) -- (1,0) -- cycle;
\draw              (-1,0) -- (0,-1) -- (1,0) -- cycle;

\draw ( 0.5, 0.0) node[anchor=north] {$\Gamma_{\delta}'$};

\draw ( 0., 0.) node[anchor=south west] {$U$};
\draw (-0., 0.) node[anchor=south east] {$U'$};
\draw ( 0., 0.) node[anchor=north west] {$\widetilde U$};
\draw (-0., 0.) node[anchor=north east] {$\widetilde U'$};

\draw ( 0.5, 0.5) node[anchor=south west] {$\Gamma_{\rm int}$};
\draw (-0.5, 0.5) node[anchor=south east] {$\Gamma_{\rm int}'$};

\draw (-0.5,-0.5) node[anchor=north east] {$\widetilde \Gamma_{\rm int}'$};
\draw ( 0.5,-0.5) node[anchor=north west] {$\widetilde \Gamma_{\rm int} $};

\draw (0,0) node {$\bullet$};
\draw (0,0) node[anchor=north] {$\bb$};

\draw[dashed] (0,1) -- (0,-1);

\end{tikzpicture}
\subcaption{Step 2}
\end{minipage}
\caption{Symmetrization of $U$ with the problematic vertex $\bm{b}$ in two steps: first over $\{ \bx_1 = 0 \}$ and second over $\{ \bx_2 = 0 \}$}
\label{figure_vertex_extension_steps}
\end{figure}

\begin{lemma}[Extension around problematic vertices]
\label{lemma_vertex_extension}
Consider a parachute patch $\CTo$ such that
$\Goess$ and $\Gonat$ are both connected with Lipschitz boundaries, and let $\bb$
be a vertex on the boundary of $\lceil \omo \rceil$. Then, there exists an extension
$\tCTo$ is of $\CTo$, $\tCTo \supset \CTo$, for which $\bb$ lies in the interior of
$\lceil \tomo \rceil$, and the edges added in $\lceil \tCTo \rceil$
as compared to $\lceil \CTo \rceil$ either entirely lie in $\partial \tomo$, or have
one interior vertex inside $\lceil \tomo \rceil$. In addition, $\Goess \subset \tGoess$,
$\Gonat \subset \tGonat$, and both $\tGoess$ and $\tGonat$ are
connected with Lipschitz boundaries. Crucially,
\begin{equation*}
C_{{\rm st},p,\CTo,\Gamma} \leq C(\CTo) C_{{\rm st},p,\tCTo,\widetilde \Gamma},
\qquad
\forall p \geq 0.
\end{equation*}
\end{lemma}

\begin{proof}
We start by ``closing'' the vertex patch around the problematic vertex $\bb$.
We first add three new ``virtual'' tetrahedra, as illustrated (top view) in
Figure~\ref{figure_vertex_extension}. This is always possible, since the boundary
of the patch $\CTo$ is Lipschitz. Intuitively, since the domain cannot lie on the
two sides of its boundary, there is always ``room'' to fit three tetrahedra.
These tetrahedra may need to have small edges, but since we only work in a reference
configuration, this is not important. Remark that we can also do this in such a way
that domain with the added tetrahedra is still Lipschitz.

Proceeding counterclockwise, we call the three tetrahedra closing the loop
around the vertex $\bb$ as $V'$, $\widetilde V'$, and $\widetilde V$.
Denoting by $N$ the number of tetrahedra sharing $\bb$ in $\CTo$,
we then divide the virtual tetrahedra $V'$, $\widetilde V'$, and $\widetilde V$
into $N$ real tetrahedra each, all sharing the vertex $\bb$. This is again shown
in Figure~\ref{figure_vertex_extension}.

Let us call $K_1,\dots,K_N$ the tetrahedra sharing the vertex $\bb$ in the original
patch. Let us denote by $K'_j$, $\widetilde K'_j$, and $\widetilde K_j$ their
counterparts in the virtual tetrahedra $V'$, $\widetilde V'$, and $\widetilde V$.
The extension and restriction operators we are going to construct
will only involve the tetrahedra $K_j$ in the patch $\CTa$ sharing the vertex $\bb$
and eventually map only into the $K_j$, $K'_j$, $\widetilde K'_j$, and $\widetilde K_j$.
Hence, we can think of the elements around $\bb$ in isolation.

Next, we abstractly map the elements around $\bb$ in such a way that the image
of $\bb$ is above $\bzero$ (i.e., at coordinates $(0,0,1)$) and that
$U,U'$, $\widetilde U'$, and $\widetilde U$, the images of $V,V'$, $\widetilde V'$,
and $\widetilde V$, are tetrahedra with three right angles at
the image of $\bb$, see Figure~\ref{figure_vertex_extension_2}.
We can further assume that these four tetrahedra have exactly
the same shape. Of course, ``physically'' mapping these elements could lead to
self-penetration with other patch elements. However, here, we only perform this operation
abstractly since no other elements are involved (we work in isolation).
The mapped configuration around $\bb$ is obtained by a piecewise affine bilipschitz mapping
whose norm only depends on the shape regularity parameter of $\CT_{\bzero}$, and hence,
arguing as in Lemma~\ref{lemma_equiv}, it does not change the end result.

The goal is now to extend the patch $\CTo$ into a new patch $\tCTo$ containing
all the tetrahedra $K'_j$, $\widetilde K'_j$, and $\widetilde K_j$ in $V'$, $\widetilde V'$,
and $\widetilde V$. For simplicity, we perform this extension in the mapped configurations
containing $U$, $U'$, $\widetilde U'$, and $\widetilde U$, see
Figure~\ref{figure_vertex_extension_2}. Specifically, following very closely
Corollary~\ref{corollary_tetrahedron}, we symmetrize $U$ twice.
The extension is thus performed in two steps: we first extend the patch by including
$U'$ by symmetrizing over $\{ \bx_1 = 0 \}$, and then carry out a second extension
which include $\widetilde U'$ and $\widetilde U$ with a symmetrization over
$\{ \bx_2 = 0 \}$. These symmetrizations around planes are performed as in
Definition~\ref{def_sym} and in Lemma~\ref{lemma_symmetrization}, and the process
is illustrated in Figure~\ref{figure_vertex_extension_steps}. As shown in
Figures~\ref{figure_vertex_extension}--\ref{figure_vertex_extension_2},
in terms of connectivity, the vertex patch around $\bb$ in the new patch will
contain 4 copies of the original one. Besides, the extension only introduces edges
with one interior vertex (namely, $\bb$), so that it does not introduce any new
problematic edges/vertices in $\lceil \tCTo \rceil$.

The boundary of $U$ consists of four faces. One of those, which we denote
by $\Gamma_{\rm int}$, connects $U$ to the remainder of the patch. We denote by
$\Gamma_{\rm top}$ the face that does not touch the vertex $\bzero$ (it is the one
in which Figures~\ref{figure_vertex_extension}--\ref{figure_vertex_extension_steps} are drawn).
We denote by $\Gamma_{\delta}$ and $\Gamma_{\partial}$ the other two faces of $U$.
They lie on the boundary of the patch subdomain $\omo$ and contain the vertex $\bzero$.
In what follows, we need to distinguish between cases of boundary
conditions on $\Gamma_{\delta}$ and $\Gamma_{\partial}$. To fix the ideas,
we will assume here that $\Gamma_{\delta} \subset \{ \bx_1 = 0 \}$
and $\Gamma_{\partial} \subset \{ \bx_2 = 0 \}$ as represented in
Figures~\ref{figure_vertex_extension_2}--\ref{figure_vertex_extension_steps},
although it is not important. For the sake of simplicity, we denote by
$\omega_U$ the domain covered by $U$, $U'$, $\widetilde U'$, and $\widetilde U$.
We also employ the notation $\Gamma_{\rm int}'$ for the mirror image of
$\Gamma_{\rm int}$ around $\{ \bx_1 = 0 \}$ and we denote by $\widetilde \Gamma_{\rm int}$
and $\widetilde \Gamma_{\rm int}'$ the images of $\Gamma_{\rm int}$ and $\Gamma_{\rm int}'$
when symmetrizing around $\{ \bx_2 = 0 \}$. All this notation is illustrated in
Figure~\ref{figure_vertex_extension_steps}.

{\bf Case 1.}
We first consider the case where natural boundary
conditions are considered on $\Gamma_\delta$ and $\Gamma_\partial$,
i.e. $\Gamma_\delta,\Gamma_\partial \subset \Gonat$.
In this case, we will construct an extension operator from
$\BH(\ccurl,U)$ to $\BH(\ccurl,\omega_U)$ and the corresponding
restriction operator, as well as divergence-conforming counterparts.
Crucially, the restriction operator will preserve the trace on
$\Gamma_{\rm int}$, so that it also corresponds to a restriction
operator for the whole patch. Here, using the notation of Definition~\ref{def_sym}, we set
$\LE^\bullet \eq \TM_2^\bullet \circ \TM_1^\bullet$ and
$\LR^\bullet \eq \TR_1^{\bullet} \circ \TR_2^{\bullet}$.
It should be observed that the natural boundary condition on
$\Gamma_{\delta}$ and $\Gamma_{\partial}$ is replaced by a natural
boundary condition on $\Gamma_{\rm int}'$, $\tGamma_{\rm int}'$,
and $\tGamma_{\rm int}$ in the extended patch. Hence, we let $\tGamma \eq \Gamma$
and we easily verify that both $\tGonat$ and $\tGoess$ are still connected and
have Lipschitz boundaries.

{\bf Case 2.}
We now consider the more complicated case where essential boundary conditions
are imposed everywhere on $\Gamma_{\delta}$ and $\Gamma_{\partial}$.
The goal here is to construct extension and restriction mappings operating
between $\BH_{0,\gamma_U}(\ccurl,U)$ and $\BH_{0,\gamma_{\omega_U}}(\ccurl,U)$
with $\gamma_U \eq \Gamma_\delta \cup \Gamma_\partial$ and
$\gamma_{\omega_U} \eq \Gamma_{\rm int}' \cup \widetilde \Gamma_{\rm int}' \cup \widetilde \Gamma_{\rm int}$
Because we only extend through essential boundary conditions, we can simply
use the trivial extension operator $\LE^\bullet \eq \TE_2^{\bullet} \circ \TE_1^{\bullet}$.
We must however, be careful in analyzing the restriction operator. In particular,
the trace of the function of $\Gamma_{\rm int}$ must be preserved, so as to ensure
conformity of the restricted function in the entire $\omo$.
The restriction operator is constructed by first folding over $\{ \bx_2 = 0 \}$
and then folding over $\{ \bx_1 = 0 \}$, i.e. $\LR^\bullet \eq \TF_1^\bullet \circ \TF_2^\bullet$.
Crucially, we can check that only $\widetilde \Gamma_{\rm int}$, $\widetilde \Gamma_{\rm int}'$,
and $\Gamma_{\rm int}'$ are folded over $\Gamma_{\rm int}$ in the restriction operation. Since
essential conditions are imposed on these parts of the boundary, the trace on
$\Gamma_{\rm int}$ is left intact. As in Case 1, the boundary conditions on
$\gamma_{\omega_U}$ in the extended patch are of the same type as the ones originally imposed
on $\gamma_U$. Therefore, we set $\tGamma \eq  \Gamma \setminus \gamma_U \cup \gamma_{\omega_U}$
and we see that both $\tGonat$ and $\tGoess$ are connected and have Lipschitz boundaries.

{\bf Case 3.}
We finally address the case of mixed boundary conditions.
Here, we assume that the boundary has been labeled so that
$\Gamma_{\delta}$ corresponds to the natural boundary condition,
whereas essential boundary conditions are required on $\Gamma_{\partial}$,
i.e. $\Gamma_{\delta} \subset \Gonat$ and $\Gamma_{\partial} \subset \Goess$.
We are therefore going to construct operators acting between
$\BH_{0,\gamma_U}(\ccurl,U)$ and $\BH_{0,\gamma_{\omega_U}}(\ccurl,\omega_U)$
where $\gamma_U = \Gamma_{\partial}$, and $\gamma_{\omega_U}$ is covered
by $\widetilde \Gamma_{\rm int}'$ and $\widetilde \Gamma_{\rm int}$.
As can be seen in Figure~\ref{figure_vertex_extension_steps}, it thus
means that the extended patch will have natural boundary conditions
on $\Gamma_{\rm int}'$ (instead of $\Gamma_{\delta}$ in the original
patch) so that both $\tGoess$ and $\tGonat$ remain connected with Lipschitz boundaries.
Concretely, we start by mirroring around $\{ \bx_1 = 0 \}$ and then we extend by zero around
$\{ \bx_2 = 0 \}$. Specifically, we set
$\LE^\bullet \eq  \TE_2^{\bullet} \circ \TM_1^\bullet$, and
$\LR^\bullet \eq \TR_1^\bullet \circ \TF_2^{\bullet}$.
It is important to note the second step of the extension properly works because in the
intermediate configuration $U \cup U'$, essential boundary conditions are imposed on both
$\Gamma_\partial$ and its mirror image $\Gamma_\partial'$ around $\{ \bx_1 = 0 \}$.
Similarly, the restriction operator does preserve the trace on $\Gamma_{\rm int}$,
since the fold around $\{ \bx_2 = 0 \}$ maps $\widetilde \Gamma_{\rm int}$ onto
$\Gamma_{\rm int}$.
\end{proof}

\begin{corollary}[Extension into a tetrahedron patch]
\label{corollary_extension_tetrahedron}
Every parachute patch admits an extension to a reference tetrahedron patch.
\end{corollary}

\begin{proof}
We recursively apply Lemma~\ref{lemma_vertex_extension} to the possible problematic vertices
until no remain. Note that the number of problematic vertices is finite and that their number
in $\tCTo$ is always diminished by at least one in comparison with $\CTo$. Finally, when there
is no problematic vertex left, we can apply Proposition~\ref{proposition_triangular_representation}
to the last $\lceil \tCTo \rceil$ to obtain a planar triangular mesh of the reference triangle.
This then first corresponds to a parachute patch of a form of a tetrahedron and finally to
a reference tetrahedron patch in the sense of Definition~\ref{definition_tetrahedron}.
More precisely, as a first possibility, $\Ganat$ is empty, i.e., $\Ga$ is empty, as in
Figure~\ref{fig_patches_1}, right. As a second possibility, $\Gaess$ is empty, i.e, $\Ga$
collects all the faces from the boundary of $\oma$ sharing the vertex $\ba$, as in
Figure~\ref{fig_patches_2}, left. In these two cases, we obtain unique boundary conditions
in the sense of Definition~\ref{definition_tetrahedron}. As a third and last possibility,
both $\Ganat$ and $\Gaess$ contain at least one face from the boundary of $\oma$ sharing
the vertex $\ba$. Then, in Proposition~\ref{proposition_triangular_representation}, we note
that we can organize all edges corresponding to, say, $\Ganat$, in one edge of the reference
triangle and all edges corresponding to $\Gaess$ in the two remaining edges of the reference
triangle. This then leads to mixed boundary conditions in the sense of
Definition~\ref{definition_tetrahedron}.
\end{proof}

\subsection{Proof of Theorem~\ref{thm_Hc_cons} for boundary patches}

We are now finally ready to prove Theorem~\ref{thm_Hc_cons} for boundary patches.

\begin{proof}[Proof of Theorem~\ref{thm_Hc_cons} for boundary patches]
First, Lemma~\ref{lemma_reference_parachutes} states that $\CTa$
is equivalent to a reference parachute patch $\hCTo \in \widehat \LT_{\kappa_{\CTa}}$,
and Lemma~\ref{lemma_equiv} ensures that
\begin{equation*}
C_{{\rm st},p,\CTa,\Gamma}
\leq
C(\kappa_{\CTa})
C_{{\rm st},p,\hCTo,\widehat \Gamma},
\end{equation*}
since $\kappa_{\hCTo}$ only depends on $\kappa_{\CTa}$.
We then follow Corollary~\ref{corollary_extension_tetrahedron} to extend $\hCTo$
into a reference tetrahedron patch $\tCTo$ with extension and restriction operator
norms only depending on $\kappa_{\hCTo}$, and hence, only on $\kappa_{\CTa}$,
so that
\begin{equation*}
C_{{\rm st},p,\hCTo,\widehat \Gamma}
\leq
C(\kappa_{\CTa})
C_{{\rm st},p,\tCTo,\widetilde \Gamma}.
\end{equation*}
Then, the result follows from Corollary~\ref{corollary_tetrahedron} which states that
\begin{equation*}
C_{{\rm st},p,\tCTo,\widetilde \Gamma}
\leq
C(\kappa_{\tCTo})
\end{equation*}
and from the fact that $C(\kappa_{\tCTo}) \leq C(\kappa_{\CTa})$.
\end{proof}

\bibliographystyle{amsplain_initials}
\bibliography{bibliography}

\providecommand{\bysame}{\leavevmode\hbox to3em{\hrulefill}\thinspace}
\providecommand{\MR}{\relax\ifhmode\unskip\space\fi MR }
\providecommand{\MRhref}[2]{%
  \href{http://www.ams.org/mathscinet-getitem?mr=#1}{#2}
}
\providecommand{\href}[2]{#2}
\begin{thebibliography}{10}

\bibitem{adams_fournier_2003a}
R. Adams and J. Fournier, \emph{Sobolev spaces}, Academic Press, 2003.

\bibitem{ainsworth_oden_2000a}
M. Ainsworth and J. Oden, \emph{A posteriori error estimation in finite element
  analysis}, Wiley, 2000.

\bibitem{Arn_Falk_Winth_FEC_06}
D.~N. Arnold, R.~S. Falk, and R. Winther, \emph{Finite element exterior
  calculus, homological techniques, and applications}, Acta Numer. \textbf{15}
  (2006), 1--155. \MR{2269741}

\bibitem{Bof_Brez_For_MFEs_13}
D. Boffi, F. Brezzi, and M. Fortin, \emph{Mixed finite element methods and
  applications}, Springer Series in Computational Mathematics, vol.~44,
  Springer, Heidelberg, 2013. \MR{3097958}

\bibitem{braess_pillwein_schoberl_2009a}
D. Braess, V. Pillwein, and J. Sch\"oberl, \emph{Equilibrated residual error
  estimates are {$p$}-robust}, Comput. Meth. Appl. Mech. Engrg. \textbf{198}
  (2009), 1189--1197.

\bibitem{braess_schoberl_2008a}
D. Braess and J. Sch\"oberl, \emph{Equilibrated residual error estimators for
  edge elements}, Math. Comp. \textbf{77} (2008), no.~262, 651--672.

\bibitem{chaumontfrelet_ern_vohralik_2020a}
T. Chaumont-Frelet, A. Ern, and M. Vohral{\'{\i}}k,
  \emph{Polynomial-degree-robust {${\boldsymbol H}(\mathrm{curl})$}-stability
  of discrete minimization in a tetrahedron}, C. R. Math. Acad. Sci. Paris
  \textbf{358} (2020), no.~9--10, 1101--1110.

\bibitem{chaumontfrelet_ern_vohralik_2021b}
\bysame, \emph{On the derivation of guaranteed and $p$-robust a posteriori
  error estimates for the {H}elmholtz equation}, Numer. Math. \textbf{148}
  (2021), no.~3, 525--573.

\bibitem{chaumontfrelet_ern_vohralik_2021a}
\bysame, \emph{Stable broken {${\boldsymbol H}(\mathrm{\bf curl})$} polynomial
  extensions and $p$-robust a posteriori error estimates by broken patchwise
  equilibration for the curl--curl problem}, Math. Comp. \textbf{91} (2022),
  no.~333, 37--74.

\bibitem{chaumontfrelet_vohralik_2021a}
T. Chaumont-Frelet and M. Vohral{\'{\i}}k, \emph{Equivalence of local-best and
  global-best approximations in {${\boldsymbol H}(\mathrm{curl})$}}, Calcolo
  \textbf{58} (2021), 53.

\bibitem{Chaum_Voh_Maxwell_equil_21}
\bysame, \emph{$p$-robust equilibrated flux reconstruction in {${\boldsymbol
  H}(\mathrm{curl})$} based on local minimizations. {A}pplication to a
  posteriori analysis of the curl--curl problem}, SIAM J. Numer. Anal.
  \textbf{61} (2023), no.~4, 1783--1818.

\bibitem{ciarlet_2002a}
P.~G. Ciarlet, \emph{The finite element method for elliptic problems}, Classics
  in Applied Mathematics, vol.~40, Society for Industrial and Applied
  Mathematics (SIAM), Philadelphia, PA, 2002, Reprint of the 1978 original
  [North-Holland, Amsterdam; MR0520174 (58 \#25001)]. \MR{MR1930132}

\bibitem{Cost_McInt_Bog_Poinc_10}
M. Costabel and A. McIntosh, \emph{On {B}ogovski\u\i\ and regularized
  {P}oincar\'e integral operators for de {R}ham complexes on {L}ipschitz
  domains}, Math. Z. \textbf{265} (2010), no.~2, 297--320. \MR{2609313
  (2011f:58030)}

\bibitem{Demk_Gop_Sch_ext_I_09}
L. Demkowicz, J. Gopalakrishnan, and J. Sch{\"o}berl, \emph{Polynomial
  extension operators. {P}art {I}}, SIAM J. Numer. Anal. \textbf{46} (2008),
  no.~6, 3006--3031. \MR{2439500 (2009j:46080)}

\bibitem{Demk_Gop_Sch_ext_II_09}
\bysame, \emph{Polynomial extension operators. {P}art {II}}, SIAM J. Numer.
  Anal. \textbf{47} (2009), no.~5, 3293--3324. \MR{2551195 (2010h:46043)}

\bibitem{Demk_Gop_Sch_ext_III_12}
\bysame, \emph{Polynomial extension operators. {P}art {III}}, Math. Comp.
  \textbf{81} (2012), no.~279, 1289--1326. \MR{2904580}

\bibitem{Dest_Met_expl_err_CFE_99}
P. Destuynder and B. M{\'e}tivet, \emph{Explicit error bounds in a conforming
  finite element method}, Math. Comp. \textbf{68} (1999), no.~228, 1379--1396.
  \MR{1648383 (99m:65211)}

\bibitem{ern_vohralik_2015a}
A. Ern and M. Vohral\'ik, \emph{Polynomial-degree-robust a posteriori estimates
  in a unified setting for conforming, nonconforming, discontinuous {G}alerkin,
  and mixed discretizations}, SIAM J. Numer. Anal. \textbf{53} (2015), no.~2,
  1058--1081.

\bibitem{ern_gudi_smears_vohralik_2022}
A. Ern, T. Gudi, I. Smears, and M. Vohral{\'{\i}}k, \emph{Equivalence of local-
  and global-best approximations, a simple stable local commuting projector,
  and optimal $hp$ approximation estimates in {${\boldsymbol
  H}(\mathrm{div})$}}, IMA J. Numer. Anal. \textbf{42} (2022), no.~2,
  1023--1049.

\bibitem{ern_guermond_2021a}
A. Ern and J.-L. Guermond, \emph{Finite {E}lements {I}. {A}pproximation and
  {I}nterpolation}, Texts in Applied Mathematics, vol.~72, Springer
  International Publishing, Springer Nature Switzerland AG, 2021. \MR{4242224}

\bibitem{Ern_Sme_Voh_heat_HO_Y_17}
A. Ern, I. Smears, and M. Vohral{\'{\i}}k, \emph{Guaranteed, locally space-time
  efficient, and polynomial-degree robust a posteriori error estimates for
  high-order discretizations of parabolic problems}, SIAM J. Numer. Anal.
  \textbf{55} (2017), no.~6, 2811--2834.

\bibitem{ern_vohralik_2020a}
A. Ern and M. Vohral\'{\i}k, \emph{Stable broken {$H^1$} and {${\boldsymbol
  H}(\mathrm{div})$} polynomial extensions for polynomial-degree-robust
  potential and flux reconstruction in three space dimensions}, Math. Comp.
  \textbf{89} (2020), no.~322, 551--594. \MR{4044442}

\bibitem{fernandes_gilardi_1997a}
P. Fernandes and G. Gilardi, \emph{Magnetostatic and electrostatic problems in
  inhomogeneous anisotropic media with irregular boundary and mixed boundary
  conditions}, Math. Meth. Appl. Sci. \textbf{47} (1997), no.~4, 2872--2896.

\bibitem{girault_raviart_1986a}
V. Girault and P. Raviart, \emph{Finite element methods for {N}avier-{S}tokes
  equations: theory and algorithms}, Springer-Verlag, 1986.

\bibitem{hong_nagamochi_2008a}
S. Hong and H. Nagamochi, \emph{Convex drawings of graphs with non-convex
  boundary constraints}, Discret. Appl. Math. \textbf{156} (2008), no.~12,
  2368--2380.

\bibitem{Lad_Cham_CRE_16}
P. Ladev\`eze and L. Chamoin, \emph{The constitutive relation error method: a
  general verification tool}, Verifying calculations---forty years on,
  SpringerBriefs Appl. Sci. Technol., Springer, Cham, 2016, pp.~59--94.
  \MR{3409341}

\bibitem{luce_wohlmuth_2004a}
R. Luce and B. Wohlmuth, \emph{A local a posteriori error estimator based on
  equilibrated fluxes}, SIAM J. Numer. Anal. \textbf{42} (2004), no.~4,
  1394--1414.

\bibitem{nedelec_1980a}
J. N\'ed\'elec, \emph{Mixed finite elements in {$\mathbb R^3$}}, Numer. Math.
  \textbf{35} (1980), 315--341.

\bibitem{prager_synge_1947a}
W. Prager and J. Synge, \emph{Approximations in elasticity based on the concept
  of function space}, Quart. Appl. Math. \textbf{5} (1947), no.~3, 241--269.

\bibitem{raviart_thomas_1977a}
P. Raviart and J. Thomas, \emph{A mixed finite element method for 2nd order
  elliptic problems}, Mathematical Aspect of Finite Element Methods,
  Springer-Verlag, 1977.

\bibitem{Repin_book_08}
S. Repin, \emph{A posteriori estimates for partial differential equations},
  Radon Series on Computational and Applied Mathematics, vol.~4, Walter de
  Gruyter GmbH \& Co. KG, Berlin, 2008. \MR{2458008}

\bibitem{Sme_Voh_RD_20}
I. Smears and M. Vohral\'{\i}k, \emph{Simple and robust equilibrated flux a
  posteriori estimates for singularly perturbed reaction--diffusion problems},
  ESAIM Math. Model. Numer. Anal. \textbf{54} (2020), no.~6, 1951--1973.
  \MR{4160328}

\bibitem{Tant_Vees_Verf_loc_RD_15}
F. Tantardini, A. Veeser, and R. Verf\"{u}rth, \emph{Robust localization of the
  best error with finite elements in the reaction-diffusion norm}, Constr.
  Approx. \textbf{42} (2015), no.~2, 313--347. \MR{3392493}

\bibitem{tutte_1963a}
W. Tutte, \emph{How to draw a graph}, Proc. London Math. Soc. \textbf{13}
  (1963), no.~3, 743--768.

\bibitem{Veeser_approx_grads_16}
A. Veeser, \emph{Approximating gradients with continuous piecewise polynomial
  functions}, Found. Comput. Math. \textbf{16} (2016), no.~3, 723--750.
  \MR{3494508}

\end{thebibliography}

\appendix

\section{Mapping a two-dimensional triangular mesh into a reference triangle}
\label{section_triangular_representation}

In this appendix, we study two-dimensional triangular meshes that correspond to the planar
meshes $\lceil \CTo \rceil$ from Section~\ref{section_boundary_patches}. We will use some
basic notions from the graph theory and use Tutte's embedding theorem~\cite{tutte_1963a}.
We start be recalling the following basic notion:

\begin{definition}[Triconnected graph]\label{def_tricon}
A graph $(\LV,\LE)$, with the vertex set $\LV$ and the edge set $\LE$,
is triconnected if it has at least four vertices and if it remains
connected if any two vertices, together with its corresponding edges,
are removed from respectively $\LV$ and $\LE$. 
\end{definition}

\begin{lemma}[Triconnected graph]
\label{lemma_triconnected}
Consider a conforming planar triangular mesh $\CT$ with at least two elements $K$,
and assume that the set $\omega \subset \mathbb R^2$ covered
by the elements of $\CT$ is connected with a connected
Lipschitz boundary. Let us denote by $\CV$ and $\CE$ the sets of vertices
and edges of $\CT$ and by $\CV^{\rm ext}$ and $\CE^{\rm ext}$
the set of boundary vertices and edges of $\CT$. If all edges
$e = [\ba,\bb] \in \CE$ that have two vertices $\ba,\bb \in \CV^{\rm ext}$
are boundary edges (i.e. $e \in \CE^{\rm ext}$), then the
undirected graph $(\CV,\CE)$ is triconnected.
\end{lemma}

\begin{remark}[Non-triconnected graph]
A counterexample showing that, in general, a triangular mesh $\CT$
does not have a triconnected graph is presented in Figure
\ref{figure_counterexample}. There, the assumption that
all edges with two boundary vertices are boundary edges is violated.
\end{remark}

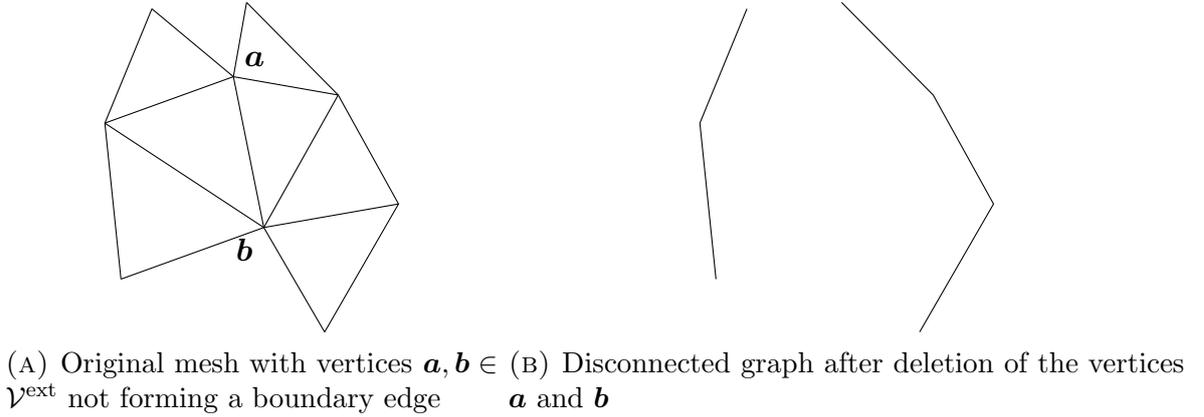
\begin{figure}
\begin{minipage}{.40\linewidth}
\centering
\begin{tikzpicture}[scale=2]

\coordinate (a) at (0.0,1.0);
\coordinate (b) at (0.2,0.0);

\coordinate (l1) at ($(a)+( 140:0.7)$);
\coordinate (l2) at ($(a)+( 200:0.9)$);
\coordinate (l3) at ($(b)+(-160:1.0)$);

\coordinate (r1) at ($(a)+( 80:0.5)$);
\coordinate (r2) at ($(a)+(-10:0.7)$);
\coordinate (r3) at ($(b)+( 10:0.9)$);
\coordinate (r4) at ($(b)+(-60:0.8)$);

\draw (a) -- (b);

\draw (a) -- (l1);
\draw (a) -- (l2);
\draw (b) -- (l2);
\draw (b) -- (l3);

\draw (l1) -- (l2) -- (l3);

\draw (a) -- (r1);
\draw (a) -- (r2);
\draw (b) -- (r2);
\draw (b) -- (r3);
\draw (b) -- (r4);

\draw (r1) -- (r2) -- (r3) -- (r4);

\draw (a) node[anchor=south west] {$\ba$};
\draw (b) node[anchor=north east] {$\bb$};
\end{tikzpicture}
\subcaption{Original mesh with vertices $\ba,\bb \in \CV^{\rm ext}$ not forming a boundary edge}
\end{minipage}
\begin{minipage}{.55\linewidth}
\centering
\begin{tikzpicture}[scale=2]

\coordinate (a) at (0.0,1.0);
\coordinate (b) at (0.2,0.0);

\coordinate (l1) at ($(a)+( 140:0.7)$);
\coordinate (l2) at ($(a)+( 200:0.9)$);
\coordinate (l3) at ($(b)+(-160:1.0)$);

\coordinate (r1) at ($(a)+( 80:0.5)$);
\coordinate (r2) at ($(a)+(-10:0.7)$);
\coordinate (r3) at ($(b)+( 10:0.9)$);
\coordinate (r4) at ($(b)+(-60:0.8)$);

\draw (l1) -- (l2) -- (l3);

\draw (r1) -- (r2) -- (r3) -- (r4);
\end{tikzpicture}
\subcaption{Disconnected graph after deletion of the vertices $\ba$ and $\bb$}
\end{minipage}
\caption{A non-triconnected graph}
\label{figure_counterexample}
\end{figure}

\begin{figure}
\begin{minipage}{.45\linewidth}
\begin{tikzpicture}[scale=2]
\coordinate (a) at (0,0);

\coordinate (a1) at ( 40:1.3);
\coordinate (a2) at (110:1.1);
\coordinate (a3) at (170:1.0);
\coordinate (a4) at (240:1.2);
\coordinate (a5) at (300:1.2);
\coordinate (a6) at (340:1.0);

\draw[white] (a1) circle (0.2);
\draw[white] (a2) circle (0.2);
\draw[white] (a6) circle (0.2);
\draw[white] (a3) circle (0.2);

\draw ($(a) + (70:0.3)$) node {$\ba$};

\draw ($(a3)-(0.3,0.0)$) node {$\bb'$};
\draw ($(a6)+(0.3,0.0)$) node {$\bc'$};

\draw (a) -- (a1);
\draw (a) -- (a2);
\draw (a) -- (a3);
\draw (a) -- (a4);
\draw (a) -- (a5);
\draw (a) -- (a6);

\draw (a1) -- (a2) -- (a3) -- (a4) -- (a5) -- (a6) -- cycle;

\draw (a3) circle (0.2);
\draw (a) circle (0.2);
\draw (a6) circle (0.2);

\draw[ultra thick] (a3) -- (a) -- (a6);

\draw[dashed] (a1) -- ($(a1)+(- 45:0.3)$);
\draw[dashed] (a1) -- ($(a1)+(  60:0.3)$);
\draw[dashed] (a1) -- ($(a1)+( 130:0.3)$);

\draw[dashed] (a2) -- ($(a2)+(  30:0.3)$);
\draw[dashed] (a2) -- ($(a2)+(  90:0.3)$);
\draw[dashed] (a2) -- ($(a2)+( 130:0.3)$);
\draw[dashed] (a2) -- ($(a2)+( 180:0.3)$);

\draw[dashed] (a3) -- ($(a3)+(  90:0.3)$);
\draw[dashed] (a3) -- ($(a3)+( 130:0.3)$);
\draw[ultra thick,dashed] (a3) -- ($(a3)+( 225:0.3)$);

\draw[dashed] (a4) -- ($(a4)+( 130:0.3)$);
\draw[dashed] (a4) -- ($(a4)+( 200:0.3)$);
\draw[dashed] (a4) -- ($(a4)+( 260:0.3)$);
\draw[dashed] (a4) -- ($(a4)+( 320:0.3)$);

\draw[dashed] (a5) -- ($(a5)+( 250:0.3)$);
\draw[dashed] (a5) -- ($(a5)+( 310:0.3)$);
\draw[dashed] (a5) -- ($(a5)+(   0:0.3)$);

\draw[dashed] (a6) -- ($(a6)+( -85:0.3)$);
\draw[ultra thick,dashed] (a6) -- ($(a6)+( -30:0.3)$);
\draw[dashed] (a6) -- ($(a6)+(  30:0.3)$);

\end{tikzpicture}
\subcaption{An interior patch with its vertex $\ba$}
\end{minipage}
\begin{minipage}{.45\linewidth}
\begin{tikzpicture}[scale=2]
\coordinate (a) at (0,0);

\coordinate (a1) at ( 40:1.3);
\coordinate (a2) at (110:1.1);
\coordinate (a3) at (170:1.0);
\coordinate (a4) at (240:1.2);
\coordinate (a5) at (300:1.2);
\coordinate (a6) at (340:1.0);

\draw (a3) circle (0.2);
\draw (a2) circle (0.2);
\draw (a1) circle (0.2);
\draw (a6) circle (0.2);

\draw ($(a3)-(0.3,0.0)$) node {$\bb'$};
\draw ($(a6)+(0.3,0.0)$) node {$\bc'$};

\draw (a1) -- (a2) -- (a3) -- (a4) -- (a5) -- (a6) -- cycle;

\draw (a3) circle (0.2);
\draw (a2) circle (0.2);
\draw (a1) circle (0.2);
\draw (a6) circle (0.2);

\draw[ultra thick] (a3) -- (a2) -- (a1) -- (a6);

\draw[dashed] (a1) -- ($(a1)+(- 45:0.3)$);
\draw[dashed] (a1) -- ($(a1)+(  60:0.3)$);
\draw[dashed] (a1) -- ($(a1)+( 130:0.3)$);

\draw[dashed] (a2) -- ($(a2)+(  30:0.3)$);
\draw[dashed] (a2) -- ($(a2)+(  90:0.3)$);
\draw[dashed] (a2) -- ($(a2)+( 130:0.3)$);
\draw[dashed] (a2) -- ($(a2)+( 180:0.3)$);

\draw[dashed] (a3) -- ($(a3)+(  90:0.3)$);
\draw[dashed] (a3) -- ($(a3)+( 130:0.3)$);
\draw[ultra thick,dashed] (a3) -- ($(a3)+( 225:0.3)$);

\draw[dashed] (a4) -- ($(a4)+( 130:0.3)$);
\draw[dashed] (a4) -- ($(a4)+( 200:0.3)$);
\draw[dashed] (a4) -- ($(a4)+( 260:0.3)$);
\draw[dashed] (a4) -- ($(a4)+( 320:0.3)$);

\draw[dashed] (a5) -- ($(a5)+( 250:0.3)$);
\draw[dashed] (a5) -- ($(a5)+( 310:0.3)$);
\draw[dashed] (a5) -- ($(a5)+(   0:0.3)$);

\draw[dashed] (a6) -- ($(a6)+( -85:0.3)$);
\draw[ultra thick,dashed] (a6) -- ($(a6)+( -30:0.3)$);
\draw[dashed] (a6) -- ($(a6)+(  30:0.3)$);

\end{tikzpicture}
\subcaption{An interior patch without its vertex $\ba$}
\end{minipage}

\begin{minipage}{.45\linewidth}
\begin{tikzpicture}[scale=2]
\coordinate (a) at (0,0);

\coordinate (a1) at (110:0.8);
\coordinate (a2) at (180:1.0);
\coordinate (a3) at (240:1.2);
\coordinate (a4) at (290:0.9);
\coordinate (a5) at (340:1.0);

\draw        (a)  circle (0.2);
\draw        (a1) circle (0.2);
\draw[white] (a2) circle (0.2);
\draw[white] (a3) circle (0.2);
\draw[white] (a4) circle (0.2);
\draw        (a5) circle (0.2);

\draw ($(a) + (70:0.3)$) node {$\ba$};

\draw ($(a1)+(0.3,0.0)$) node {$\bb'$};
\draw ($(a5)+(0.3,0.0)$) node {$\bc'$};

\draw (a) -- (a1);
\draw (a) -- (a2);
\draw (a) -- (a3);
\draw (a) -- (a4);
\draw (a) -- (a5);

\draw (a1) -- (a2) -- (a3) -- (a4) -- (a5);

\draw[ultra thick] (a1) -- (a) -- (a5);

\draw[dashed] (a1) -- ($(a1)+(  90:0.3)$);
\draw[ultra thick,dashed] (a1) -- ($(a1)+( 170:0.3)$);

\draw[dashed] (a2) -- ($(a2)+(  70:0.3)$);
\draw[dashed] (a2) -- ($(a2)+( 120:0.3)$);
\draw[dashed] (a2) -- ($(a2)+( 190:0.3)$);
\draw[dashed] (a2) -- ($(a2)+( 250:0.3)$);

\draw[dashed] (a3) -- ($(a3)+( 140:0.3)$);
\draw[dashed] (a3) -- ($(a3)+( 200:0.3)$);
\draw[dashed] (a3) -- ($(a3)+( 280:0.3)$);
\draw[dashed] (a3) -- ($(a3)+( 340:0.3)$);

\draw[dashed] (a4) -- ($(a4)+(-140:0.3)$);
\draw[dashed] (a4) -- ($(a4)+(- 50:0.3)$);
\draw[dashed] (a4) -- ($(a4)+(  10:0.3)$);

\draw[dashed] (a5) -- ($(a5)+( -90:0.3)$);
\draw[ultra thick,dashed] (a5) -- ($(a5)+(  10:0.3)$);
\end{tikzpicture}
\subcaption{A boundary patch with its vertex $\ba$}
\end{minipage}
\begin{minipage}{.45\linewidth}
\begin{tikzpicture}[scale=2]
\coordinate (a1) at (110:0.8);
\coordinate (a2) at (180:1.0);
\coordinate (a3) at (240:1.2);
\coordinate (a4) at (290:0.9);
\coordinate (a5) at (340:1.0);

\draw (a1) circle (0.2);
\draw (a2) circle (0.2);
\draw (a3) circle (0.2);
\draw (a4) circle (0.2);
\draw (a5) circle (0.2);

\draw[ultra thick] (a1) -- (a2) -- (a3) -- (a4) -- (a5);

\draw ($(a1)+(0.3,0.0)$) node {$\bb'$};
\draw ($(a5)+(0.3,0.0)$) node {$\bc'$};

\draw[dashed] (a1) -- ($(a1)+(  90:0.3)$);
\draw[ultra thick,dashed] (a1) -- ($(a1)+( 170:0.3)$);

\draw[dashed] (a2) -- ($(a2)+(  70:0.3)$);
\draw[dashed] (a2) -- ($(a2)+( 120:0.3)$);
\draw[dashed] (a2) -- ($(a2)+( 190:0.3)$);
\draw[dashed] (a2) -- ($(a2)+( 250:0.3)$);

\draw[dashed] (a3) -- ($(a3)+( 140:0.3)$);
\draw[dashed] (a3) -- ($(a3)+( 200:0.3)$);
\draw[dashed] (a3) -- ($(a3)+( 280:0.3)$);
\draw[dashed] (a3) -- ($(a3)+( 340:0.3)$);

\draw[dashed] (a4) -- ($(a4)+(-140:0.3)$);
\draw[dashed] (a4) -- ($(a4)+(- 50:0.3)$);
\draw[dashed] (a4) -- ($(a4)+(  10:0.3)$);

\draw[dashed] (a5) -- ($(a5)+( -90:0.3)$);
\draw[ultra thick,dashed] (a5) -- ($(a5)+(  10:0.3)$);
\end{tikzpicture}
\subcaption{A boundary patch without its vertex~$\ba$}
\end{minipage}

\caption
{
Vertex patches: Examples of path joining vertices before and after the vertex has been removed
}
\label{figure_vertex_patches}
\end{figure}
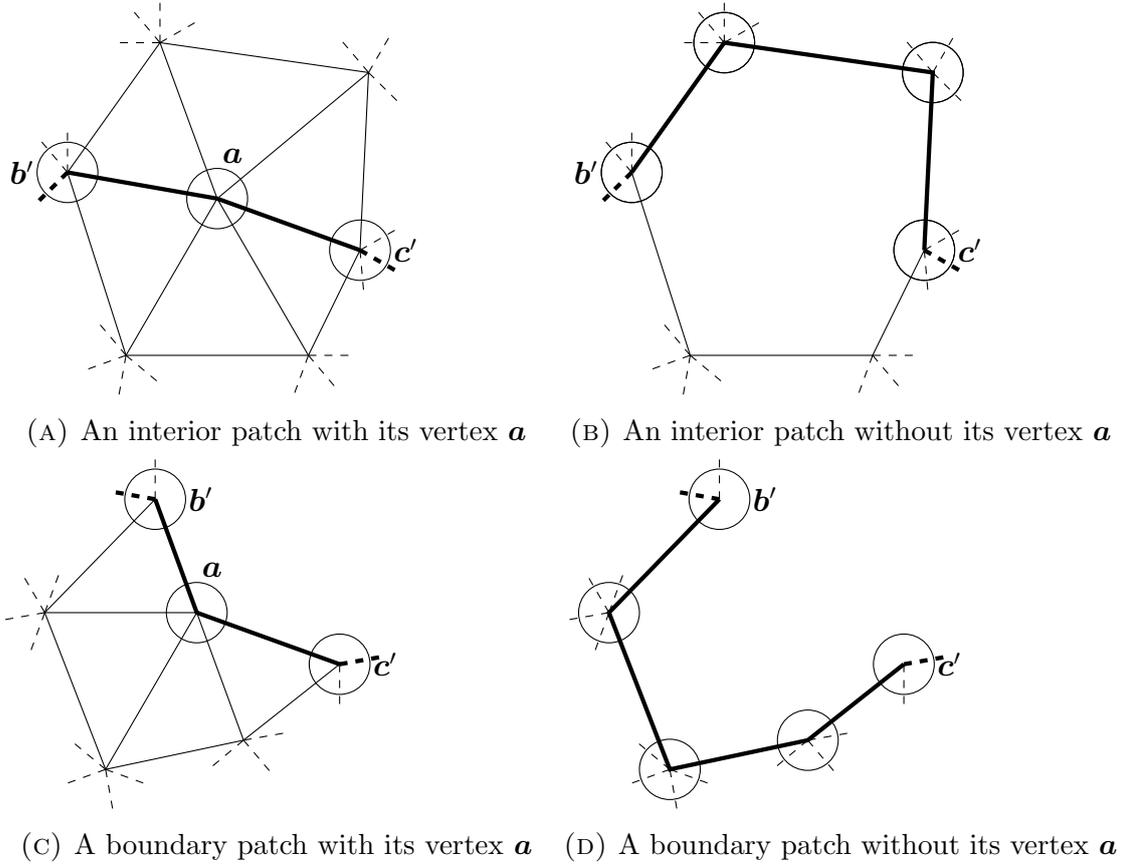

\begin{figure}
\begin{minipage}{.45\linewidth}
\begin{tikzpicture}[scale=2]
\coordinate (a) at (0,1);
\coordinate (b) at (0,0);

\coordinate (a1) at ($(a)+(-30:1.2)$);
\coordinate (a2) at ($(a)+(20:0.8)$);
\coordinate (a3) at ($(a)+(100:1.0)$);

\coordinate (b1) at ($(a)+(190:0.9)$);
\coordinate (b2) at ($(b)+(210:1.1)$);
\coordinate (b3) at ($(b)+(260:0.8)$);
\coordinate (b4) at ($(b)+(320:0.8)$);
\coordinate (b5) at ($(a)+(-30:1.2)$);

\draw[dashed] (a1) -- ($(a1)+( -80:0.3)$);
\draw[dashed] (a1) -- ($(a1)+(  10:0.3)$);
\draw[dashed] (a1) -- ($(a1)+(  80:0.3)$);

\draw[dashed] (a2) -- ($(a2)+(- 10:0.3)$);
\draw[dashed] (a2) -- ($(a2)+(  80:0.3)$);

\draw[ultra thick,dashed] (a3) -- ($(a3)+(  30:0.3)$);
\draw[dashed] (a3) -- ($(a3)+( 135:0.3)$);

\draw[dashed] (b1) -- ($(b1)+(  90:0.3)$);
\draw[dashed] (b1) -- ($(b1)+( 150:0.3)$);
\draw[dashed] (b1) -- ($(b1)+( 220:0.3)$);

\draw[dashed] (b2) -- ($(b2)+( 120:0.3)$);
\draw[dashed] (b2) -- ($(b2)+( 180:0.3)$);
\draw[dashed] (b2) -- ($(b2)+( 230:0.3)$);
\draw[dashed] (b2) -- ($(b2)+( 300:0.3)$);

\draw[dashed] (b3) -- ($(b3)+(-150:0.3)$);
\draw[dashed] (b3) -- ($(b3)+(- 90:0.3)$);
\draw[dashed] (b3) -- ($(b3)+(- 30:0.3)$);

\draw[dashed] (b4) -- ($(b4)+(-120:0.3)$);
\draw[ultra thick,dashed] (b4) -- ($(b4)+(- 60:0.3)$);
\draw[dashed] (b4) -- ($(b4)+(  30:0.3)$);

\draw ($(a)+(135:0.3)$) node {$\ba$};
\draw ($(b)+(180:0.3)$) node {$\bb$};

\draw ($(a3)+(180:0.3)$) node {$\ba'$};
\draw ($(b4)+(  0:0.3)$) node {$\bb'$};

\draw (b) -- (a);

\draw (a) -- (a1);
\draw (a) -- (a2);
\draw (a) -- (a3);
\draw (a) -- (b1);

\draw (b) -- (b1);
\draw (b) -- (b2);
\draw (b) -- (b3);
\draw (b) -- (b4);
\draw (b) -- (b5);

\draw (a1) -- (a2) -- (a3) -- (b1) -- (b2) -- (b3) -- (b4) -- (b5) -- cycle;

\draw (b) circle (0.2);
\draw (a) circle (0.2);
\draw (a3) circle (0.2);
\draw (b4) circle (0.2);

\draw[ultra thick] (a3) -- (a) -- (b) -- (b4);
\end{tikzpicture}
\end{minipage}
\begin{minipage}{.45\linewidth}
\begin{tikzpicture}[scale=2]
\coordinate (a) at (0,1);
\coordinate (b) at (0,0);

\coordinate (a1) at ($(a)+(-30:1.2)$);
\coordinate (a2) at ($(a)+(20:0.8)$);
\coordinate (a3) at ($(a)+(100:1.0)$);

\coordinate (b1) at ($(a)+(190:0.9)$);
\coordinate (b2) at ($(b)+(210:1.1)$);
\coordinate (b3) at ($(b)+(260:0.8)$);
\coordinate (b4) at ($(b)+(320:0.8)$);
\coordinate (b5) at ($(a)+(-30:1.2)$);

\draw[dashed] (a1) -- ($(a1)+( -80:0.3)$);
\draw[dashed] (a1) -- ($(a1)+(  10:0.3)$);
\draw[dashed] (a1) -- ($(a1)+(  80:0.3)$);

\draw[dashed] (a2) -- ($(a2)+(- 10:0.3)$);
\draw[dashed] (a2) -- ($(a2)+(  80:0.3)$);

\draw[ultra thick,dashed] (a3) -- ($(a3)+(  30:0.3)$);
\draw[dashed] (a3) -- ($(a3)+( 135:0.3)$);

\draw[dashed] (b1) -- ($(b1)+(  90:0.3)$);
\draw[dashed] (b1) -- ($(b1)+( 150:0.3)$);
\draw[dashed] (b1) -- ($(b1)+( 220:0.3)$);

\draw[dashed] (b2) -- ($(b2)+( 120:0.3)$);
\draw[dashed] (b2) -- ($(b2)+( 180:0.3)$);
\draw[dashed] (b2) -- ($(b2)+( 230:0.3)$);
\draw[dashed] (b2) -- ($(b2)+( 300:0.3)$);

\draw[dashed] (b3) -- ($(b3)+(-150:0.3)$);
\draw[dashed] (b3) -- ($(b3)+(- 90:0.3)$);
\draw[dashed] (b3) -- ($(b3)+(- 30:0.3)$);

\draw[dashed] (b4) -- ($(b4)+(-120:0.3)$);
\draw[ultra thick,dashed] (b4) -- ($(b4)+(- 60:0.3)$);
\draw[dashed] (b4) -- ($(b4)+(  30:0.3)$);

\draw ($(a3)+(180:0.3)$) node {$\ba'$};
\draw ($(b4)+(  0:0.3)$) node {$\bb'$};

\draw (a1) -- (a2) -- (a3) -- (b1) -- (b2) -- (b3) -- (b4) -- (b5) -- cycle;

\draw (b4) circle (0.2);
\draw (a1) circle (0.2);
\draw (a2) circle (0.2);
\draw (a3) circle (0.2);

\draw[ultra thick] (b4) -- (a1) -- (a2) -- (a3);
\end{tikzpicture}
\end{minipage}
\caption
{
Edge $[\ba,\bb]$ such that both $\ba$, $\bb$ lie in $\CV^{\rm int}$: example of path joining two vertices before and after the edge is removed.
}
\label{figure_interior_edge}
\end{figure}
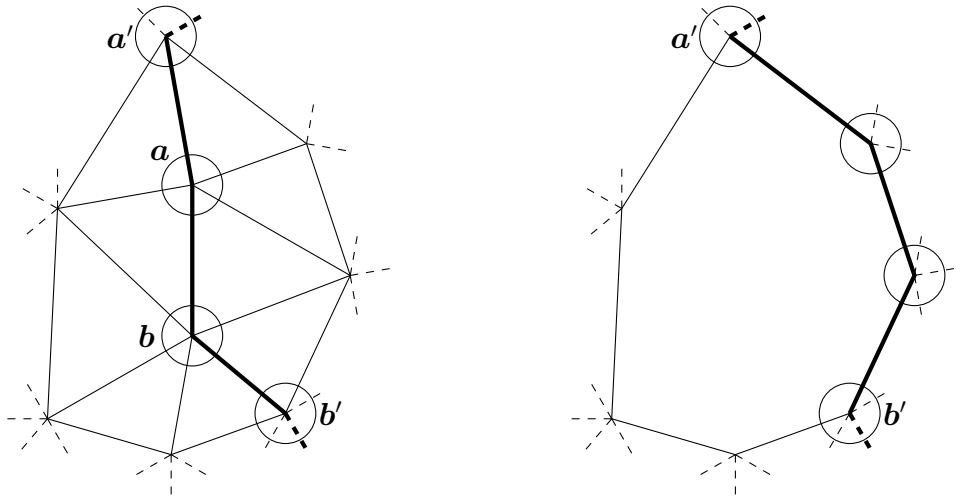

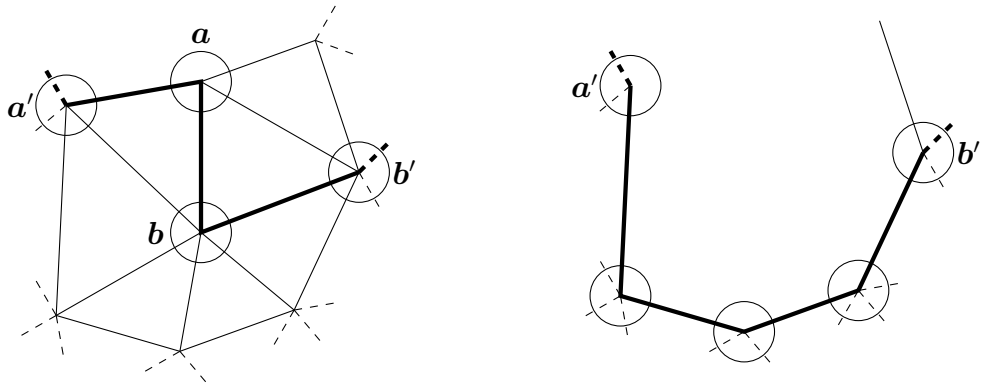
\begin{figure}
\begin{minipage}{.45\linewidth}
\begin{tikzpicture}[scale=2]
\coordinate (b) at (0,0);
\coordinate (a) at (0,1);

\coordinate (a1) at ($(a)+( 190:0.9)$);
\coordinate (a2) at ($(a)+(- 30:1.2)$);
\coordinate (a3) at ($(a)+(  20:0.8)$);

\coordinate (b1) at ($(b)+( 210:1.1)$);
\coordinate (b2) at ($(b)+( 260:0.8)$);
\coordinate (b3) at ($(b)+( 320:0.8)$);

\draw[white] (a1) circle (0.2);
\draw[white] (a2) circle (0.2);
\draw[white] (b1) circle (0.2);
\draw[white] (b2) circle (0.2);
\draw[white] (b3) circle (0.2);

\draw[ultra thick,dashed] (a1) -- ($(a1)+( 120:0.3)$);
\draw[dashed] (a1) -- ($(a1)+( 220:0.3)$);

\draw[dashed] (b1) -- ($(b1)+( 120:0.3)$);
\draw[dashed] (b1) -- ($(b1)+( 210:0.3)$);
\draw[dashed] (b1) -- ($(b1)+( 280:0.3)$);

\draw[dashed] (b2) -- ($(b2)+( 210:0.3)$);
\draw[dashed] (b2) -- ($(b2)+( 310:0.3)$);

\draw[dashed] (b3) -- ($(b3)+( 240:0.3)$);
\draw[dashed] (b3) -- ($(b3)+( 310:0.3)$);
\draw[dashed] (b3) -- ($(b3)+(  10:0.3)$);

\draw[dashed] (a2) -- ($(a2)+( -60:0.3)$);
\draw[ultra thick,dashed] (a2) -- ($(a2)+(  45:0.3)$);

\draw[dashed] (a3) -- ($(a3)+( -20:0.3)$);
\draw[dashed] (a3) -- ($(a3)+(  60:0.3)$);

\draw ($(a)+(0.0,0.3)$) node {$\ba$};
\draw ($(b)-(0.3,0.0)$) node {$\bb$};

\draw ($(a)-(0.3,0.0)+( 190:0.9)$) node {$\ba'$};
\draw ($(a)+(- 30:1.2)+(0.3,0.0)$) node {$\bb'$};

\draw (b) -- (a);

\draw (a) -- (a1);
\draw (a) -- (a2);
\draw (a) -- (a3);

\draw (b) -- (a1);
\draw (b) -- (b1);
\draw (b) -- (b2);
\draw (b) -- (b3);
\draw (b) -- (a2);

\draw (a1) -- (b1) -- (b2) -- (b3) -- (a2) -- (a3);

\draw (a)  circle (0.2);
\draw (b)  circle (0.2);
\draw (a1) circle (0.2);
\draw (a2) circle (0.2);

\draw[ultra thick] (a1) -- (a) -- (b) -- (a2);
\end{tikzpicture}
\end{minipage}
\begin{minipage}{.45\linewidth}
\begin{tikzpicture}[scale=2]
\coordinate (b) at (0,0);
\coordinate (a) at (0,1);

\coordinate (a1) at ($(a)+( 190:0.9)$);
\coordinate (a2) at ($(a)+(- 30:1.2)$);
\coordinate (a3) at ($(a)+(  20:0.8)$);

\coordinate (b1) at ($(b)+( 210:1.1)$);
\coordinate (b2) at ($(b)+( 260:0.8)$);
\coordinate (b3) at ($(b)+( 320:0.8)$);

\draw[ultra thick,dashed] (a1) -- ($(a1)+( 120:0.3)$);
\draw[dashed] (a1) -- ($(a1)+( 220:0.3)$);

\draw[dashed] (b1) -- ($(b1)+( 120:0.3)$);
\draw[dashed] (b1) -- ($(b1)+( 210:0.3)$);
\draw[dashed] (b1) -- ($(b1)+( 280:0.3)$);

\draw[dashed] (b2) -- ($(b2)+( 210:0.3)$);
\draw[dashed] (b2) -- ($(b2)+( 310:0.3)$);

\draw[dashed] (b3) -- ($(b3)+( 240:0.3)$);
\draw[dashed] (b3) -- ($(b3)+( 310:0.3)$);
\draw[dashed] (b3) -- ($(b3)+(  10:0.3)$);

\draw[dashed] (a2) -- ($(a2)+( -60:0.3)$);
\draw[ultra thick,dashed] (a2) -- ($(a2)+(  45:0.3)$);

\draw (a1) circle (0.2);
\draw (a2) circle (0.2);
\draw (b1) circle (0.2);
\draw (b2) circle (0.2);
\draw (b3) circle (0.2);

\draw (a1) -- (b1) -- (b2) -- (b3) -- (a2) -- (a3);
\draw[ultra thick] (a1) -- (b1) -- (b2) -- (b3) -- (a2);

\draw ($(a1)-(0.3,0.0)$) node {$\ba'$};
\draw ($(a2)+(0.3,0.0)$) node {$\bb'$};
\end{tikzpicture}
\end{minipage}
\caption
{
Edge $[\ba,\bb]$ such that $\bb$ lies in $\CV^{\rm int}$ and $\ba$ lies in $\CV^{\rm ext}$: example of path joining two vertices before and after the edge is removed.
}
\label{figure_edge_one_vertex}
\end{figure}

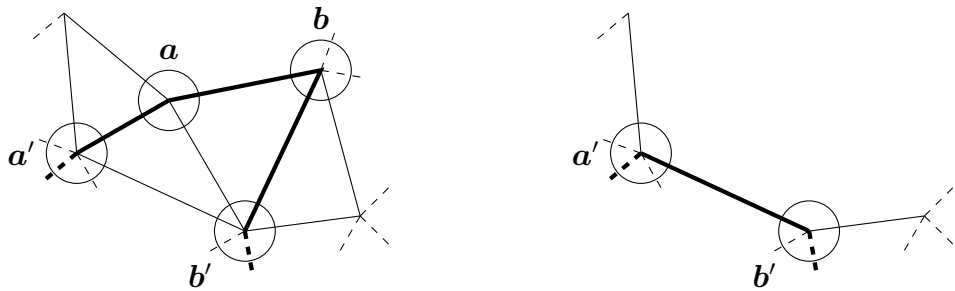
\begin{figure}
\begin{minipage}{.45\linewidth}
\begin{tikzpicture}[scale=2]

\draw[white] (-1.25,-0.5) rectangle (1.5,1.75);

\coordinate (a) at (0.0,1.0);
\coordinate (b) at (1.0,1.2);

\coordinate (a1) at ($(a) + (140:0.9)$);
\coordinate (a2) at ($(a) + (210:0.7)$);
\coordinate (a3) at ($(a) + (300:1.0)$);

\coordinate (b1) at ($(b) + (-75:1.0)$);

\draw (a) -- (b);

\draw (a) -- (a1);
\draw (a) -- (a2);
\draw (a) -- (a3);

\draw (b) -- (a3);
\draw (b) -- (b1);

\draw (a1) -- (a2) -- (a3) -- (b1);

\draw (a)  circle (0.2);
\draw (b)  circle (0.2);
\draw (a2) circle (0.2);
\draw (a3) circle (0.2);

\draw ($(a)+(0.0,0.2)$) node[anchor=south] {$\ba$};
\draw ($(b)+(0.0,0.2)$) node[anchor=south] {$\bb$};

\draw ($(a2)-(0.2,0.0)$) node[anchor=east] {$\ba'$};
\draw ($(a3)+(-135:0.2)$) node[anchor=north east] {$\bb'$};

\draw[ultra thick] (a2) -- (a) -- (b) -- (a3);

\draw[ultra thick,dashed] (a2) -- ($(a2)+(-140:0.3)$);
\draw[ultra thick,dashed] (a3) -- ($(a3)+(- 80:0.3)$);

\draw[dashed] (a1) -- ($(a1) + ( 220:0.3)$);
\draw[dashed] (a2) -- ($(a2) + ( 160:0.3)$);
\draw[dashed] (a2) -- ($(a2) + (- 60:0.3)$);
\draw[dashed] (a3) -- ($(a3) + (-150:0.3)$);
\draw[dashed] (b1) -- ($(b1) + (  45:0.3)$);
\draw[dashed] (b1) -- ($(b1) + (- 45:0.3)$);
\draw[dashed] (b1) -- ($(b1) + (-120:0.3)$);
\draw[dashed] (b)  -- ($(b)  + (  70:0.3)$);
\draw[dashed] (b)  -- ($(b)  + (- 10:0.3)$);

\end{tikzpicture}
\end{minipage}
\begin{minipage}{.45\linewidth}
\begin{tikzpicture}[scale=2]

\draw[white] (-1.25,-0.5) rectangle (1.5,1.75);

\coordinate (a) at (0.0,1.0);
\coordinate (b) at (1.0,1.2);

\coordinate (a1) at ($(a) + (140:0.9)$);
\coordinate (a2) at ($(a) + (210:0.7)$);
\coordinate (a3) at ($(a) + (300:1.0)$);

\coordinate (b1) at ($(b) + (-75:1.0)$);

%
%

\draw (a1) -- (a2) -- (a3) -- (b1);

\draw (a2) circle (0.2);
\draw (a3) circle (0.2);


\draw ($(a2)-(0.2,0.0)$) node[anchor=east] {$\ba'$};
\draw ($(a3)+(-135:0.2)$) node[anchor=north east] {$\bb'$};

\draw[ultra thick] (a2) -- (a3);

\draw[ultra thick,dashed] (a2) -- ($(a2)+(-140:0.3)$);
\draw[ultra thick,dashed] (a3) -- ($(a3)+(- 80:0.3)$);

\draw[dashed] (a1) -- ($(a1) + ( 220:0.3)$);
\draw[dashed] (a2) -- ($(a2) + ( 160:0.3)$);
\draw[dashed] (a2) -- ($(a2) + (- 60:0.3)$);
\draw[dashed] (a3) -- ($(a3) + (-150:0.3)$);
\draw[dashed] (b1) -- ($(b1) + (  45:0.3)$);
\draw[dashed] (b1) -- ($(b1) + (- 45:0.3)$);
\draw[dashed] (b1) -- ($(b1) + (-120:0.3)$);

\end{tikzpicture}
\end{minipage}
\caption
{
Edge $[\ba,\bb]$ such that both $\ba$, $\bb$ lie in $\CV^{\rm ext}$: example of path joining two vertices before and after the edge is removed.
}
\label{figure_exterior_edge}
\end{figure}

\begin{proof}
Notice first that the assumption that $\CT$ contains at least two elements ensures that
$\CV$ contains at least four vertices, so that we fit in the context of Definition~\ref{def_tricon}.
Considering a mesh $\CT$ satisfying the assumptions above, we need to show that if we remove
any pair of vertices $\ba,\bb \in \CV$, the associated graph $(\CV,\CE)$ remains connected.
It means that if $\bc,\bd \in \CV$ are two other points, we can join $\bc$ to $\bd$
via a sequence of edges without passing through $\ba$ or $\bb$.

For the sake of clarity, we use the additional notation
$\CV^{\rm int} \eq \CV \setminus \CV^{\rm ext}$ and
$\CE^{\rm int} \eq \CE \setminus \CE^{\rm ext}$ in the reminder of the proof.
Let us first consider the operation of removing a single
vertex $\ba \in \CV$. Then, since the set $\omega \subset \mathbb R^2$
covered by $\CT$ has a Lipschitz boundary, the patch of elements
$K \in \CT$ sharing the vertex $\ba$ either forms an open domain around
$\ba$ when $\ba \in \CV^{\rm int}$, or it forms an open domain with
$\ba$ on the boundary if $\ba \in \CV^{\rm ext}$. In both
cases, the boundary of the vertex patch is connected and consists of a
subset of $\CE$, and this property remains true after the vertex $\ba$
has been removed. This is illustrated in Figure~\ref{figure_vertex_patches}.
If there exists a path joining two vertices $\bb,\bc \in \CV \setminus \{\ba\}$
through $\ba$, then the path must go through two vertices $\bb',\bc' \in \CV$
on the boundary of the vertex patch surrounding $\ba$. Once $\ba$ is removed,
$\bb'$ and $\bc'$ can still be connected via remaining edges on the boundary
of the patch. This shows the graph of the mesh is biconnected: it remains
connected after we remove one vertex. Here, we need to show that the graph
is triconnected, meaning that it remains connected after two vertices have
been removed.

Let us first consider the case where the two vertices $\ba,\bb \in \CV$
removed from the graph do not share an edge, i.e. $[\ba,\bb] \not \in \CE$.
Then, we can first remove, say, $\ba$ and apply the reasoning above to
show that the graph remains connected. Because there is no edge connecting
$\ba$ and $\bb$, the vertex patch around $\bb$ remains untouched after the
deletion of $\ba$. As a result, the reasoning presented above still applies,
and the graph remains connected.

We therefore only need to consider cases where we remove two vertices
$\ba,\bb \in \CV$ such that $e = [\ba,\bb] \in \CE$. We then have to consider
two cases, either $e \in \CE^{\rm int}$ or not.
If $e \in \CE^{\rm int}$, due to our assumptions, then either one or two vertices are
interior, and in either case, the boundary of edge patch remains
connected after the edge is removed. This is depicted on Figures
\ref{figure_interior_edge} and~\ref{figure_edge_one_vertex}. In
both cases, if $\bc,\bd \in \CV \setminus \{\ba,\bb\}$ are two
vertices connected through $\ba$ or $\bb$, we can still connect
them through a path going around the boundary of the edge patch.
We finally consider the case where $e \in \CE^{\rm ext}$. In this
case too, the boundary of the edge patch is connected, and it remains
connected after the edge is removed. The process of modifying a path
going through $e = [\ba,\bb] \in \CE^{\rm ext}$ after it is removed is
shown in Figure~\ref{figure_exterior_edge}.
\end{proof}

\begin{proposition}[Mapping a two-dimensional triangular mesh into a reference triangle]
\label{proposition_triangular_representation}
Consider a triangular mesh $\CT$ covering a domain $\omega \subset \mathbb R^2$
and either composed of a single element $K$ or satisfying the assumptions of
Lemma~\ref{lemma_triconnected}. Then, there exists a bilipschitz mapping $\Psi$
from $\overline{\omega}$ to the reference triangle
$\widehat T \eq \{(y_1,y_2) \in [0,1]^2 \; | \; y_1+y_2 \leq 1\;\}$ such that $\Psi|_K$
is affine for each $K \in \CT$. In addition, if $\{\Gamma^\flat,\Gamma^\sharp\}$ is a partition
of $\partial \omega$ into connected components consisting of entire edges, then we can always
choose the mapping $\Psi$ so that $\Psi(\overline{\Gamma^\flat}) = \widehat E$ or
$\Psi(\overline{\Gamma^\sharp}) = \widehat E$, with
$\widehat E = \{ (y_1,y_2) \in [0,1]^2 \; | \; y_1+y_2 = 1 \}$.
\end{proposition}

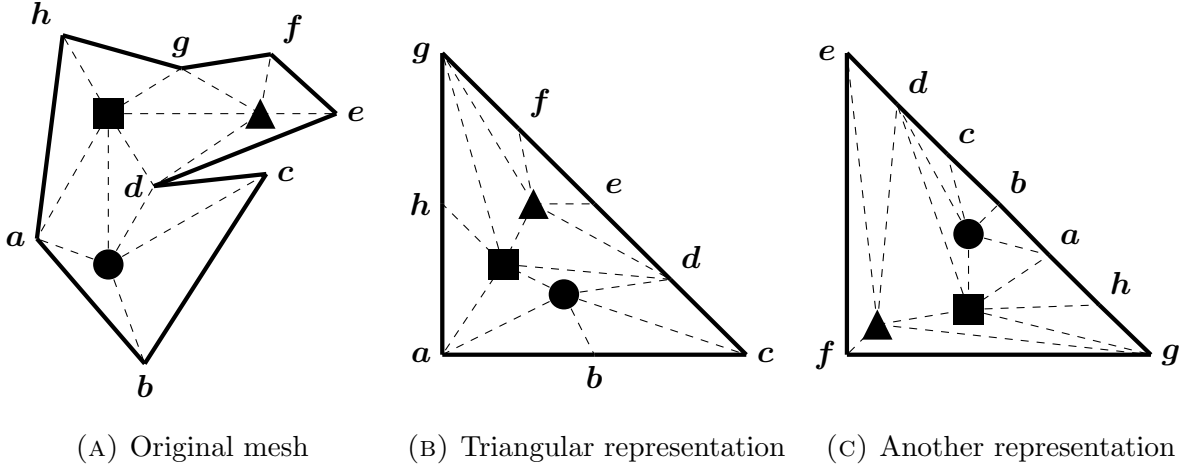
\begin{figure}
\begin{minipage}{.32\linewidth}
\begin{tikzpicture}[scale=2]

\draw[white] (-0.7,-1.0) rectangle (1.7,1.8);

\coordinate (a) at (0,0);
\coordinate (b) at (1,1);
\coordinate (c) at (0,1);

\coordinate (a1) at ($(a)+(-200:0.5)$);
\coordinate (a2) at ($(a)+(- 70:0.7)$);
\coordinate (a3) at ($(a)+(  30:1.2)$);
\coordinate (a4) at ($(a)+(  60:0.6)$);

\coordinate (b1) at ($(b)+(   0:0.5)$);
\coordinate (b2) at ($(b)+(  80:0.4)$);
\coordinate (b3) at ($(b)+( 150:0.6)$);

\coordinate (c1) at ($(c)+( 120:0.6)$);

\draw (a1) node[anchor=east] {$\ba$};
\draw (a2) node[anchor=north] {$\bb$};
\draw (a3) node[anchor=west] {$\bc$};
\draw (a4) node[anchor=east] {$\bd$};
\draw (b1) node[anchor=west] {$\be$};
\draw (b2) node[anchor=south west] {${\boldsymbol f}$};
\draw (b3) node[anchor=south] {$\bg$};
\draw (c1) node[anchor=south east] {$\bh$};

\draw[dashed] (a) -- (c);
\draw[dashed] (b) -- (c);

\draw[dashed] (a) -- (a1);
\draw[dashed] (a) -- (a2);
\draw[dashed] (a) -- (a3);
\draw[dashed] (a) -- (a4);

\draw[dashed] (b) -- (a4);
\draw[dashed] (b) -- (b1);
\draw[dashed] (b) -- (b2);
\draw[dashed] (b) -- (b3);

\draw[dashed] (c) -- (a1);
\draw[dashed] (c) -- (a4);
\draw[dashed] (c) -- (b3);
\draw[dashed] (c) -- (c1);

\draw[ultra thick] (a1) -- (a2);
\draw[ultra thick] (a2) -- (a3);
\draw[ultra thick] (a3) -- (a4);
\draw[ultra thick] (a4) -- (b1);
\draw[ultra thick] (b1) -- (b2);
\draw[ultra thick] (b2) -- (b3);
\draw[ultra thick] (b3) -- (c1);
\draw[ultra thick] (c1) -- (a1);

\fill (a) circle (0.1);
\fill ($(c)-(0.1,0.1)$) rectangle ($(c)+(0.1,0.1)$);
\fill ($(b)+(-0.1,-0.1)$) -- ($(b)+(0.1,-0.1)$) -- ($(b)+(0.0,0.1)$) -- cycle;

\end{tikzpicture}
\subcaption{Original mesh}
\label{figure_triangular_representation_original}
\end{minipage}
\begin{minipage}{.32\linewidth}
\begin{tikzpicture}[scale=2]

\draw[white] (-0.2,-0.4) rectangle (2.2,2.4);

\coordinate (a) at (0.80,0.40);
\coordinate (b) at (0.60,1.00);
\coordinate (c) at (0.40,0.60);

\coordinate (a1) at (0.00,0.00);
\coordinate (a2) at (1.00,0.00);
\coordinate (a3) at (2.00,0.00);
\coordinate (a4) at (1.50,0.50);
\coordinate (b1) at (1.00,1.00);
\coordinate (b2) at (0.50,1.50);
\coordinate (b3) at (0.00,2.00);
\coordinate (c1) at (0.00,1.00);

\draw (a1) node[anchor=east] {$\ba$};
\draw (a2) node[anchor=north] {$\bb$};
\draw (a3) node[anchor=west] {$\bc$};
\draw (a4) node[anchor=south west] {$\bd$};
\draw (b1) node[anchor=south west] {$\be$};
\draw (b2) node[anchor=south west] {${\boldsymbol f}$};
\draw (b3) node[anchor=east] {$\bg$};
\draw (c1) node[anchor=east] {$\bh$};

\draw[dashed] (a) -- (c);
\draw[dashed] (b) -- (c);

\draw[dashed] (a) -- (a1);
\draw[dashed] (a) -- (a2);
\draw[dashed] (a) -- (a3);
\draw[dashed] (a) -- (a4);

\draw[dashed] (b) -- (a4);
\draw[dashed] (b) -- (b1);
\draw[dashed] (b) -- (b2);
\draw[dashed] (b) -- (b3);

\draw[dashed] (c) -- (a1);
\draw[dashed] (c) -- (a4);
\draw[dashed] (c) -- (b3);
\draw[dashed] (c) -- (c1);

\draw[ultra thick] (a1) -- (a2);
\draw[ultra thick] (a2) -- (a3);
\draw[ultra thick] (a3) -- (a4);
\draw[ultra thick] (a4) -- (b1);
\draw[ultra thick] (b1) -- (b2);
\draw[ultra thick] (b2) -- (b3);
\draw[ultra thick] (b3) -- (c1);
\draw[ultra thick] (c1) -- (a1);

\fill (a) circle (0.1);
\fill ($(c)-(0.1,0.1)$) rectangle ($(c)+(0.1,0.1)$);
\fill ($(b)+(-0.1,-0.1)$) -- ($(b)+(0.1,-0.1)$) -- ($(b)+(0.0,0.1)$) -- cycle;

\end{tikzpicture}
\subcaption{Triangular representation}
\label{figure_triangular_representation_mapped1}
\end{minipage}
\begin{minipage}{.32\linewidth}
\begin{tikzpicture}[scale=2]

\draw[white] (-0.2,-0.4) rectangle (2.2,2.4);

\coordinate (a) at (0.80,0.80);
\coordinate (b) at (0.20,0.20);
\coordinate (c) at (0.80,0.30);

\coordinate (a1) at (1.33,0.66);
\coordinate (a2) at (1.00,1.00);
\coordinate (a3) at (0.66,1.33);
\coordinate (a4) at (0.33,1.66);
\coordinate (b1) at (0.00,2.00);
\coordinate (b2) at (0.00,0.00);
\coordinate (b3) at (2.00,0.00);
\coordinate (c1) at (1.66,0.33);

\draw (a1) node[anchor=south west] {$\ba$};
\draw (a2) node[anchor=south west] {$\bb$};
\draw (a3) node[anchor=south west] {$\bc$};
\draw (a4) node[anchor=south west] {$\bd$};
\draw (b1) node[anchor=east] {$\be$};
\draw (b2) node[anchor=east] {${\boldsymbol f}$};
\draw (b3) node[anchor=west] {$\bg$};
\draw (c1) node[anchor=south west] {$\bh$};

\draw[dashed] (a) -- (c);
\draw[dashed] (b) -- (c);

\draw[dashed] (a) -- (a1);
\draw[dashed] (a) -- (a2);
\draw[dashed] (a) -- (a3);
\draw[dashed] (a) -- (a4);

\draw[dashed] (b) -- (a4);
\draw[dashed] (b) -- (b1);
\draw[dashed] (b) -- (b2);
\draw[dashed] (b) -- (b3);

\draw[dashed] (c) -- (a1);
\draw[dashed] (c) -- (a4);
\draw[dashed] (c) -- (b3);
\draw[dashed] (c) -- (c1);

\draw[ultra thick] (a1) -- (a2);
\draw[ultra thick] (a2) -- (a3);
\draw[ultra thick] (a3) -- (a4);
\draw[ultra thick] (a4) -- (b1);
\draw[ultra thick] (b1) -- (b2);
\draw[ultra thick] (b2) -- (b3);
\draw[ultra thick] (b3) -- (c1);
\draw[ultra thick] (c1) -- (a1);

\fill (a) circle (0.1);
\fill ($(c)-(0.1,0.1)$) rectangle ($(c)+(0.1,0.1)$);
\fill ($(b)+(-0.1,-0.1)$) -- ($(b)+(0.1,-0.1)$) -- ($(b)+(0.0,0.1)$) -- cycle;

\end{tikzpicture}
\subcaption{Another representation}
\label{figure_triangular_representation_mapped2}
\end{minipage}
\caption{Triangular representation of a mesh}
\end{figure}

\begin{proof}
We first note that if $\CT$ consists of a single element $K$, then it is clear
that the result is true by considering a simple affine map associating the relevant
vertices of $K$ to the ones of $\widehat T$. We therefore focus on the case where
$\CT$ has at least two elements hereafter.

Due to Lemma~\ref{lemma_triconnected}, we know that the graph $(\CV,\CE)$,
where $\CV$ and $\CE$ are the vertices and edges of $\CT$, is triconnected.
Then~\cite[(9.2)]{tutte_1963a}, Tutte's embedding theorem ensures that we
can place the boundary vertices $\CV^{\rm ext}$ so that they correspond
to the vertices of an arbitrary convex polygon $P$, and draw the graph
$(\CV,\CE)$ in the plane such that the outer face of the graph is $P$. In
fact, we can always do so for a large family of star-shaped polygons
\cite[Theorem 10]{hong_nagamochi_2008a}, and it is in particular possible to
place the boundary vertices on the boundary of the reference triangle $\widehat T$.
Since the original mesh covering $\omega$ and the drawing of the graph $(\CV,\CE)$
in $\widehat T$ are two drawings of the same graph, the mapping $\Psi$ that is piecewise
affine on $\CT$ and maps the vertices of $\CT$ to the coordinates of the drawing
in $\widehat T$ is uniquely defined and satisfies the first statement of the proposition.
This process is illustrated in Figure~\ref{figure_triangular_representation_mapped1}.

If we further partition the boundary of $\omega$, then either $\Gamma^\flat$ or $\Gamma^\sharp$
consists of at least two edges. To fix the ideas, let us assume that $\Gamma^\flat$ has at least
two edges. Then, we can place the vertices on the boundary of $\widehat T$ so that $\Gamma^\flat$
is mapped on the horizontal and vertical edges of $\widehat T$ (see Figure
\ref{figure_triangular_representation_mapped2}), and $\Gamma^\sharp$ is then
mapped onto the remaining edge. We proceed the other way around if $\Gamma^\flat$
consists of a single edge.
\end{proof}

\end{document}